\documentclass[a4paper, 11pt,reqno]{amsart}
\usepackage{amsfonts, amsthm, amssymb, amsmath, stackengine, scalerel}
\usepackage{subfig}
\usepackage{mathrsfs,array,tikz-cd}
\usepackage{eucal,fullpage,times,color,enumerate,accents, comment}
\usepackage[all]{xy}
\usepackage{url}
\usepackage{hyperref}
\usepackage{enumitem}
\usepackage[new]{old-arrows}
\usepackage{extpfeil}
\usepackage{verbatim}  
\usetikzlibrary{graphs,decorations.pathmorphing,decorations.markings}
\usepackage{stackengine}
\usepackage{mleftright}
\usepackage{float}
\usepackage{bbm} 
\usepackage{multicol} 

\input xy
\xyoption{all}

\newextarrow{\xbigtoto}{{20}{20}{20}{20}}
   {\bigRelbar\bigRelbar{\bigtwoarrowsleft\rightarrow\rightarrow}}

\newcommand{\calV}{{\mathcal{V}}}

\newcommand{\calA}{{\mathcal{A}}}

\newcommand{\calM}{{\mathcal{M}}}
\newcommand{\calD}{{\mathcal{D}}}
\newcommand{\calQ}{\mathcal{Q}}
\newcommand{\E}{\mathcal{E}}
\newcommand{\calF}{\mathcal{F}}

\newcommand{\calC}{\mathcal{C}}
\newcommand{\calS}{\mathcal{S}}

\newcommand{\calT}{\mathcal{T}}
\newcommand{\calW}{\mathcal{W}}

\newcommand{\Spec}{\mathrm{Spec}}

\newcommand{\Hom}{\mathrm{Hom}}

\newcommand{\Map}{\mathrm{Map}}
\newcommand{\End}{\mathrm{End}}
\newcommand{\Aut}{\mathrm{Aut}}

\newcommand{\Sym}{\mathrm{Sym}}

\newcommand{\Set}{\mathrm{Set}}

\newcommand{\ob}{\mathrm{ob}}

\newcommand{\rk}{\mathrm{rk}}

\newcommand{\D}{\mathcal{D}}

\newcommand{\im}{\mathrm{im}}
\newcommand{\id}{\mathrm{id}}

\newcommand{\coker}{\mathrm{coker}}

\newcommand{\red}{\mathrm{red}}

\newcommand{\cl}{\mathrm{cl}}

\newcommand{\Var}{\mathcal{V}\mathrm{ar}}

\newcommand{\opp}{\mathrm{op}}    

\renewcommand{\ker}{\mathrm{ker}}

\newcommand{\argu}{\text{(---)}}

\newcommand{\M}{\mathcal{M}}

\newcommand{\Z}{\mathbb{Z}}
\newcommand{\N}{\mathbb{N}}

\newcommand{\C}{\mathbb{C}}

\newcommand{\thT}{\mathbb{T}}

\newcommand{\lrangles}[1]{\langle#1\rangle}

\newcommand{\SC}{\calS\calC}
\newcommand{\Wtop}{W_{\mathrm{top}}}
\newcommand{\Wbot}{W_{\mathrm{bot}}}
\newcommand{\MSF}{{\overline{M}| F}}
\newcommand{\piMSF}{\pi_0(\MSF)}
\newcommand{\OSF}{{{O|\calS F}}}

\newcommand{\gm}[1]{{{|#1|}}}
\newcommand{\cofib}{\mathrm{cofib}}

\newcommand{\tx}{3\times 3}

\def\forkindep{\mathrel{\raise0.2ex\hbox{\ooalign{\hidewidth$\vert$\hidewidth\cr\raise-0.9ex\hbox{$\smile$}}}}}

\newcommand{\cmot}{\chi^{\mathrm{mot}}}
\newcommand{\sk}{\mathbbm{1}_k}

\makeatletter
\newcommand*\dotp{\mathpalette\dotp@{.5}}
\newcommand*\dotp@[2]{\mathbin{\vcenter{\hbox{\scalebox{#2}{$\m@th#1\bullet$}}}}}
\makeatother

\DeclareSymbolFont{matha}{OML}{txmi}{m}{it}
\DeclareMathSymbol{\varv}{\mathord}{matha}{118}

\newcommand{\s}{\tilde{s}}
\newcommand{\tk}{\,\text{,}\,}

\newcommand{\cosimp}[3]{\xymatrix@1{#1 \ar@<.4ex>[r] \ar@<-.4ex>[r] & {\ }#2 \ar@<0.8ex>[r] \ar[r] \ar@<-.8ex>[r] & {\ } #3 \ar@<1.2ex>[r] \ar@<.4ex>[r] \ar@<-.4ex>[r] \ar@<-1.2ex>[r] & \cdots }}
\newsavebox{\pullback}
\sbox\pullback{%
	\begin{tikzpicture}%
	\draw (0,0) -- (1ex,0ex);%
	\draw (1ex,0ex) -- (1ex,1ex);%
	\end{tikzpicture}}

\definecolor{quotemark}{gray}{0.7}
\makeatletter
\newlength\origparskip

\newcommand{\fquote}{%
	\@ifnextchar[{\fquote@i}{\fquote@i[]}
}

\def\fquote@i[#1]{%
	\@ifnextchar[{\fquote@ii{#1}}{\fquote@ii{#1}[]}
}%

\def\fquote@ii#1[#2]{%
	\def\pqm@tempa{#1}%
	\def\pqm@tempb{#2}%
	\noindent
	\list
	{}
	{\setlength{\leftmargin}{0.3\textwidth}%
		\setlength{\rightmargin}{0.1\textwidth}%
		\setlength{\origparskip}{\parskip}}%
	\item[]%
	\begin{picture}(0,0)%
	\put(-15,-8){\makebox(0,0){\scalebox{4}{%
				\textcolor{quotemark}{\textquotedblright}}}}%
	\end{picture}%
	\begingroup
	\itshape
	\ignorespaces}%

\def\endfquote{%
	\endgroup
	\par
	\raggedleft
	\ifx\pqm@tempa\empty
	\else
	{\bfseries --- \pqm@tempa\par}%
	\setlength{\parskip}{\origparskip}%
	\ifx\pqm@tempb\empty
	\else
	(\pqm@tempb)%
	\fi
	\fi
	\par
	\endlist}
\makeatother

\def\doubleunderline#1{\underline{\underline{#1}}}

\def\doubleoverline#1{\overline{\overline{#1}}}

\newcommand{\bigslant}[2]{{\raisebox{.2em}{$#1$}\left/\raisebox{-.2em}{$#2$}\right.}}

\newcommand{\equalizer}[2]{\xymatrix@1{#1 \ar@<.4ex>[r] \ar@<-0.4ex>[r] & {\ } #2}}

\newcommand{\adjunction}[4]{\xymatrix@1{#1{\ } \ar@<-0.3ex>[r]_{ {\scriptstyle #2}} & {\ } #3 \ar@<-0.3ex>[l]_{ {\scriptstyle #4}}}}

\makeatletter
\newcommand*{\relrelbarsep}{.386ex}
\newcommand*{\relrelbar}{%
	\mathrel{%
		\mathpalette\@relrelbar\relrelbarsep
	}%
}
\newcommand*{\@relrelbar}[2]{%
	\raise#2\hbox to 0pt{$\m@th#1\relbar$\hss}%
	\lower#2\hbox{$\m@th#1\relbar$}%
}
\providecommand*{\rightrightarrowsfill@}{%
	\arrowfill@\relrelbar\relrelbar\rightrightarrows
}
\providecommand*{\leftleftarrowsfill@}{%
	\arrowfill@\leftleftarrows\relrelbar\relrelbar
}
\providecommand*{\xrightrightarrows}[2][]{%
	\ext@arrow 0359\rightrightarrowsfill@{#1}{#2}%
}
\providecommand*{\xleftleftarrows}[2][]{%
	\ext@arrow 3095\leftleftarrowsfill@{#1}{#2}%
}
\makeatother

\newcommand{\dsquare}[4]{
	\begin{tikzcd}[ampersand replacement=\&]
	#1\ar[r, >->]\ar[d, {Circle[open]}->] \ar[dr, phantom, "\square"]\& #2 \ar[d, {Circle[open]}->]\\
	#3 \ar[r, >->] \& #4
	\end{tikzcd}}

\newcommand{\dsquaref}[8]{
	\begin{tikzcd}[ampersand replacement=\&]
	#1\ar[r, >->,"#5"]\ar[d, swap,{Circle[open]}->,"#6"] \ar[dr, phantom, "\square"]\& #2 \ar[d, {Circle[open]}->,"#8"]\\
	#3 \ar[r, >->,"#7"] \& #4
	\end{tikzcd}}

\newcommand{\squaref}[8]{
	\begin{tikzcd}[ampersand replacement=\&]
	#1\ar[r,"#5"]\ar[d, swap,"#6"] \& #2 \ar[d, "#8"]\\
	#3 \ar[r, "#7"] \& #4
	\end{tikzcd}}

\newcommand{\rtail}{\rightarrowtail}
\newcommand{\ltail}{\leftarrowtail}
\newcommand{\xrtail}[1]{\overset{#1}{\rtail}}
\newcommand{\xltail}[1]{\overset{#1}{\leftarrowtail}}
\newcommand{\otail}{\,\,\stackengine{0pt}{\hspace{.81ex}$\rightarrow$}{$\circ$}{O}{l}{F}{F}{L}} 
\newcommand{\xotail}[1]{\overset{\,\,\,\,\,#1}{\otail}}

\newcommand{\ones}[2]{\begin{pmatrix}
		#1 \\ #2
\end{pmatrix}}

\newcommand{\Cplus}{{\calC^{\oplus}}}
\newcommand{\Ghat}{{\widehat{G}\calC}}

\newcommand{\Ar}{\mathrm{Ar}}
\newcommand{\Ars}{\Ar_{\square}}
\newcommand{\Art}{\Ar_{\triangle}}

\begin{document}
\bibliographystyle{alpha}
\newtheorem{theorem}{Theorem}[section]
\newtheorem*{theorem*}{Theorem}
\newtheorem*{condition*}{Condition}
\newtheorem*{definition*}{Definition}
\newtheorem*{corollary*}{Corollary}
\newtheorem*{conjecture}{Conjecture}
\newtheorem{proposition}[theorem]{Proposition}
\newtheorem{lemma}[theorem]{Lemma}
\newtheorem{corollary}[theorem]{Corollary}
\newtheorem{claim}[theorem]{Claim}
\newtheorem{conclusion}[theorem]{Conclusion}
\newtheorem{hypothesis}[theorem]{Hypothesis}
\newtheorem{summarytheorem}[theorem]{Summary Theorem}
\newtheorem{summaryexample}[theorem]{Summary Example}

\theoremstyle{definition}
\newtheorem{definition}[theorem]{Definition}
\newtheorem{question}{Question}
\newtheorem{remark}[theorem]{Remark}
\newtheorem{observation}[theorem]{Observation}
\newtheorem{discussion}[theorem]{Discussion}
\newtheorem{guess}[theorem]{Guess}
\newtheorem{example}[theorem]{Example}
\newtheorem{condition}[theorem]{Condition}
\newtheorem{warning}[theorem]{Warning}
\newtheorem{notation}[theorem]{Notation}
\newtheorem{construction}[theorem]{Construction}
\newtheorem{problem}[theorem]{Problem}
\newtheorem{fact}[theorem]{Fact}
\newtheorem{thesis}[theorem]{Thesis}
\newtheorem{convention}[theorem]{Convention}
\newtheorem{summary}[theorem]{Summary}
\newtheorem{reminder}[theorem]{Reminder}

\newtheorem{naivedefinition}[theorem]{Naive Definition}
\newtheorem{maintheorem}{Theorem}
\renewcommand*{\themaintheorem}{\Alph{maintheorem}}
\newtheorem*{theorem:GilletGrayson}{Theorem A}
\newtheorem*{theorem:Generators}{Theorem B}
\newtheorem*{theorem:R1}{Theorem C}
\newtheorem*{theorem:R2}{Theorem D}

\title{$K_1(\Var)$ is presented by stratified birational equivalences}

\author{Ming Ng}
\begin{abstract} This paper provides a complete presentation of $K_1(\Var)$, the $K_1$ group of varieties, resolving and simplifying a problem left open in \cite{ZakhK1}. Our approach adapts Gillet-Grayson's $G$-Construction to define an un-delooped $K$-theory space of varieties. There are two levels on which one can read the present paper. On a technical level, we streamline and extend previous results on the $K$-theory of exact categories to a broader class of categories, including $\Var$. On a more conceptual level, our investigations bring into focus an interesting generalisation of automorphisms (``double exact squares'') which generate $K_1$. For varieties, this corresponds to what we call stratified birational equivalences, but the construction extends to a wide range of non-additive contexts (e.g. $o$-minimal structures, definable sets etc.). This raises a challenging question: what kind of information do these generalised automorphisms calibrate? 
\end{abstract}

 \thanks{Research partially supported by EPSRC Grant EP/V028812/1 and a FY2024 JSPS Postdoctoral Fellowship (Short-Term).}
\maketitle

Our understanding of $K$-theory is changing. Recent efforts to extend tools from classical algebraic $K$-theory to non-additive settings have led us to make decisions on what its essential features are. One perspective, influenced by Waldhausen's $S_\bullet$-construction \cite{Wald}, is that $K$-theory is a framework for analysing the finite assembly and disassembly of objects; non-additive applications of this insight can be found in Campbell's $\widetilde{S}_\bullet$-construction \cite{Campbell} as well as Zakharevich's Assemblers \cite{ZakhAss}. A related perspective emphasises the view that $K$-theory breaks an object into two types of pieces. This underpins Campbell-Zakharevich's framework of {\em CGW categories} \cite{CGW}, which formalises key similarities between exact categories and the category of varieties $\Var_k$.

In a different line of work: the study of $K_1$ of arbitrary exact categories properly began with Gillet-Grayson’s $G$-construction \cite{GG}, which provides an elementary description of its generators. This description was later refined by Sherman \cite{ShermanAbelian, ShermanExact} and Nenashev \cite{NenGen}, before culminating in Nenashev's characterisation of the complete set of relations for $K_1$ \cite{Nen0, Nen1}. 

The present paper unites the two lines of investigation by extending the $K_1$ results to a subclass of CGW categories known as {\em pCGW categories}. These include not only exact categories and varieties, but also finite sets, $o$-minimal structures, and definable sets. Our analysis leads to two alternative, complete presentations of $K_1$ applicable to a broad range of non-additive contexts. For varieties, our results show $K_1(\Var_k)$ is generated by what we we call {\em stratified} birational equivalences. Informally, these are isomorphisms between dense open subvarieties whose complements are also isomorphic; in the reducible case, the open isomorphism induces birational equivalences on the irreducible components that meet the open subsets (see Section~\ref{sec:stratified}). We conclude with a broad range of interesting test problems exploring the implications.

\subsection*{Background Overview} Let us develop the previous remark that $K$-theory is an abstract framework for breaking an object into two different types of pieces. Consider the following two definitions.
\begin{itemize}
	\item Let $R$ be a ring. We define $K_0(R)$ as

	$$K_0(R):=\bigslant{\left\{\stackanchor{free abelian group}{fin. gen. proj. $R$-modules}\right\}}{\stackanchor{$[M]=[M']$, if $M\cong M'$}{$[M]=[M']+[M'']$, if $M'\to M\to M''$}}$$
	where $M'\to M\to M''$ is a short exact sequence. 
	\item Let $\Var_k$ be the category of $k$-varieties, i.e. reduced separated schemes of finite type over field $k$. We define $K_0(\Var_k)$ as

	$$K_0(\Var_k):=\bigslant{\left\{\stackanchor{free abelian group}{$k$-varieties}\right\}}{\stackanchor{$[X]=[X']$, if $X\cong X'$}{$[X]=[U]+[X\setminus U]$, if $U\hookrightarrow X$ is a closed immersion}}$$
	
\end{itemize}
The analogy is clear. A short exact sequence $M'\to M\to M''$ decomposes the $R$-module $M$ into two distinct pieces, $M'$ and $M''$, with $M'\to M$ an admissible mono and $M\to M''$ as an admissible epi. Similarly, $K_0(\Var_k)$ decomposes a variety $X$ into $U$ and $X\setminus U$, with $U\hookrightarrow X$ a closed immersion and $X\setminus U\hookrightarrow X$ an open immersion. 
In both cases, an object is broken into two pieces and viewed as an abstract sum of the two components. But what are the essential features of this decomposition? 

Quillen \cite{QuilenExact} introduced the notion of an exact category as a pair $(\calC, \calS)$, where $\calC$ is an additive category, and $\calS$ is a family of sequences $M' \to M \to M''$ satisfying conditions analogous to short exact sequences in abelian categories. This includes the natural condition that admissible monos $M' \to M$ are kernels of admissible epis $M \to M''$, and admissible epis are cokernels of admissible monos. These conditions enabled Quillen to construct a $K$-theory spectrum $K\calC$, recovering $K_0(R)$ on $\pi_0$ when $\calC$ is the category of finitely-generated projective $R$-modules.

Campbell-Zakharevich \cite{CGW} re-examined Quillen's framework, and made a crucial observation: instead of requiring the two classes of morphisms to compose, it suffices to encode their interaction formally. This added flexibility allows us to extend Quillen's $K$-theory to non-additive settings like varieties, where sequences like $U \hookrightarrow X \hookleftarrow X \setminus U$ clearly do not compose. This is made precise by \cite[Thm 4.3]{CGW}, which adapts Quillen's argument to construct a $K$-theory spectrum $K\Var_k$ recovering $K_0(\Var_k)$ on $\pi_0$. 

Importantly, the Campbell-Zakharevich construction allows us to define the $K$-groups $K_n(\Var_k):=\pi_n (K\Var_k)$ for all $n$. It is natural to ask:

\begin{question}\label{qn:key} What kind of information do the higher $K$-groups of varieties encode?
\end{question}

This is a challenging question. In classical algebraic $K$-theory, the coarseness of $K_0$ as an invariant may be measured by the fact that $K_0(F)=\Z$ for all fields $F$, whereas $K_1(F)\cong F^\times$. 
Is there an analogous story in the setting of varieties? More explicitly, how might we measure the loss of information in $K_0(\Var_k)$? To what extent can this information be recovered in the higher $K$-groups? The following summary theorem gives a snapshot of the current landscape.

\begin{summarytheorem}\label{sumthm:KVar} Let $k$ be an algebraically closed field of characteristic 0, and equip $K_0(\Var_k)$ with a ring structure by defining $[X]\cdot [Y]:= [(X\times_k Y)_{\red}]$. Two $k$-varieties $X,Y$ are said to be \emph{piecewise isomorphic} if $X$ and $Y$ admit finite partitions 
	$$X_1,\dots, X_n \qquad\text{and}\qquad  Y_1,\dots, Y_n$$
into locally closed subvarieties such that $X_i\cong Y_i$ for all $n$. The following is known:
\begin{enumerate}[label=(\roman*)]
	\item Define $SK_0(\Var_k)$ as the freely generated {\em semiring} on $[X]$ subject to $[X]=[Z]+[X\setminus Z]$. \underline{Then} two $k$-varieties $X,Y$ are piecewise isomorphic iff $[X]=[Y]$ in $SK_0(\Var_k)$. 
	\item Let $X,Y$ $k$-varieties such that $\dim X\leq 1$. \underline{Then} $[X]=[Y]$ in $K_0(\Var_k)$ iff they are piecewise isomorphic.
	\item There exists $k$-varieties $X$ and $Y$ such that $[X]=[Y]$ in $K_0(\Var_k)$ and yet fail to be piecewise isomorphic.
	\item Let $X$ be a $k$-variety of any non-negative dimension containing only finitely many rational curves. \underline{Then} for any $k$-variety $Y$, $[X]=[Y]$ in $K_0(\Var_k)$ iff they are piecewise isomorphic.
\end{enumerate}
\end{summarytheorem}
\begin{proof} (i) appears to be folklore, and is recorded in \cite{Beke} as well as \cite[Corollary 1.4.9, Chapter 2]{CNSMotivicIntegration}. (ii) is \cite[Proposition 6]{LiuSeb}. For (iii), various constructions are now known but the most well-known example is due to Borisov \cite{Borisov}. (iv) is \cite[Theorem 5]{LiuSeb}.
\end{proof} 

Summary Theorem~\ref{sumthm:KVar} sharpens our understanding of what is at stake. Given our high-level characterisation of $K$-theory as an abstract framework for analysing the finite assembly and decompositions of objects, the following question was natural:

\begin{question}[{{\cite[Question 1.2]{LL}}}] Is it true that two $k$-varieties are piecewise isomorphic iff they agree in $K_0(\Var_k)$?
\end{question}

In the setting of characteristic 0, item (iii) of the Summary Theorem answers no, signalling a loss of information on the level of $K_0$. Item (i) tells us the information is lost precisely because $K_0(\Var_k)$ involves group completion -- akin to an Eilenberg Swindle. Item (ii) tells us that piecewise isomorphism and equivalence in $K_0(\Var_k)$ coincide so long as the varieties are of sufficiently low dimension. Put otherwise, the algebraic barriers to geometric information only occur at the higher dimensions. Item (iv) is subtler, and raises interesting questions about how taking piecewise isomorphisms of complex varieties relates to the ampleness of their canonical line bundles (cf. the algebraic hyperbolicity conjecture for surfaces). 

In light of this discussion, let us return to Question~\ref{qn:key}. Some promising initial progress has been made. Using the formalism of Assemblers, Zakharevich constructs an alternative $K$-theory spectrum of varieties, before leveraging its connection with Waldhausen categories to obtain a partial characterisation of $K_1(\Var_k)$ \cite[Theorem B]{ZakhK1}. Inspired by Borisov's work \cite{Borisov}, this was later developed in \cite{ZakhLefschetz} to illuminate a subtle geometric insight: the failure to extend birational automorphisms of varieties to piecewise isomorphisms is tightly connected to the Lefschetz motive $[\mathbb{A}^1]$ being a zero divisor in $K_0(\Var_k)$. In a different vein: \cite{CWZ} identifies non-trivial elements in $K_n(\Var_k)$ by lifting various motivic measures $K_0(\Var_k)\to K_0(\calC)$ to the level of spectra $K\Var_k\to K\calC$. 

\subsection*{Discussion of Main Results} Until recently, a full characterisation of any higher $K$-group of varieties was not known. In her original paper, Zakharevich \cite[Theorem B]{ZakhK1} identifies the generators of $K_1(\Var_k)$ and some key relations, but does not prove their completeness. Independently from us, an intriguing recent collaboration between algebraic topologists and experts in homological stability has uncovered a homological proof \cite[Prop. 4.1]{SrokaScissors} that Zakharevich's presentation is in fact complete. 

We take a different approach. Whereas \cite{SrokaScissors} utilises homological methods to analyse $K_1$, the present paper instead relies on techniques from simplicial homotopy theory. Further, whereas \cite{ZakhK1} relies on the connection between $\Var_k$ and Waldhausen categories, we instead focus on the (tighter) connection between $\Var_k$ and exact categories. This sets up the following theorem.

\begin{maintheorem}[Theorem~\ref{thm:Gconstruction}]\label{thm:GG} Let $\calC$ be a pCGW category, and $\calS\calC$ the simplicial set obtained by applying the $S_\bullet$-construction. Then, there exists a simplicial set $G\calC$ such that there is a homotopy equivalence
	$$|G\calC|\simeq \Omega |\calS\calC|.$$
	In particular, $\pi_n |G\calC|=K_n\calC$ for all $n$.
\end{maintheorem}

In broad strokes: Theorem~\ref{thm:GG}  extends Gillet-Grayson's $G$-construction on exact categories \cite{GG} to a wider class of categories including $\Var_k$. The beauty of the $G$-construction is that it translates a topological problem (i.e. characterising $\pi_1$ of a loop space) into a simplicial one, which is more combinatorial and easier to work with. To show that this gives us sufficient leverage to characterise $K_1$ will, of course, take the rest of the paper. Let us also remark that while one can prove Theorem~\ref{thm:GG} by adapting the original proof \cite{GG} to our setting, we provide a more streamlined argument (Theorem~\ref{thm:thmC'}) inspired by Grayson's framework of dominant functors \cite{GraysonThmC}.

The previous remarks underscore a more fundamental difference. Both \cite{ZakhK1} and \cite{SrokaScissors} are concerned with the $K$-theory of Assemblers, whereas our paper builds on \cite{CGW} to develop the $K$-theory of so-called {\em pCGW categories}.  Precise definitions will be given in due course; for now, it suffices to think of Assemblers and pCGW categories as two distinct yet equivalent ways of defining the $K$-theory spectrum of varieties.\footnote{For the curious reader: the weak equivalence of these spectra as spaces is \cite[Theorems 7.8 and 9.1]{CGW}.} This difference becomes apparent when comparing our respective presentations of $K_1(\Var_k)$. In our language: 

 \begin{maintheorem}[Proposition~\ref{prop:baseline}]\label{thm:K1} Let $\calC$ be a pCGW category. Then $K_1(\calC)$ is generated by {\em double exact squares}, i.e. by pairs of distinguished squares in $\calC$ with identical nodes
\begin{equation}\label{eq:Bl}
l:=\left(\, \dsquaref{O}{C}{A}{B}{ }{ }{f_1}{g_1} \quad,\quad  \dsquaref{O}{C}{A}{B}{ }{ }{f'_1}{g'_1} \,\right),
\end{equation}
modulo the following relations
\begin{itemize}
	\item[(B1)] $\left\langle\left(\dsquaref{O}{A}{O}{A}{}{}{}{1}\,,\,\dsquaref{O}{A}{O}{A}{}{}{}{1}\right)\right\rangle=0$;
	\item[(B2)]  $\left\langle\left(\dsquaref{O}{O}{A}{A}{}{}{1}{}\,,\,\dsquaref{O}{O}{A}{A}{}{}{1}{}\right)\right\rangle=0;$
	\item[(B3)] 
	Suppose $f_2\colon A\xrtail{f_0} B\xrtail{f_1} C$ and $f'_2\colon A\xrtail{f'_0} B\xrtail{f'_1} C$. Under technical conditions, the composition splits:
		$$\!\!\!\!\!\!\!\!\!\!\!\!\!\!\!\!\!\!\!\!\!\!\small{\left\langle\left(\dsquaref{O}{\frac{B}{A}}{A}{B}{}{}{f_0}{g_0} \,,\, \dsquaref{O}{\frac{B}{A}}{A}{B}{}{}{f'_0}{g'_0} \right)\right\rangle + \left\langle\left(\dsquaref{O}{\frac{C}{B}}{B}{C}{}{}{f_1}{g_1} \,,\, \dsquaref{O}{\frac{C}{B}}{B}{C}{}{}{f'_1}{g'_1} \right)\right\rangle} = 
	\small{	\left\langle \left(\dsquaref{O}{\frac{C}{A}}{A}{C}{}{}{f_2}{g_2} \,,\, \dsquaref{O}{\frac{C}{A}}{A}{C}{}{}{f'_2}{g'_2} \right)\right\rangle}.$$
\end{itemize}
 \end{maintheorem}
\[\]

Surprisingly, this presentation appears to be new, even in the context of exact categories. How does it compare with $K_1$ of an Assembler? The following informal discussion may be illuminating.

\subsubsection*{On Generators} Double exact squares describe the breaking of an object into two distinct pieces -- for instance, Equation~\eqref{eq:Bl} shows $B$ being broken into $A$ and $C$. Interestingly, these squares generalise the usual notion of an automorphism -- see Example~\ref{ex:Aut}. By contrast, \cite{ZakhK1} shows that $K_1$ of an Assembler is generated by {\em piecewise isomorphisms}, which break an object into $n$ many pieces simultaneously. This difference reflects a trade-off between simplicity vs. flexibility. Our Theorem~\ref{thm:K1} presents a simpler set of generators for $K_1(\Var_k)$ than \cite{ZakhK1}, which can be advantageous when e.g. constructing derived motivic measures, as done in \cite{CWZ}.\footnote{Technically, \cite{CWZ} views $\Var_k$ as a so-called {\em subtractive category} before applying the {\em $\widetilde{S}_\bullet$-construction} as defined in \cite{Campbell}, but this is equivalent to viewing $\Var_k$ as a CGW category and applying the $S_\bullet$-construction; see \cite[Example 7.4]{CGW}.} On the other hand, the generality of piecewise automorphisms makes the Assemblers formalism better suited for investigating e.g. scissors congruence of convex polytopes, as done in \cite{SrokaScissors}, where simultaneous decomposition is more natural.\footnote{{\em Details.} Define two polytopes $P$ and $Q$ to be scissors congruent if: (i) $P=\bigcup_{i=1}^m P_i$ and $Q=\bigcup_{i=1}^m Q_i$ such that $P_i\cong Q_i$, and (ii) $P_i\cap P_j=Q_i\cap Q_j=\emptyset$ for $i\neq j$. The key hypothesis here is convexity -- arbitrary pairwise unions $P_j \cup P_k$ may fail to form a convex polytope, so decomposition and reassembly is best done simultaneously. 
} 

\subsubsection*{On Relations} There is an interesting discrepancy between the $K_1$ relations of \cite{ZakhK1} and Theorem~\ref{thm:K1}. In Zakharevich's presentation (which we restate as Theorem~\ref{thm:ZakhB}), the composition of piecewise automorphisms always splits in $K_1$. More precisely: $$\left\langle A\xrightrightarrows[f_2]{f_1} B\right\rangle + \left\langle B\xrightrightarrows[g_2]{g_1} C \right\rangle = \left\langle A\xrightrightarrows[g_2f_2]{g_1f_1} C\right\rangle \qquad \text{in} \, K_1,$$
 where $f_i,g_i$ are piecewise automorphisms. Figure~\ref{fig:K1Split} gives an informal illustration.
	\begin{figure}[h!]
	\centering
	\includegraphics[scale=0.9]{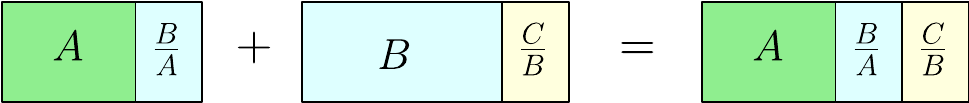}
	\caption{LHS: the piecewise automorphisms induced by closed immersions $A\xrtail{f} B$ and $B\xrtail{g} C$. RHS: the piecewise automorphism induced by their composition $A\xrtail{ gf} C$.} \label{fig:K1Split}
\end{figure}  	

By contrast, in our Theorem~\ref{thm:K1}, composition only splits under a technical condition. Proposition~\ref{prop:ass} clarifies this by showing
$$\lrangles{f} + \lrangles{g}=\lrangles{g\circ f} + \lrangles{l_2}  \qquad \text{in} \, K_1,$$
with $\lrangles{l_2}$ measuring the obstruction to splitting. Importantly, this obstruction term $\lrangles{l_2}$ is non-zero in general -- see 
Example~\ref{ex:KeyExample}. Our investigations lead to Theorem~\ref{thm:Nenashev}, the final main result of our paper. This gives an alternative presentation of $K_1$, extending previous work of Nenashev \cite{Nen1}. 

\begin{maintheorem}[Corollary~\ref{cor:Nen}]\label{thm:Nenashev} Let $\calC$ be a pCGW category. Then $K_1(\calC)$ is generated by double exact squares subject to the following relations:
\begin{enumerate}[label=(N\arabic*)]
	\item $\lrangles{l}=0$ if $l$ is a pair of identical squares, say
	\begin{equation*}
	l=\left(\dsquaref{O}{C}{A}{B}{}{}{f}{g} \quad ,\quad \dsquaref{O}{C}{A}{B}{}{}{f}{g}\right).
	\end{equation*}
	\item Given a so-called optimal $\tx$ diagram
	\begin{equation*} 
	\left(\begin{tikzcd}
	X_{00}\ar[r, >->, "f_{0}"] 
	\ar[d,>->,swap,"h_{0}"] & X_{01} \ar[d, >->,"h_{1}"] & X_{02}  
	\ar[d,>->, "h_{2}"] \ar[l, {Circle[open]}->,swap,"g_{0}"]\\
	X_{10}  \ar[r,>->,"f_{1}"]  &	X_{11}  \ar[dr,phantom,"\circlearrowleft"] & X_{12} \ar[l, {Circle[open]}->,swap,"g_1"] \\
	X_{20} \ar[u, {Circle[open]}->,"j_{0}"] \ar[r,>->,"f_2"]& X_{21} \ar[u, {Circle[open]}->,"j_1"]& X_{22}  \ar[u, {Circle[open]}->,swap,"j_2"] \ar[l, {Circle[open]}->,swap,"g_{2}"]
	\end{tikzcd} \qquad\text{,}\qquad  \begin{tikzcd}
	X_{00}\ar[r, >->, "f'_{0}"] 
	\ar[d,>->,swap,"h'_{0}"]  & X_{01} \ar[d, >->,"h'_{1}"] & X_{02}  
	\ar[d,>->, "h'_{2}"] \ar[l, {Circle[open]}->,swap,"g'_{0}"]\\
	X_{10}  \ar[r,>->,"f'_{1}"]  &	X_{11}  \ar[dr,phantom,"\circlearrowleft"] & X_{12} \ar[l, {Circle[open]}->,swap,"g'_1"] \\
	X_{20} \ar[u, {Circle[open]}->,"j'_{0}"] \ar[r,>->,"f'_2"]& X_{21} \ar[u, {Circle[open]}->,"j'_1"]& X_{22}  \ar[u, {Circle[open]}->,swap,"j'_2"] \ar[l, {Circle[open]}->,swap,"g'_{2}"]
	\end{tikzcd}\right)
	\end{equation*}
	defined by the following 6 double exact squares
	\begin{equation*}
l_{i}:= \left(	\dsquaref{O}{X_{i2}}{X_{i0}}{X_{i1}}{ }{ }{ f_{i}}{g_{i} } \, \,,\,\,	\dsquaref{O}{X_{i2}}{X_{i0}}{X_{i1}}{ }{ }{ f'_{i}}{g'_{i} }\right) \qquad	l^{i}:= \left(\dsquaref{O}{X_{2i}}{X_{0i}}{X_{1i}}{ }{}{h_{i} }{ j_{i}} \, \,,\,\, \dsquaref{O}{X_{2i}}{X_{0i}}{X_{1i}}{ }{}{h'_{i} }{ j'_{i}} \right)
	\end{equation*}
for all $i\in \{0,1,2\}$, the following 6-term relation holds
	\begin{equation*}
	\lrangles{l_0} + \lrangles{l_2} - \lrangles{l_1} = \lrangles{l^0} + \lrangles{l^2} - \lrangles{l^1} \quad .
	\end{equation*}
\end{enumerate}
\end{maintheorem}


This subtle discrepancy brings into focus key differences between the two $K$-theory frameworks, particularly in their choice of weak equivalences. Zakharevich models Assemblers via a special class of Waldhausen categories whose cofibration sequences all split up to weak equivalence (see \cite[Theorem 1.9]{ZakhK1}). This may seem a strong condition, but Zakharevich relies on a non-standard choice of weak equivalences that effectively hardcodes the splitting at the categorical level. To illustrate, consider the sequence $$\ast\hookrightarrow \mathbb{P}^1\hookleftarrow \mathbb{A}^1.$$
In the pCGW setting,  where weak equivalences are simply isomorphisms, this sequence clearly fails to split since $\mathbb{P}^1\not\cong\mathbb{A}^1\coprod \{\ast\}$. By contrast, in Zakharevich's framework, it splits up to weak equivalence since 
$$\mathbb{P}^1\dashleftarrow \{\mathbb{A}^1\}, \{\ast\} \xrtail{=} \{\mathbb{A}^1\}, \{\ast\}$$
defines a weak equivalence in the sense of \cite{ZakhK1}, Definition 1.7. We discuss the implications for the $K$-theory presentation more fully in Section~\ref{sec:NAKtheory}. 

Finally, we remark that our results extend to settings not modelled by the Assemblers framework. One obvious example is exact categories. More interestingly, there is suggestive evidence that they also apply to matroids (Example~\ref{ex:Matroid}), although some details remain to be worked out (see Section~\ref{sec:ProbMatroids}).

\subsubsection*{Implications for Characterising $K_n$}

Many of the results of the present paper are inspired by Nenashev's work \cite{NenGen} characterising $K_1(\calC)$ for an exact category $\calC$. Grayson \cite{GraysonBinary} later extended this to characterise $K_n(\calC)$ for all $n$, and we expect our approach to generalise similarly to the higher $K$-groups of varieties (or, more generally, the higher $K$-groups of pCGW categories). It is currently unclear how one might analogously extend the methods from \cite{ZakhK1} or \cite{SrokaScissors}.


\subsection*{Acknowledgements} This paper benefitted greatly from discussions with C. Eppolito, M. Kamsma, A. Nanavaty, A. Nenashev, M. Sarazola, S. Vasey, C. Weibel, C. Winges, T. Wittich and I. Zakharevich -- it is a pleasure to thank them all for their thoughtful insights as well as their generosity in sharing ideas. Special thanks are due to C. Malkiewich, B. Noohi and J. Pajwani. The author also thanks M. Malliaris for encouragement, and for reminding him that a cup is not always a doughnut.

\tableofcontents

\section{Preliminaries}
\subsection{CGW Categories} The key definition in \cite{CGW} is the {\em CGW category}. Informally, this is a category equipped with two subclasses of maps, $\M$ and $\E$ (analogous to admissible monos and epis in exact categories), along with a collection of square diagrams (``distinguished squares'') that encode how the $\M$ and $\E$-morphisms interact. 

This framework is presented using the language of double categories. Recall that a \emph{double category} $\calC$ is an internal category in $\mathrm{Cat}$. For our purposes, we will require the following refinement.

\begin{definition}\label{def:goodDC} A \emph{good double category} is a triple of categories $(\calC,\M,\E)$ presented by the data:
	\begin{itemize}
		\item \emph{Objects.} All three categories have the same objects: $\ob(\E)=\ob(\M)=\ob(\calC)$.
		\item \emph{Morphisms.} 
		\begin{itemize}
			\item[] \textbf{$\M$-morphisms:} $\M$ is a subcategory of $\calC$. Its morphisms are denoted $\rtail$.
			\item[] \textbf{$\E$-morphisms:} Either $\E$ or $\E^{\opp}$ is a subcategory of $\calC$. Its morphisms are denoted $\otail$. 
		\end{itemize}
		\item \emph{Distinguished Squares.} A collection of square diagrams, denoted
	\[	\begin{tikzcd} 
A \ar[r, >->,"f'"]\ar[d, swap, {Circle[open]}->,"g'"] \ar[dr, phantom, "\square"]& B \ar[d, {Circle[open]}->,"g"]\\
C \ar[r, >->,"f"] & D
\end{tikzcd}\]
	where $f,f'\in \M$ and $g,g'\in \E$. These squares are closed under horizontal and vertical composition. 
\smallskip
In addition, we require:
\begin{itemize}
	\item[\(\diamond\)] {\bf Commutativity.} Each distinguished square is a commutative diagram in $\calC$ (after adjusting for whether $\E$ or $\E^\opp$ is taken as a subcategory). 
	\item[\(\diamond\)] {\bf Closure under isomorphisms.} Any commutative square in $\calC$ with the indicated shape above is distinguished whenever either {\em both} $\M$-morphisms or {\em both} $\E$-morphisms are isomorphisms. 
\end{itemize}
	\end{itemize}
\end{definition}

\begin{convention}
	We denote a good double category by $\calC=(\M,\E)$, though we often abbreviate to $\calC$ when no confusion arises.  When $\calC$ is viewed as an ordinary 1-category, we call it the \emph{ambient category}, emphasising that it contains $\M$ and either $\E$ or $\E^{\opp}$ as subcategories.  
	In particular:
	\[
	A \xotail{g} B \;=\;
	\begin{cases}
	A \xrightarrow{g} B & \text{if } \E\subseteq\calC,\\
	B \xrightarrow{g} A & \text{if } \E^{\opp}\subseteq\calC.
	\end{cases}
	\]
\end{convention}

We now introduce a couple of helper definitions, before defining a CGW category. 

\begin{definition}\label{def:help} Let $C=(\M,\E)$ be a good double category, and $\D$ be any (ordinary) category.
\begin{enumerate}
\item Define $\Ars\E$ 
	\begin{itemize}
		\item Objects: Morphisms $A\otail B$ in $\E$.
		\item Morphisms:
		$\Hom_{\Ars\E}(A\xotail{ g} B, A'\xotail{g'} B')=
		\left\{\begin{tabular}{c}
		\hbox{distinguished} \\  \hbox{squares}
		\end{tabular}
		\hbox{$\dsquaref{A}{A'}{B}{B'}{}{g}{}{g'}$}\right\}.$	 
	\end{itemize}
$\Ars\M$ is defined analogously.
\item Define $\Art\D$ 
\begin{itemize}
	\item Objects: Morphisms $A\rightarrow B$ in $\D$.
	\item Morphisms: 
	$\Hom_{\Art\D}(A\xrightarrow{\,\,\,f} B, A'\xrightarrow{\,\,\,f'} B')=\left\{      \begin{tabular}{c}
	commutative \\ squares
	\end{tabular}\, \begin{tikzcd}
	A \ar[r,"\cong"] \ar[d,"f",swap] \ar[dr,phantom,"\circlearrowleft"] & A' \ar[d,"f'"]\\
	B \ar[r] & B'
	\end{tikzcd} \right\}.$	 
\end{itemize}
\end{enumerate}
	
\end{definition}

\begin{definition}[CGW Category]\label{def:CGWcat} A \emph{CGW category} $(\calC,\varphi,c,k)$ consists of the following data
	\begin{itemize}
		\item A good double category $\calC=(\M,\E)$;
		\item An isomorphism of categories $\varphi\colon \mathrm{iso}\M\to \mathrm{iso}\E$ which is identity on objects;

	\item An equivalence of categories 
		$$k\colon \Ars\E\longrightarrow\Art\M \qquad \qquad \qquad\text{and}\qquad\qquad\quad  c\colon \Ars\M\to \Art\E$$
		\[ \begin{tikzcd} A \ar[r,>->,"f"] \ar[dr,phantom,"\square"] \ar[d,{Circle[open]->},swap,"g"] & A' \ar[d,{Circle[open]->},"g'"] \\
		B \ar[r,>->,"f'",blue]  & B'
		\end{tikzcd} \longmapsto 
		\begin{tikzcd} \ker(g) \ar[d,>->,swap,"k(g)"] \ar[dr,phantom,yshift=0.4em,xshift=-0.2em,"\circlearrowleft"] \ar[r,>->,"\cong"]& \ker(g') \ar[d,>->,"k(g')"]\\
		B \ar[r,>->,blue,"f'"]& B'
		\end{tikzcd} \qquad \qquad\quad \begin{tikzcd} A \ar[r,>->,"f"] \ar[dr,phantom,"\square"] \ar[d,{Circle[open]->},swap,"g"] & A' \ar[d,{Circle[open]->},violet,"g'"] \\
		B \ar[r,>->,"f'"]  & B'
		\end{tikzcd} \longmapsto 
		\begin{tikzcd} A' \ar[d,{Circle[open]->},violet,"g'",swap] \ar[dr,phantom,yshift=0.1em,xshift=0.2em,"\circlearrowleft"] &\coker(f) \ar[d,{Circle[open]->},"\cong"]\ar[l,{Circle[open]->},swap,"c(f)"]\\
		B' & \coker(f') \ar[l,{Circle[open]->},swap,"c(f')"]
		\end{tikzcd}  \]
		\end{itemize}
		satisfying the axioms:
		\begin{itemize}
			\item[(Z)] \emph{Basepoint object.} $\calC$ contains an object $O$ initial in both $\E$ and $\M$.
			\item[(I)] \emph{Isomorphisms.} Let $\psi\colon A\to B$ be an isomorphism in ambient category $\calC$. Then: 
			\begin{itemize}
				\item $\psi$ belongs to $\mathrm{iso}\M$, which we denote suggestively as $\psi\colon A\rtail B$. 
				\item If $\E$ is a subcategory of $\calC$, then $\varphi(\psi)\colon A\otail B$ corresponds to $\psi\colon A\to B$ in $\calC$.
				\item If $\E^\opp$ is a subcategory of $\calC$, then $\varphi(\psi)\colon A\otail B$ corresponds to $\psi^{-1}\colon B\to A$ in $\calC$.
			\end{itemize}
			 \item[(M)] \emph{Monicity.} Every $\M$ and $\E$-morphism is monic. Furthermore:
			 \begin{itemize}
			 	\item  Every $\M$-morphism is monic in $\calC$. 
			 	\item If $\E\subseteq \calC$, then every $\E$-morphism is monic in $\calC$; otherwise if $\E^\opp\subseteq \calC$ then every $\E^\opp$-morphism is epi in $\calC$.
			 \end{itemize}
			 \item[(K)] \emph{Formal kernels and cokernels.} For any $f\colon A\rtail B$ in $\M$, there exists a \emph{formal cokernel}, denoted $c(f)\colon\coker (f)\otail B,$ 
		with a distinguished square as below left.
			 \begin{equation*}
			 \dsquaref{O}{\coker(f)}{A}{B}{ }{ }{f}{c(f)} \qquad \dsquaref{O}{A}{\ker(g)}{B}{ }{ }{k(g)}{g}
			 \end{equation*}
			 Dually, for any $g\colon A\otail B$ in $\E$, there exists a \emph{formal kernel}, denoted $k(g)\colon \ker(g)\rtail B$, with a distinguished square as above right. 
			 
Formal cokernels are unique up to (codomain-preserving) isomorphism: if $f'\colon C\otail B$ is another $\E$-morphism with distinguished square
			 \begin{equation*}
			 \dsquaref{O}{C}{A}{B}{ }{ }{f}{f'}, 
			 \end{equation*}
			 then there exists an isomorphism $\gamma\colon \coker(f)\rtail C$ such that the rightmost square in
			 \begin{equation*}
			 \begin{tikzcd}
		O \ar[dr,phantom,"\square"] \ar[d,{Circle[open]}-> ] \ar[r,>->]& \coker(f)\ar[dr,phantom,"\square"] \ar[d,{Circle[open]}->,"c(f)"] \ar[r,>->,"\gamma"]& C \ar[d,{Circle[open]}->,"f'"]\\
		A \ar[r,>->,swap,"f"]& B \ar[r,>->,swap,"1"] & B
			 \end{tikzcd}
			 \end{equation*}
is distinguished. We sometimes call $c(f)$ the {\em canonical quotient} of $f$. Formal kernels are unique in the analogous sense.
 
			 \end{itemize}
\end{definition}

\begin{remark}\label{rem:INV} The $\E$-morphisms of a distinguished square have isomorphic formal kernels and its $\M$-morphisms have isomorphic formal cokernels -- this follows from $c$ and $k$ mapping distinguished squares to the morphisms of $\Ar_{\triangle}$. Thus, by Axiom (K), $c$ and $k$ are inverse on objects (up to isomorphism).\footnote{\label{fn:INV} 
	Strictly speaking, $c$ and $k$ are only inverses up to codomain-preserving isomorphism, not inverses on the nose. 
In practice, however, eliding this distinction is usually harmless since distinguished squares commute and are closed under isomorphisms.} 
\end{remark}

There is a natural notion of structure-preserving functors and subcategories in the CGW context. A {\em CGW functor} of CGW categories is a double functor 
$$F\colon (\M,\E)\to (\M',\E')$$
that commutes with the functors $c$ and $k$, as follows:
\[\begin{tikzcd}
\Ars\M \ar[r,"c"] \ar[d,swap,"\Ars F"]& \Art \E \ar[d,"\Art F"] \\
\Ars\M' \ar[r,"c'"]& \Art \E' 
\end{tikzcd}\qquad \begin{tikzcd}
\Ars\E \ar[r,"k"] \ar[d,swap,"\Ars F"]& \Art \M \ar[d,"\Art F"] \\
\Ars\E' \ar[r,"k'"]& \Art \M' 
\end{tikzcd} \quad .\] 
For a CGW category $(\calC,\varphi,c,k)$, a {\em CGW subcategory} is a sub-double category $\calD\subseteq \calC$ such that the obvious restrictions $(\calD,\phi|_{\calD}, c|_{\calD}, k|_{\calD})$ forms a CGW category. 

\begin{convention} When context permits, we suppress $(\varphi,c,k)$ and simply write $\calC=(\M,\E)$ for a CGW category, or just $\calC$. 
\end{convention}

We now discuss Axioms (K) and (I) in more detail below, followed by some illustrative examples.

\subsubsection*{Quotients in CGW Categories} CGW categories are agnostic about whether formal cokernels arise from taking quotients in the {\em additive} setting (e.g. $R$-modules) or taking complements in the {\em non-additive setting} (e.g. varieties). Either way, the formal properties remain consistent. To reinforce this perspective, we adopt the following suggestive convention.

\begin{convention}[``Quotient''] We typically denote the formal cokernel of $f\colon A\rtail B$ as $\frac{B}{A}$, whenever the map $f$ is clear from context, which we refer to as a  \emph{quotient}.
\end{convention}

The following property about quotients, made possible by Helper Definition~\ref{def:help}, will play a key role in our analysis.

\begin{lemma}[Quotients respect Filtrations]\label{lem:quotFilt} Let  $$P_{0}\xrtail{f_1} P_{1}\xrtail{g_1} P_{2}$$
be a sequence of $\M$-morphisms in a CGW category. Choose formal quotients 
$$(f_2,g_2,h_2)$$
corresponding to the $\M$-morphisms $(f_1,g_1\circ f_1,g_1)$. Then, there exists a diagram of distinguished squares
\begin{equation}\label{eq:quot-filt}
\begin{tikzcd}
P_0 \ar[r, >->,"f_1"] \ar[dr,phantom,"\square"] & P_1 \ar[dr,phantom,"\square"] \ar[r, >->,"g_1"] & P_2 \\
O \ar[r,>->] \ar[u,{Circle[open]}->]&	P_{1/0} \ar[r, >->,"j_1"] \ar[u, {Circle[open]}->,"f_2"] \ar[dr,phantom,"\square"]  & P_{2/0}\ar[u, {Circle[open]}->,"g_2"] \\
& O \ar[r,>->] \ar[u,{Circle[open]}->]& P_{2/1} \ar[u, {Circle[open]}->,"j_2"] 
\end{tikzcd}	\qquad ,
\end{equation}
such that $h_2=g_2\circ j_2$. 
\end{lemma}
\begin{proof} This adapts \cite[Lemma 2.10]{CGW}. Notice $g_1$ induces a morphism
	$$(P_0\rtail P_1) \longrightarrow (P_0\rtail P_2) \in\Ar_{\triangle}\M\,.$$
Applying the functor $k^{-1}\colon \Ar_{\triangle}\M\to \Ar_{\square}\E$ constructs the top right square of Diagram~\eqref{eq:quot-filt}. In general, the quotients $k^{-1}(P_0\xrtail{f_1} P_1)$ and $k^{-1}(P_0\xrtail{g_1\circ f_1}P_2)$ need not coincide with our chosen $(f_2,g_2)$. Nonetheless, by Axiom (K), all choices of quotients are canonically related by codomain-preserving isomorphisms. Hence, since distinguished squares commute and are closed under isomorphism, we may identify these choices (cf. Footnote~\ref{fn:INV}). Extending this argument, we may also choose $j_2\colon P_{2/1}\otail P_{2/0}$ such that $h_2=g_2\circ j_2$. 
\end{proof}

\subsubsection*{Isomorphisms in CGW Categories} In their original definition, CGW categories were not required to be good double categories. Our reason for invoking goodness is to make the interaction between distinguished squares and isomorphisms more precise. This not only simplifies the original Axiom (I) in \cite[Definition 2.5]{CGW}, but also provides the technical control needed for our simplicial arguments to work smoothly (e.g. Lemmas~\ref{lem:SherLoopAdd} and~\ref{lem:SherLoopSplit}).

Ultimately, this is no real loss in generality since our definition still covers all the major examples in the original paper \cite{CGW}. The underlying reason for this is the following observation.

\begin{observation}\label{obs:goodISOs} Let $\calC=(\M,\E)$ be any CGW category. Suppose 
\begin{equation}\label{eq:PBPODistSq}
\begin{tikzcd} 
A \ar[r, >->,"f'"]\ar[d, swap, {Circle[open]}->,"g'"] & B \ar[d, {Circle[open]}->,"g"]\\
C \ar[r, >->,"f"] & D
\end{tikzcd}
\end{equation}
defines a commutative diagram in the ambient category $\calC$. If $f$ and $f'$ are isomorphisms, then Diagram~\eqref{eq:PBPODistSq} defines a pullback and a pushout in $\calC$. The same holds if $g$ and $g'$ are isomorphisms. 
\end{observation}
We omit the routine proof; see e.g. \cite[Lemma 4.3]{5autManifolds}. Observation~\ref{obs:goodISOs} essentially restates the obvious: pullback and pushout squares are closed under isomorphisms. As we shall see, this will be useful for verifying goodness.

\subsubsection*{Examples} We review some standard examples of CGW categories, plus two new ones; some further details can be found in \cite[\S 4]{CGW} and \cite[\S 3]{SarazolaShapiro}.

\begin{example}[Exact Categories]\label{ex:Exact} For an exact category $\calC$, define a CGW category $\calC=(\M,\E)$ by setting
$$\M=\{\text{admissible monomorphisms}\}\qquad \E=\{\text{admissible epimorphisms}\}^\opp.$$	
The basepoint object is the zero object in $\calC$. Notice by taking the opposite category, the zero object becomes initial in $\E$, as required by Axiom (Z).
 The distinguished squares are the biCartesian squares (= both pushouts and pullbacks in the ambient category $\calC$). By Observation~\ref{obs:goodISOs}, $\calC=(\M,\E)$ is a good double category. The equivalences $k$ and $c$ map admissible epis to kernels and admissible monos to cokernels, respectively. 

\end{example}

\begin{example}[Finite Sets]\label{ex:FinSet} Given $\mathrm{FinSet}$, define a CGW category $\mathrm{FinSet}=(\M,\E)$ by setting 
	$$\M=\E=\{\text{injections}\}.$$
The basepoint object $O$ is the empty set, the distinguished squares are the pushout squares. By Observation~\ref{obs:goodISOs}, this defines a good double category. The equivalences $c$ and $k$ are given by taking any inclusion $A\hookrightarrow B$ to the inclusion $B\setminus A\hookrightarrow B$. 

\end{example}

\begin{example}[Varieties]\label{ex:Var} Given $\Var_k$, define a CGW category $\Var_k=(\M,\E)$ by setting
	$$\M=\{\text{closed immersions}\} \qquad \E=\{\text{open immersions}\}.$$
The basepoint object $O$ is the empty variety. The distinguished squares are the pullback squares
	\[\dsquaref{A}{B}{C}{D}{}{}{f}{g}\]
in which $\im f\cup \im g= D$; notice this implies goodness of $\calC$.\footnote{In particular, we get $\im f\cup \im g= D$ for free if either f or g are isomorphisms.} Axioms (I) and (M) follow from standard properties of closed and open immersions. For Axiom (K), let $c$ and $k$ take a morphism to the inclusion of its complement. In particular, their corresponding distinguished squares are unique due to the uniqueness of reduced scheme structure\footnote{The key hypothesis here is that varieties are {\em reduced}. For general schemes, enforcing Axiom (K) is obstructed by the fact that a closed subset may admit different structure sheaves -- see e.g. \cite[Example 3.2.6]{Hartshorne}. 
} on locally closed subsets \cite[Tag 01J3]{stacks-project}. Finally, to see why $c$ and $k$ define the required functors in Definition~\ref{def:CGWcat}, use the fact that distinguished squares compose. 

\end{example}

For those interested in model theory, the following example will be suggestive.

\begin{example}[Definable Sets] Fix $\Sigma$ to be a first-order language and $M$ a $\Sigma$-structure. The term {\em definable} will always mean with parameters from $M$. Following \cite{KrajicekScanlon}, denote $\mathrm{Def}(M)$ as the category with objects the definable sets of $M$ and its powers $M^n$, and morphisms the definable functions.\footnote{Alternatively: one can work syntactically and define definable sets as functors from the category of $\thT$-models to $\Set$, where $\thT$ is some fixed first-order theory. Zakharevich defined its $K$-theory via her Assemblers framework in \cite[Example 3.4]{ZakhAss}.} We upgrade this to a CGW category by setting
	$$\M=\E=\{\text{definable injections}\}.$$ 		
Axioms (I) and (M) are thus satisfied by definition. For Axiom (Z), set the basepoint object $O$ as $\emptyset$. Define the distinguished squares as the pullback squares
\[\dsquaref{A}{B}{C}{D}{}{}{f}{g}\]
where $\im f \cup \im g=D$. 
Finally, any definable injection $f\colon C\rtail D$ can be mapped to its formal cokernel 
\[\dsquaref{O}{D\setminus f(C) }{C}{D}{}{}{f}{c(f)}\]
where $c(f)$ is the obvious inclusion of the complement. The same holds for the formal kernels.
\end{example}

\begin{remark}\label{rem:logic} For the curious non-logician: the category of definable sets $\mathrm{Def}(M)$ can be viewed as an abstraction of $\Var_k$. Consider $\mathbb{C}$ as an algebraically closed field with characteristic 0. In which case, the objects of $\mathrm{Def}(\mathbb{C})$ are the constructible subsets of complex affine $n$-spaces \cite[Corollary 3.2.8]{DMarker}. In fact, there exists an isomorphism of Grothendieck rings $K_0(\Var_\C)\cong K_0(\mathrm{Def}(\C))$ \cite{Beke} -- this indicates how logic may pull the cut-and-paste combinatorics away from the underlying algebraic geometry.
\end{remark}

Finally, let us mention another generalisation of exact categories: {\em proto-exact categories}, introduced by Dyckerhoff-Kapranov \cite{HigherSegal}. A particularly challenging example comes from \cite{ProtoExactMatroids}, which showed that the category of pointed matroids form a proto-exact category, yielding a $K$-theory spectrum.

 Informally, a  {\em matroid} abstracts the notion of linear independence. It consists of a finite set $E$ and a collection of subsets (``{\em flats}'') that are closed under dependency -- akin to subspaces in a vector space.\footnote{For those interested in homological stability: compare this, perhaps, with the perspective from Tits buildings.} Matroids bridge combinatorics and geometry, and have surprisingly deep links to algebraic geometry and Hodge Theory \cite{Katz,BakerHodgeCombinatorics}. It is therefore very interesting that they also define a CGW category.

\begin{example}[Matroids]\label{ex:Matroid} Let $M=(E,\calF,\bullet_M)$ be a {\em pointed matroid}, where $E$ is a finite set, and $\calF\subseteq 2^{E}$ the set of flats of matroid $M$ and $\bullet_M$ the distinguished base-point. A {\em strong map} of pointed matroids $f\colon M\to N$ is a function $f\colon E_M\to E_N$ such that $f(\bullet_M)=\bullet_N$ and $f^{-1}A\in\calF(M)$ for all $A\in\calF(N)$. By \cite[Lemma 2.12]{ProtoExactMatroids}, pointed matroids and strong maps form a category $\mathrm{Mat}_\bullet$. 

Next, denote $\widetilde{E}:=E\setminus \{\bullet_M\}$. Given any $S\subseteq \widetilde{E}$, denote $M|S$ to be the {\em restriction of $M$ to $S$} and $M/S$ to be the {\em contraction of $M$ to $S$} (for details, see e.g. \cite{OxleyMatroids} or \cite[\S 2]{ProtoExactMatroids}.)
We upgrade $\mathrm{Mat}_\bullet$ to a CGW category $\mathrm{Mat}_\bullet=(\M,\E)$ by setting
	$$\M=\{\text{strong maps that can be factored $N\xrightarrow{\sim} M|S\hookrightarrow M$, for some $S\subseteq \widetilde{E}_M$}\}$$
		$$\E=\{\text{strong maps that can be factored $M\twoheadrightarrow M/S\xrightarrow{\sim} N$, for some $S\subseteq \widetilde{E}_M$}\}^{\opp}.$$
Notice $\E$ is the {\em opposite} category of contractions, analogous to Example~\ref{ex:Exact}. In fact, one can apply \cite[Lemma 5.3]{ProtoExactMatroids} to check that $\M$ and $\E$ are closed under isomorphisms and composition, satisfying Axiom (I). Finally, let the distinguished squares be the biCartesian squares in $\mathrm{Mat}_\bullet$. 

Axioms (M) and (Z) follow from \cite[Lemma 5.4]{ProtoExactMatroids}: a strong map $f$ is monic in $\mathrm{Mat}_\bullet$ iff $f$ is injective on the underlying set, and $f$ is epi iff $f$ is surjective. In particular, this implies all morphisms in $\M$ and $\E$ are monic, and so the matroid $O:=(\{\ast\},\ast)$ is initial in both. In addition, translating \cite[Props. 5.7 and 5.8]{ProtoExactMatroids} to our setting: any $P\xrtail{i'} Q\xotail{j'} N$ or 
$P\xotail{j} M\xrtail{i} N$ can be completed into a distinguished square
\[\dsquaref{P}{Q}{M}{N}{i'}{j}{i}{j'}\,.\]
Setting $P=O$, this gives the formal kernels and cokernels of Axiom (K), unique up to isomorphism by the biCartesian property. Finally, since biCartesian squares compose, this assignment extends to define the required functors $c$ and $k$.

\end{example}

\subsection{The $K$-Theory of pCGW categories}\label{sec:pCGW} The main result of \cite[\S 4]{CGW} is that Quillen's $Q$-Construction \cite{QuilenExact} can be applied to any CGW category to define a $K$-theory spectrum. However, this paper will focus on a particularly well-behaved class of CGW categories, which we call {\em pCGW categories.} 

Informally, pCGW categories are CGW categories $\calC=(\M,\E)$ where $\M$ is closed under a formal kind of pushout. This, of course, generalises the familiar fact that admissible monos are  closed under pushouts in exact categories but it is instructive to understand why a generalisation is needed. Consider Example~\ref{ex:FinSet} where $\M$ is the category of finite sets and injections. In which case,
$$A\leftarrow \emptyset \to A \qquad \text{where $A\neq\emptyset$}$$
does not have a pushout in $\M$ since the map $A\coprod A\to A$ is not monic. Nonetheless, this issue can be circumvented by weakening the universal pushout property, as follows. 

\begin{definition}[Restricted Pushout]\label{def:restPO} Let $\M$ be a category whose morphsims are all monic. Suppose $D$ is a span
	$$C\leftarrow A \rightarrow B.$$
Define $\M_D$ to be the category whose objects are pullback squares of the form
	\[\squaref{A}{B}{C}{X}{}{}{}{} \qquad, \]
and whose morphisms are natural transformations where all components are the identity {\em except} at $X$. A {\em class of optimal squares} for $D$ is a subcategory $\widetilde{\M}_D\subseteq \M_D$ satisfying:
\begin{enumerate}[label=(\roman*)]
	\item $\widetilde{\M}_D$ has all squares where $A\xrightarrow{\cong} C$ above is an isomorphism.
	\item $\widetilde{\M}_D$ has an initial object. We denote this $B\star_A C$, and call it the {\em restricted pushout} of $D$. 
\end{enumerate}
We say $\M$ \emph{contains all restricted pushouts} if for every span $D$ there exists a chosen optimal class $\widetilde{\M}_D$ with initial object $B\star_A C$.

\end{definition}

We now introduce the key definition of a pCGW category, before turning to examples. 

\begin{definition}[pCGW Category]\label{def:pCGW} Let $\calC=(\M,\E)$ be a CGW category. We call $\calC$ a \emph{pCGW category} if $\M$ contains all restricted pushouts, subject to the following axioms:
	\begin{itemize}
		\item[(A)] \emph{Formal Direct Sums.} Denote the restricted pushout of $C\leftarrowtail O\rtail B$ as $B\oplus C:=B\star_O C$, which we also call {\em formal direct sums}. Then, there exists a canonical pair of distinguished squares  
		\[\dsquaref{O}{B}{C}{B\oplus C}{}{}{p_C}{q_B}\quad \text{and}\quad \dsquaref{O}{C}{B}{B \oplus C}{}{}{p_B}{q_C},\]
		which we call \emph{direct sum squares}. 
		\item[(PQ)]\emph{Preserves Quotients.}  Given any $B\star_A C$, there exists a canonical isomorphism $\frac{B\star_A C}{C}\cong \frac{B}{A}$. This assembles into a commutative diagram in the ambient category $\calC$
		\[\begin{tikzcd}
		A \ar[r,>->,"f"] \ar[d,>->,"g"]& B \ar[d,>->,"g'"] \ar[dr,phantom,"\circlearrowleft"] & \frac{B}{A}\ar[l,swap,"c(f)",{Circle[open]}->] \ar[d,>->,"\cong"]\ \\
		C \ar[r,>->,"f'"]& B\star_AC & \frac{B\star_A C}{C} \ar[l,swap,"c(f')",{Circle[open]}->]
		\end{tikzcd}\]
Note the RHS square need not be distinguished; it is only required to commute in $\calC$.

		\item[(DS)] \emph{Compatibility with Distinguished Squares.} Given a diagram of distinguished squares \[\begin{tikzcd}
		C \ar[d, {Circle[open]}->] \ar[dr, phantom, "\square"]& A \ar[dr, phantom, "\square"]\ar[d, {Circle[open]}->]  \ar[l, >->] \ar[r, >->] & B \ar[d, {Circle[open]}->] \\
		C' & A' \ar[l, >->] \ar[r, >->] & B' 
		\end{tikzcd}\]
		there is an induced map $B\star_A C\otail B'\star_{A'}C'$ such that the two induced squares 
		\[\dsquare{B}{B\star_AC}{B'}{B'\star_{A'}C'}\qquad \dsquare{C}{B\star_AC}{C'}{B'\star_{A'}C'}\]
		are distinguished.
	\end{itemize}
\end{definition}

\begin{remark} For transparency, Definition~\ref{def:pCGW} requires only that $\M$ contains all restricted pushouts; no such requirement is placed on $\E$. 
\end{remark}

\begin{convention}\label{conv:restrPO-quotient} We adopt the following conventions for quotients and formal direct sums.
\begin{itemize}
	\item Since we generally work only up to isomorphism (cf. Remark~\ref{rem:INV}), we typically fix a representative of $\frac{B\star_A C}{C}$ and tacitly identify it with some choice of quotient of $f$ via the canonical isomorphism of Axiom~(PQ).
	
	\item Axiom~(A) specifies canonical direct sum squares. In light of Axiom~(PQ), we take
	$$
	\begin{tikzcd}
	O \ar[r,>->] \ar[d,>->] & B \ar[d,>->,"p_B"] \ar[dr,phantom,"\circlearrowleft"] 
	& \ar[l,{Circle[open]}->,swap,"="] B \ar[d,>->,"="] \\
	C \ar[r,>->,"p_C"] & B\oplus C & \ar[l,{Circle[open]}->,swap,"q_B"] B
	\end{tikzcd}\qquad.
	$$
	In particular, this implies $p_B=q_B$ if $\E\subseteq\calC$, and $q_B\circ p_B=1_B$ if $\E^\opp\subseteq\calC$. Hence, for suggestiveness, we sometimes write, e.g. $p_B = 1\oplus C$.
	
\end{itemize}	
\end{convention}

\begin{remark} Definition~\ref{def:restPO} of restricted pushouts is a synthesis of \cite[Def. 5.3]{CGW} and \cite[Def. 3.2]{SarazolaShapiro}. The reason behind requiring Condition (i) for optimality is to ensure that restricted pushouts behave functorially in the following sense: 
\end{remark}

\begin{fact}\label{facts:restrictedpushouts} Given a pCGW category, consider the diagram 
	$$C\leftarrowtail A\rtail B \rtail B'$$
	Then $B'\star_B \left(B\star_A C\right)\cong B'\star_A C$. More explicitly, the composite of restricted pushouts in Diagram~\eqref{eq:rpusCOMPOSE} is the restricted pushout of the outer span.		
	\begin{equation}\label{eq:rpusCOMPOSE}
	\begin{tikzcd}
	A \ar[r,>->] \ar[d,>->] & B \ar[d,>->] \ar[r,>->] & B' \ar[d,>->] \\
	C \ar[r,>->] & B\star_A C \ar[r,>->]& B'\star_B \left(B\star_A C\right)
	\end{tikzcd} 
	\end{equation}	
\end{fact}
\begin{proof} Translate the proof of \cite[Corollary A.2]{SarazolaShapiro}.
\end{proof}

We now revisit the examples from before.\footnote{To avoid later confusion: when we specify the class of optimal squares to be ``{\em all pullback squares} such that \dots ''  we are of course restricting to $\M$ (and not all pullback squares in the ambient CGW category).}

\begin{example}[Exact Categories] Let $\star$ be the usual pushout, and let the optimal squares be the class of pullback squares such that the induced map $B\cup_A C\to X$ is an admissible mono. Then $\star$ is well-defined since the pushout of any span of admissible monos is also a pullback. Axiom (PQ) follows from such pushouts preserving cokernels, Axiom (DS) from the pasting law for pushouts.
	For Axiom (A), pick the obvious biCartesian squares in the exact category:
	\[\begin{tikzcd}
	B \ar[r,>->,"p_B"] \ar[d] & B\oplus C \ar[d,->>,"q_C"] \\
	O \ar[r]&  C
	\end{tikzcd}\qquad\begin{tikzcd}
	C \ar[r,>->,"p_C"] \ar[d] & B\oplus C \ar[d,->>,"q_B"] \\
	O \ar[r]&  B
	\end{tikzcd}.\]
\end{example}

\begin{example}[Varieties]\label{ex:pCGWVar} Let $\star$ denote the pushout in the category of schemes (not just varieties), and let the optimal squares be all pullback squares. Leveraging the results in \cite{Schwede} and the Stacks Project \cite[Tag 0ECH]{stacks-project}, one can verify:
\begin{itemize}
	\item[$\diamond$] Pushouts along closed immersions exist in the category of schemes, and yield squares of closed immersions \cite[Corollary 3.9]{Schwede}. 
	\item[$\diamond$] Such pushouts are also pullbacks.

	\item[$\diamond$] If \( A \rtail B \), \( A \rtail C \) are closed immersions of varieties, the pushout scheme \( B \cup_A C \) is also a variety.
\end{itemize}
A key observation: many properties of the pushout — reducedness, separatedness, locally of finite type, and the pullback property — can be checked affine-locally. To show the pushout is of {\em finite} type (not just {\em locally finite} type), it suffices to check it is also quasi-compact -- this follows from gluing two quasi-compact schemes along a quasi-compact subscheme. Finally, since optimal squares are pullbacks (by hypothesis), this ensures the induced pushout morphism $B\cup_A C\to X$ is a closed immersion.  In sum: $\star$ is well-defined.

Axioms (PQ) and (DS) follow from the pushout property and examining the pushout construction in \cite[\S 2-3]{Schwede}.\footnote{It is easy to see the pushout property yields a map $v\colon B\star_A C\to B'\star_{A'} C'$; verifying that $v$ also defines an open immersion, however, is somewhat involved. For a more formal perspective on why Axiom (DS) holds, see \cite[Proposition A.4]{SarazolaShapiro}.} For Axiom (A), let the direct sum squares be the coproduct squares with standard coproduct injection maps.

\end{example}

\begin{example}[Definable Sets] Let $M$ be a $\Sigma$-structure that: (i) eliminates imaginaries; and (ii) contains at least two distinct elements. Then the (ambient) category $\mathrm{Def}(M)$ admits coproducts\footnote{This result is well-known for the weak syntactic category \cite{Harnik}, but the same argument works semantically. Fix definable sets $D_0\subseteq M^k$ and $D_1\subseteq M^l$. Let $P=M^k\times M^l\times M$, and define an equivalence relation $E\subseteq P^2$ by
		$$ (\overline{x}_0,\overline{y}_0,z_0)\sim (\overline{x}_1,\overline{y}_1,z_1)\iff (z_0=z_1\land \overline{x}_0=\overline{x}_1)\vee (z_0\neq z_1\land \overline{y}_0=\overline{y}_1).$$
		By elimination of imaginaries, $Q:=P/E$ is definable. In fact, it can be definably identified with $M^k\coprod M^l$, and so there exists definable injections $M^k,M^l\rtail Q$ whose disjoint images cover $Q$. Restricting to $D_0$ and $D_1$ yields the coproduct $D_0\coprod D_1\subseteq Q$.} and quotients of equivalence relations. Given a span of definable injections $B\xltail{f} A\xrtail{g} C$, define 
	$$B\star_AC:=\bigslant{ B\coprod C}{\sim}$$
	where $\sim$ is the obvious equivalence relation generated by $f(a)\sim g(a), a\in A$. A routine check shows that $\sim$ is expressible as a {\em definable} equivalence relation, with transitivity following from injectivity of $f$ and $g$. In other words, $B\star_A C$ defines a pushout in the ambient category $\mathrm{Def}(M)$. 
Now let the optimal squares be all pullback squares; another straightforward check shows $\star$ is well-defined. 

Axioms (PQ) and (DS) follow for similar reasons as in $\Var_k$. For Axiom (A), the (definable) coproduct square is distinguished since $\M=\E=\{$definable injections$\}$.

\end{example}

\begin{example}[Matroids]\label{ex:MatroidspCGW} There is a technical barrier. Suppose $M_0$ and $M_1$ are matroids with groundsets $E_0$ and $E_1$, and assume $M_0|T=M_1|T=N$ where $E_0\cap E_1=T$. An {\em amalgam} of $M_0$ and $M_1$ is a matroid $M$ on $E_0\cup E_1$ such that $M|E_0=M_0$ and $M|E_1=M_1$. Unfortunately, as noted in \cite[\S 12.4]{OxleyMatroids}, amalgams do not always exist.  In our setting, this means an arbitrary span $M_0\ltail N\rtail M_1$ may not be completable into a commutative square
	\[\begin{tikzcd}
	N \ar[r,>->] \ar[d,>->] & M_0 \ar[d,>->]\\
	M_1 \ar[r,>->] & M
	\end{tikzcd}\]
so long as we require that $M$ has groundset $E_0\cup E_1$. 
One solution is to relax this requirement, and allow $M$ to have any groundset. The question then becomes: given a span $M_0\ltail N\rtail M_1$, is it always possible to embed $M_0$ and $M_1$ into a larger matroid $M$? Can we regard this larger matroid as being initial in the sense of Definition~\ref{def:restPO}? Further discussion of this problem is deferred to Section~\ref{sec:ProbMatroids}. 


\end{example}

We now set up the $K$-theory of pCGW categories via Waldhausen's $S_\bullet$-construction.

\begin{construction}[$S_\bullet$-Construction]\label{cons:SDOT}  Let $\calC$ be a pCGW category. Define $S \calC$ to be the simplicial set with $n$-simplices $S_n\calC$ given by flag diagrams
\[\begin{tikzcd} 
	M_{00}\ar[r, >->] & M_{01} \ar[r, >->] & M_{02} \ar[r, >->] & \dots\ar[r, >->]  & M_{0n} \\
	&	M_{11} \ar[r, >->] \ar[u, {Circle[open]}->]& M_{12} \ar[r, >->]\ar[u, {Circle[open]}->]  & \dots \ar[r,>->] & M_{1n}  \ar[u, {Circle[open]}->]\\
	&& M_{22} \ar[r,>->]  \ar[u, {Circle[open]}->]  & \dots   \ar[r,>->] & M_{2n} \ar[u, {Circle[open]}->] \\
	&&& \vdots\ar[u, {Circle[open]}->]  & \vdots \ar[u, {Circle[open]}->] \\
	&&&& M_{nn} \ar[u, {Circle[open]}->] 
	\end{tikzcd}\]	
subject to the conditions
\begin{enumerate}[label=(\roman*)]
	\item $M_{ii}=O$ for all $i$
	\item Every subdiagram 
	\[\begin{tikzcd}
	M_{ji} \ar[r,>->] \ar[dr,phantom,"\square"]& M_{jl}
\\
M_{ki} \ar[r,>->] \ar[u, {Circle[open]}->] & M_{kl}\ar[u, {Circle[open]}->] 	\end{tikzcd}\]
for $j<k$ and $i<l$ is distinguished. 
\end{enumerate}	
We will often represent an $n$-simplex as a sequence of $\M$-morphisms
$$O=M_0\rtail M_1\rtail M_2\rtail \dots \rtail  M_n$$
together with choices of quotients 
$$M_{j/i}:=\frac{M_j}{M_i} \qquad {i<j}.$$	
Degeneracy maps are obtained by duplicating an $M_i$, face maps by forgetting an $M_i$, with the addendum that forgetting $M_0$ means factoring out by $M_1$.
\end{construction}

This sets up the following translation of Theorems 4.3 and 7.8 in \cite{CGW}.

\begin{theorem}[Presentation Theorem]\label{thm:PresK0} Let $\calC$ be a pCGW category and define its $K$-theory spectrum
	$$K\calC:=\Omega|\calS\calC|,$$
with associated $K$-groups $K_n(\calC):=\pi_n K\calC$. Then $K_0(\calC)$ is the free abelian group generated by objects of $\calC$ modulo the relation that for any distinguished square
$$\dsquare{A}{B}{D}{C},$$
we have $[D]+[B]=[A]+[C].$	
\end{theorem}

\begin{remark} In fact, Theorem~\ref{thm:PresK0} holds for any CGW category equipped with a formal direct sum satisfying Axiom (A), Definition~\ref{def:pCGW}. In the case of matroids, define 
$M_1\oplus M_2$ to be the matroid on ground set $E_1\coprod E_2$, with flats of the form $F_1\coprod F_2$ where $F_1\in \calF(M_1)$ and $F_2\in \calF(M_2)$. 
\end{remark}

\subsection{Simplicial Loops \& Fibers} Re-examining Presentation Theorem~\ref{thm:PresK0}, notice the loop space features in the definition of $K\calC$. Following \cite[\S 2]{GG}, we present a simplified model of this construction.

\begin{convention} A simplicial set is a contravariant functor $X\colon \Delta^\opp\to \Set$. We sometimes write $X[n]$ as shorthand for $X([n])$. If $A,B\in\Delta$, we write $AB$ to mean the disjoint union of $A$ followed by $B$, where elements of $A$ are below those of $B$. A 0-simplex is sometimes called a {\em vertex}, a 1-simplex an {\em edge}.
\end{convention}

To motivate, recall that the loop space $\Omega Z$ of a pointed topological space $Z$ is the space of based loops $\Map(S^1,Z)$. Now fix a simplicial set $X$ with basepoint $O\in X[0]$. A simplicial loop may look like two 1-simplices glued together at the ends
$$\begin{tikzcd}
O \ar[r,"f", bend left] \ar[r, "g", bend right,swap,]& x
\end{tikzcd}$$
while a homotopy of such loops may look like
\[\begin{tikzcd}
O \ar[r,bend left, "f_0"] \ar[r,bend right,swap, "g_0"] \ar[rr,bend left=80, "f_1"] \ar[rr,bend right=80, swap, "g_1"]& x \ar[r,"h"] & y
\end{tikzcd}\]
where we glue 2-simplices in the triangles $f_0hf_1^{-1}$ and $g_0hg_1^{-1}$ along a shared 1-simplex $h$. One can then extend this picture in a natural way to the higher homotopies as follows:

\begin{construction}[Simplicial Loops]\label{cons:LOOP} Given any simplicial set $X$, define 
	$$\Omega X(A):=\displaystyle\lim_{\leftarrow} \left( \begin{tikzcd}
	\{O\} \ar[r,hook] & X([0]) & X([0]A) \ar[d] \ar[l]\\
	& X([0]A) \ar[u] \ar[r] & X(A)
	\end{tikzcd}\right).$$
\end{construction}

Notice that Construction~\ref{cons:LOOP} works for any simplicial set $X$. This suggests the obvious definition:

\begin{definition}[$G$-Construction] Let $\calC$ be a pCGW category. The {\em $G$-Construction} on $\calC$ is defined by applying the Simplicial Loop Construction to $\calS\calC$,
	$$G\calC := \Omega S\calC.$$
\end{definition}

Although the $G$-construction is well-defined for ordinary CGW categories, we will rely on the restricted pushouts of $\M$-morphisms to show that $G\calC\simeq K\calC$ (and thus $K_n(\calC)\cong \pi_n|G\calC|$). Our approach mirrors what was done by Gillet-Grayson \cite{GG}, who established $G\calC\simeq K\calC$ for exact categories. A key notion in their argument is the so-called {\em right fiber}. 
 
\begin{definition}[Right Fiber]\label{def:fiber} Suppose $F\colon X\to Y$ is a map of simplicial sets, $A\in\Delta$ and $\rho\in Y(A)$. We define $\rho|F$ (``the right fiber over $\rho$'' ) by 
	
	$$(\rho|F)(B):=\displaystyle\lim_{\leftarrow}\left( \begin{tikzcd}
	&& X(B) \ar[d]\\
	& Y(AB) \ar[r] \ar[d] & Y(B)\\
	\{\rho\} \ar[r,hook] &  Y(A)
	\end{tikzcd}\right).$$
One may regard $\rho|F$ as the simplicial analogue of the homotopy fiber of $|F|$ over $\rho$. We write $\rho | Y$ for $\rho|1_Y$.  
\end{definition}

We can now restate our problem. To show that $G\calC\simeq K\calC$, we need to show that geometric realisation and the loop space constructions commute up to homotopy equivalence, i.e. 
$$|\Omega\calS\calC|\simeq \Omega|\calS\calC|.$$
The following key observation tells us when this happens.

	\begin{observation}[Key Observation]\label{obs:homotopycartesian} For any simplicial set $X$, consider the commutative square 
		\begin{equation}\label{eq:GhomCar}
		\begin{tikzcd}
		\Omega X \ar[r,"t"] \ar[d,"b"] & O|X \ar[d,"q"]\\
		O|X \ar[r,"q"] & X
		\end{tikzcd}.
		\end{equation}
where $b$ and $t$ are the projection maps 
and $q$ the obvious face map. \underline{Then}:
\begin{enumerate}[label=(\roman*)]
	\item $O|X$ is contractible.
	\item $O|q\simeq \Omega X$
	\item  $|\Omega X|\simeq \Omega |X|$ iff this square is homotopy Cartesian.	
\end{enumerate}
\end{observation} 
\begin{proof} (i) follows from \cite[Lemma 1.4]{GG}. 
	 (ii) is clear from unpacking definitions. 
	For (iii), first take the homotopy pullback $P$ of 
	$$|O|X|\to |X|\leftarrow |O|X|$$
	in the homotopy category of spaces. Taking the geometric realisation of Diagram~\eqref{eq:GhomCar}, this yields a map $|\Omega X|\to P$, which is a homotopy equivalence iff Diagram~\eqref{eq:GhomCar} defines a homotopy pullback. Finally, since $O|X$ is contractible, the homotopy pullback $P$ is equivalent to the homotopy pullback of 
	$$ \ast \to |X| \leftarrow \ast, $$
	which is the loop space $\Omega|X|$.
\end{proof}

Key Observation~\ref{obs:homotopycartesian} suggests the following proof strategy. By item (iii), in order to show $G\calC\simeq K\calC$ it suffices to verify that the square
	\begin{equation}
\begin{tikzcd}
\Omega \calS\calC \ar[r,"t"] \ar[d,"b"] & O|\calS\calC \ar[d,"q"]\\
O|\calS\calC \ar[r,"q"] & \calS\calC
\end{tikzcd}.
\end{equation}
is homotopy Cartesian. By item (ii), this is equivalent to showing that 
	\begin{equation}
\begin{tikzcd}
O|q  \ar[r,"t"] \ar[d,"b"] & O|\calS\calC \ar[d,"q"]\\
O|\calS\calC \ar[r,"q"] & \calS\calC
\end{tikzcd}.
\end{equation}
is homotopy Cartesian. 

To demonstrate this, we will rely on a simplicial translation of Quillen's Theorem B:

\begin{theorem}[\cite{GG}, Theorem B']\label{thm:B'} Suppose $F\colon X\to Y$ is a map of simplicial sets. Suppose for any $A\in \Delta,$ any $\rho\in Y(A)$, and any $f\colon A'\to A$, the induced map
	$$ \rho|F \to f^\ast \rho|F$$
is a homotopy equivalence. Then the square
$$\squaref{\rho|F}{X}{\rho|Y}{Y}{ }{ }{ }{}$$
is homotopy Cartesian.	
\end{theorem}

\begin{convention} Following \cite{GG,GGErratum}, we hereafter refer to Theorem~\ref{thm:B'} as Theorem B'.
\end{convention}

\section{A Technical Result on Right Fibers}\label{sec:RightFib}

\begin{convention}\label{conv:pcgw-C} Hereafter, unless stated otherwise: $\calC$ denotes a fixed pCGW category, and all definitions and constructions are made with respect to its pCGW structure.
\end{convention}

The goal of this section is to prove Theorem~\ref{thm:thmC'}, which essentially says: given a nice simplicial map $F\colon Y\to \calS\calC$ where $\calC$ is a pCGW category, the right fiber $O|F$ admits a nice description. 
The results here are technical, and are inspired by Grayson's framework of {\em dominant functors} \cite{GraysonThmC}. 

There are two main applications of Theorem~\ref{thm:thmC'}. First, the key result that $G\calC\simeq K\calC$ is obtained as a straightforward corollary (Theorem~\ref{thm:Gconstruction}). Second, it sets up the proof of Theorem~\ref{thm:Sherman}, which gives an initial characterisation of the generators of $K_1(\calC)$; the details will be deferred to Section~\ref{sec:generators}.

\subsection{$H$-Space Structure} We introduce the new notion of a {\em moral subset}, which is similar to a simplicial subset of $\SC$ except we allow for a shift in dimensions.

\begin{definition}[Moral Subset] Fix $k\in \N$. A {\em $k$-moral subset} of $\SC$ is a simplicial set $Y$ together with a map
	$$F\colon Y\longrightarrow \calS\calC$$	
such that:
\begin{enumerate}[label=(\roman*)]
	\item If $k=0$, then $F$ is the inclusion of a simplicial subset $Y\subseteq \SC$.
	\item If $k\geq 1$, fix a base-point $O_Y:=(O\rtail \dots \rtail O)\in \SC[k-1]$. Then, all $n$-simplices in $Y$ have the form
	$$(O_Y\rtail J_0\rtail J_1\rtail \dots \rtail J_n)\in \SC([k-1][n])$$
The map $F$ is defined by forgetting $O_Y$ and factoring the rest by $J_0$.
	\item $Y$ inherits its face and degeneracy maps from $\SC$ in the obvious way. In addition, we require $Y$ to be closed under direct sums of $n$-simplices.\footnote{By which we mean we can take direct sums of distinguished squares in the obvious manner, see Lemma~\ref{lem:DirectSum}.}
\end{enumerate}
\end{definition}


\begin{example} The quotient map $q\colon O|\SC\to \SC$ from Observation~\ref{obs:homotopycartesian} is a 1-moral subset. 
\end{example}

\begin{construction}[Addition Map]\label{cons:Hspace} Fix $A\in \Delta$, where $A=[a]$. Given a $k$-moral subset $F\colon Y \to \calS\calC$ and $\overline{M}\in \SC (A)$, represent a $q$-simplex $W$ in $\overline{M}|  F$ as
	\begin{equation}\label{eq:WMSF}
	W=\begin{pmatrix}
	\doubleunderline{\Wtop}\\ \Wbot
	\end{pmatrix}=\begin{pmatrix} &  && &(O_Y&\rtail) &  \doubleunderline{J_0\rtail J_1 \rtail \dots \rtail J_q}\\ O\rtail M_1&\rtail &\dots &\rtail &M_a &\rtail &  L_0\rtail L_1\rtail \dots \rtail L_q\end{pmatrix}
	\end{equation}	
where $\Wtop$ is a simplex of $Y$, $\Wbot$ a simplex of $\SC$ and the double line represents the identity
 \[\begin{tikzcd}
O=F\left(J_0\right) \ar[d,equal] \ar[r,>->]& \dots \ar[r,>->]& F\left(J_q\right) \ar[d,equal]  \\
O=\frac{L_0}{L_0} \ar[r,>->]& \dots \ar[r,>->] & \frac{L_q}{L_0}
\end{tikzcd}.\]
Parentheses are placed around $(O_Y\rtail)$ in Equation~\eqref{eq:WMSF} to indicate that when $Y$ is a $0$-moral subset, we delete $O_Y\rtail$ from $\Wtop$ and set $J_0=O$. 

Define addition $(W,W')\mapsto W+W'$ by setting
$$W+W':=\begin{pmatrix} &(O_Y  &\rtail) & J_0\oplus J'_0\rtail \dots \rtail J_q\oplus J'_q\\ O \rtail \dots \rtail& M_a &\rtail & \doubleoverline{L_0\star_{M_a}L'_0\rtail \dots \rtail L_q\star_{M_a}L'_q} \end{pmatrix},$$
with the following quotients:
\begin{itemize}
	\item For $J_{i}\oplus J'_i$:
	$$\frac{J_i\oplus J'_i}{J_j\oplus J'_j}:=\frac{J_i}{J_j}\oplus \frac{J'_i}{J'_j}.$$
	\item For $L_i\star_{M_a}L'_i:$
	$$\frac{L_i\star_{M_a}L'_i}{M_j}:=\begin{cases}
	\frac{L_i}{M_a}\oplus \frac{L'_i}{M_a} \qquad \text{if } j=a\\
	\frac{L_i\star_{M_a}L'_i}{M_j} \qquad  \,\,\text{ by Axiom (K), if } 1\leq j<a \,\,
	\end{cases},$$
and recursively,
	$$\frac{L_i\star_{M_a} L'_i}{L_j\star_{M_a}L'_j}:=F\left(\frac{J_i\oplus J'_i}{J_j\oplus J'_j}\right).$$

\end{itemize}
\end{construction}



\begin{claim}\label{claim:Hspace} The addition map above turns $\overline{M}|  F$ into a homotopy associative and commutative $H$-space, making $\piMSF$ a monoid.
\end{claim}
\begin{proof}
This follows from a series of basic checks:
	\begin{enumerate}[label=(\alph*)]
		\item {\em Well-definedness.} One can check Construction~\ref{cons:Hspace} indeed defines a simplicial map $$+\colon \overline{M}|F\times \overline{M}|F\to \overline{M}|F.$$
		Choices of quotients are valid in $\SC$ by Lemma~\ref{lem:Hspace1}. The top row of $W+W'$ defines a simplex in $Y$ since moral subsets are closed under direct sums.  Finally, apply Fact~\ref{facts:restrictedpushouts} and Lemma~\ref{lem:DirectSum} to verify that $(W,W')\mapsto W+W'$ commutes with face and degeneracy maps.
		\item {\em Identity.} The 0-simplex 
		\begin{equation}\label{eq:addID}
		\begin{pmatrix}  &(O_Y \rtail) & \doubleunderline{\,O\, }\\ O \rtail M_1\rtail \dots \rtail &M_a  \xrtail{1} & M_a \end{pmatrix}, 
		\end{equation}
	represents the additive identity. 
Why? Use the fact that restricted pushouts are initial to deduce $J\oplus O\cong J$ for any $J\in\calC$, and $L_i\star_{M_a} M_a\cong L_i$. When adding the higher $n$-simplices, extend the 0-simplex by adding degeneracies.
		\item {\em Associativity and Commutativity.} Since restricted pushouts (and direct sums) are initial, they are associative and commutative up to natural isomorphism. These induce simplicial homotopies that make $|\MSF|$ into a homotopy associative and commutative $H$-space. 
	\end{enumerate}
\end{proof}

\subsection{The Main Result} We now leverage the $H$-space structure of $O|F$ to describe the right fiber, assuming $F$ satisfies a technical condition we call {\em cofinality.} 

\begin{definition}[Cofinality]\label{def:cofinal} Suppose $F\colon Y\to \calS\calC$ is a $k$-moral subset. We say $F$ has {\em cofinal image} if:
\begin{enumerate}[label=(\roman*)]
	\item {\bf Case 1: $k=0$.} For any $O\rtail M\in \SC,$ there exists $T\in\ob\calC$ such that $$O\rtail M\oplus T \quad \in Y[1].$$
	\item {\bf Case 2: $k\geq 1$.} For any $C\in\ob\calC$, 
	$$O_Y\rtail C\,\,\in Y[0]\qquad \text{and}\qquad O_Y\rtail O\rtail C\in Y[1].$$
\end{enumerate}	
\end{definition}

\begin{example} It is clear the map $q\colon O|\SC\to \SC$ has cofinal image by construction.
\end{example}



This sets up the main theorem of this section.

\begin{theorem}\label{thm:thmC'}  Let $\calC$ be a pCGW category, and $F\colon Y\to \calS\calC$ a moral subset with cofinal image. \underline{Then} 
	\begin{equation}\label{eq:ThmC'}
	\begin{tikzcd}
	O|F\ar[d] \ar[r] & Y\ar[d]\\
	O|\calS\calC \ar[r]& \calS\calC 
	\end{tikzcd}	
	\end{equation}
	is a homotopy Cartesian square. 
\end{theorem}

\begin{proof} Say the map $F$ is \emph{dominant} if $+$ induces a group structure on $\piMSF$ given \emph{any} $\overline{M}\in\calS\calC(A)$ for \emph{any} $A\in\Delta$. By Claim~\ref{claim:Hspace}, we already know $\piMSF$ is a monoid; dominance thus asserts the existence of inverses. Theorem~\ref{thm:thmC'} then follows from establishing two main implications.
	\begin{itemize}
		\item Step 1: If $F$ has cofinal image, then $F$ is dominant.
		\item Step 2: If $F$ is dominant, then Diagram~\eqref{eq:ThmC'} is a homotopy 
		Cartesian square.
	\end{itemize}
	
\subsubsection*{Step 1: Cofinality implies dominance} Fix some $\overline{M}\in\calS\calC(A)$. The monoid
 $\piMSF$ can be presented as
\begin{itemize}
	\item Generators: Vertices of $\overline{M}| F$, represented as
	$$W=\begin{pmatrix}  &&&&(O_Y&\rtail) &  \doubleunderline{J_0}\\ O\rtail M_1&\rtail &\dots &\rtail &M_a &\rtail &  N\end{pmatrix}$$
	
	\item Relations: 1-simplices of $\overline{M}|  F$.
\end{itemize} 	
As before, we omit the choice of quotients. In fact, any two vertices with the same presentation will have isomorphic quotients, and are thus connected by a 1-simplex.\footnote{The choice of 1-simplex is obvious, but justification takes some work. Informal proof sketch: construct the obvious simplex for the bottom row using the fact that quotients respect filtrations (Lemma~\ref{lem:quotFilt}) and are unique up to isomorphism (Axiom (K)). To show it is indeed a simplex, we must show the rightmost column of isomorphism squares are all distinguished; by goodness, it suffices to show they commute in the ambient category. We get the top square for free by Axiom (K). To see the square below it also commutes, use the fact that all $\E$-morphisms are monic  (if $\E\subseteq \calC$) or epi (if $\E^{\opp}\subseteq \calC$); this uses Axiom (M). Keep going.}
Hence, simplify the generators to
\begin{equation}\label{eq:W}
W=\begin{pmatrix}
(O_Y&\rtail) & \doubleunderline{J_0} \\ M_a&\rtail& N
\end{pmatrix}.
\end{equation}

To construct its inverse in $\piMSF$, it will be helpful to consider the case of $0$-moral subsets separately.
\subsubsection*{\textbf{Case I: $k= 0$}} Since $F$ has cofinal image, there exists some $T\in\ob\calC$ such that 
$$O\rtail \frac{N}{M_a}\oplus T\quad \in Y[1].$$
Setting $N':=M_a\oplus T$, it follows from Fact~\ref{facts:restrictedpushouts} that $$N\star_{M_a} N'\cong N\oplus T.$$ 
Thus, define a new vertex $W'$ whereby
$$W+ \underbrace{\begin{pmatrix}
&& \doubleunderline{O} \\ M_a&\rtail& N'
\end{pmatrix}}_{W'}=\begin{pmatrix}
&& \doubleunderline{\quad \,O\, \quad } \\ M_a&\rtail& N\oplus T
\end{pmatrix}.$$
We show $W'$ is the inverse of $W$ by verifying that $W+W'$ lies in the connected component of the additive identity. This is because
$$\begin{pmatrix}
&& \doubleunderline{\quad O \rtail \frac{N}{M_a}\oplus T\quad\,\, } \\ M_a&\rtail& M_a \rtail N\oplus T
\end{pmatrix}$$
defines a 1-simplex in $\piMSF$, since $\frac{N}{M_a}\oplus T\cong \frac{N\oplus T}{M_a}$ by Lemma~\ref{lem:ADTsquares}. Since the $0^{\mathrm{th}}$ face map acts by forgetting the second $M_a$ on the bottom row and factoring the top by $\frac{N}{M_a}\oplus T$, conclude that the 1-simplex connects the additive identity (Equation~\eqref{eq:addID}) to $W+W'$.

\subsubsection*{\textbf{Case II: $k\geq 1$}} The argument is similar. Given $W$ as above, use cofinality to define a vertex $W'$ whereby
$$W+ \underbrace{\begin{pmatrix}
	O_Y&\rtail & \doubleunderline{\quad \,\frac{N}{M_a}\quad \,} \\ M_a&\rtail& M_a\oplus J_0
	\end{pmatrix}}_{W'}=\begin{pmatrix}
O_Y&\rtail & \doubleunderline{ \frac{N}{M_a} \oplus J_0} \\ M_a&\rtail& N\oplus J_0
\end{pmatrix}.$$
The sum $W+W'$ is then connected to the additive identity via the 1-simplex
$$\begin{pmatrix}
O_Y&\rtail& \doubleunderline{ O \rtail \frac{N}{M_a}\oplus J_0} \\ M_a&\rtail& M_a \rtail N\oplus J_0
\end{pmatrix}.$$

In sum: given any $W\in \piMSF$, we can use cofinality to construct its inverse with respect to $+$. Since  $\overline{M}\in\calS\calC(A)$ and $A\in\Delta$ were chosen arbitrarily, this shows that $F$ is dominant.
	
	\subsubsection*{Step 2: $O|F$ as a homotopy pullback} Fix $\overline{M}\in\calS\calC(A)$ for some $A\in\Delta$, and fix $f\colon A'\to A$. By Theorem B' (Theorem~\ref{thm:B'}), it suffices to show that the base-change map $f^\ast\colon \MSF\to f^\ast\MSF $ is a homotopy equivalence.  By Step 1, we can use the fact that $\piMSF$ is a group since $F$ is dominant.
	
	\subsubsection*{Step 2a: A reduction} 
	Let $g\colon [0]\to A'$ be any morphism in $\Delta$. To show that $f^\ast$ is a homotopy equivalence, it suffices to show that $(fg)^\ast=g^\ast f^\ast$ and $g^\ast$ are. In fact, since both $fg$ and $g$ have $[0]$ as source, it suffices to show that 
	$$f_i\colon [0]\to A,\qquad f_i(0)=i\,\,\text{for}\, i\in A,$$
	induces a homotopy equivalence for any $A\in\Delta$. Notice $f^\ast_i$ defines a map
	$$f^\ast_i\colon \MSF \to O| F$$
	since $O$ is the only vertex of $\calS \calC$. 
	
	\subsubsection*{Step 2b: The base case} Define a map 
	\begin{align}
	h_0\colon O|  F &\longrightarrow \MSF \nonumber\\
	{\begin{pmatrix} &(O_Y\rtail) & \doubleunderline{J_0\rtail \dots \rtail J_q}\\ &O \rtail & L_0\rtail \dots \rtail L_q \end{pmatrix}} & \longmapsto {\begin{pmatrix} & (O_Y\rtail) & J_0\rtail \dots \rtail J_q\\ O\rtail \dots \rtail  &M_a \rtail &\doubleoverline{M_a\oplus  L_0\rtail \dots \rtail M_a\oplus L_q}\end{pmatrix}}\nonumber 
	\end{align}
	with quotients defined as
	\begin{itemize}
		\item $\frac{M_a\oplus L_j}{M_a\oplus L_k}:=\frac{L_j}{L_k}\left(= F\left(\frac{J_j}{J_k}\right)\right)$, 
		\item $\frac{M_a\oplus L_j}{M_a}:=L_j$, \qquad $\frac{M_a \oplus L_j}{M_i}:=\frac{M_a}{M_i}\oplus L_j$ 
	\end{itemize}

	To show that $f^\ast_i$ is a homotopy equivalence (for arbitrary $i$), it suffices to establish the following claim.
	
	\begin{claim}\label{claim:thmC'} The maps $f^\ast_i\circ h_0$ and $h_0$ are homotopy equivalences. 
	\end{claim}
	\begin{proof}[Proof of Claim] Two main checks.
		\begin{enumerate}[label=(\roman*)]
			\item \emph{On $f^\ast_i\circ h_0$.} The map $f^\ast_i\circ h_0\colon O|F\to O|F$ sends 
			
			$$ \begin{pmatrix} & (O_Y\rtail)& \doubleunderline{J_0\rtail \dots \rtail J_q}\\ &O \rtail & L_0\rtail \dots \rtail L_q \end{pmatrix}  \longmapsto
			\begin{pmatrix}  (O_Y\rtail) & J_0\rtail \dots \rtail J_q\\ O \rtail &\doubleoverline{\frac{M_a}{M_i}\oplus  L_0\rtail \dots \rtail \frac{M_a}{M_i}\oplus  L_q}\end{pmatrix} $$ 
			for $0\leq i\leq a$. Since $O|F$ is an $H$-space, we can reformulate $f_i^\ast\circ h_0$ more suggestively as 
			$$\left(f_i^\ast\circ h_0\right) (W) = W + \begin{pmatrix}
			(O_Y\rtail) & O \\
			O\rtail & \doubleoverline{\frac{M_a}{M_i}}
			\end{pmatrix}^{(q)},$$
where the notation ``$(q)$'' indicates a (potentially degenerate) $q$-simplex formed by repeating the vertex $q$ times. Since $\pi_0(O|F)$ is a group on the vertices of $O|F$, there exists a vertex $V$ such that 
			$$\begin{pmatrix}
		(O_Y\rtail )	& O \\
			O\rtail & \doubleoverline{\frac{M_a}{M_i}}
			\end{pmatrix} +   V \sim \begin{pmatrix}
		(O_Y\rtail )	& O \\
			O\rtail & \doubleoverline{O}
			\end{pmatrix}.$$
Since $+$ is a simplicial map, it commutes with face and degeneracy maps. In particular, degeneracies preserve identities and inverses: if $x + x^{-1} \sim e$ for a vertex $x$, then $$s_0^q(x) + s_0^q(x^{-1})\sim s_0^{q}(x+x^{-1}) \sim s_0^q(e).$$
Thus, define $h_1\colon O|F\to O | F$ by setting
			$$ h_1(W) = W +  V^{(q)}$$
			for any simplex $W$. Since $+$ is homotopy associative and commutative, conclude that
			$$f_i^\ast \circ h_0\circ h_1\sim 1,\qquad h_1\circ f_i^\ast \circ h_0\sim 1.$$
			\item \emph{On $h_0$.} Notice: $f^\ast_a\circ h_0$ is isomorphic to the identity map on $O|F$. It therefore suffices to show $h_0\circ f^\ast_a$ is homotopic to the identity map $1$ on $\MSF$. But this follows from the natural isomorphism 
			$$h_0\circ f^\ast_a \cong 1,$$
			induced by the isomorphism
			$$M_a\oplus \frac{L_j}{M_a}\cong \frac{M_a\oplus L_j}{M_a}= L_j,\qquad \text{for all}\, j,$$
		a consequence of Lemma~\ref{lem:ADTsquares} (iii) and the choice of quotients by $h_0$.\footnote{Notice the specific choice of quotients by $h_0$ is crucial; otherwise, the isomorphism may fail to hold since e.g. not all short exact sequences split.}
		\end{enumerate}
		This completes proof of Claim~\ref{claim:thmC'}.
	\end{proof}
	
	\subsubsection*{Step 3: Finish.} Fix a $k$-moral subset $F\colon Y\to\calS\calC$ that has cofinal image. Step 1 showed that $F$ is dominant, i.e. $\piMSF$ is a group for any $\overline{M}\in\SC(A)$. Step 2 combines this with Theorem B' to show that Diagram~\eqref{eq:ThmC'} is homotopy Cartesian. This proves the theorem.
\end{proof}

The following corollary justifies viewing the right fiber (Definition~\ref{def:fiber}) as the simplicial analogue of a homotopy fiber, and will be useful in Section~\ref{sec:generators}.

\begin{corollary}\label{cor:HtpyFib} Suppose $F\colon Y\to \calS\calC$ is $k$-moral subset with cofinal image. Then $|O|F|$ is homotopy equivalent to the homotopy fiber of $|F|$.
\end{corollary}
\begin{proof} By Observation~\ref{obs:homotopycartesian},  $O|\calS\calC$ is contractible. Since the homotopy fiber of $|F|$ is the homotopy pullback of the cospan  $\ast \to |\calS\calC| \xleftarrow{|F|} Y $, the statement follows.
\end{proof}

In addition, we now obtain a key result of this paper regarding the $G$-construction.

\begin{theorem}\label{thm:Gconstruction} Given any pCGW category $\calC$, there is a homotopy equivalence $$|G\calC|\xrightarrow{\sim}\Omega |S\calC|.$$
Further, direct sum induces an $H$-space structure on $G\calC$.
\end{theorem}
\begin{proof} Let us review Key Observation~\ref{obs:homotopycartesian}. By item (iii), $|G\calC|=|\Omega\calS\calC|\simeq \Omega|\calS\calC|$ iff
		\begin{equation*}
	\begin{tikzcd}
	\Omega \calS\calC \ar[r,"t"] \ar[d,"b"] & O|\calS\calC \ar[d,"q"]\\
	O|\calS\calC \ar[r,"q"] & \calS\calC
	\end{tikzcd}
	\end{equation*}
is homotopy Cartesian. By item (ii), we have $O|q\simeq \Omega \calS\calC$. Since $q$ is a $1$-moral subset with cofinal image, the rest follows from Theorem~\ref{thm:thmC'}; the $H$-space structure on $G\calC$ comes from Construction~\ref{cons:Hspace}.
\end{proof}

\begin{remark} The equivalence in Theorem~\ref{thm:Gconstruction} established between $G\calC$ and $K\calC$ is one of topological spaces, not of infinite loop spaces or spectra.
\end{remark}

\begin{discussion}[Comparison with other proofs] Our approach combines ideas from two existing proofs of Theorem~\ref{thm:Gconstruction} when $\calC$ is an exact category. While all three proofs rely on Theorem B', there are some key differences worth mentioning.
	
	In broad strokes, Theorem B' says: if the induced map on fibers $\rho|F\to f^\ast \rho|F$ is a homotopy equivalence for {\em any} $f\colon A'\to A$, then $\rho|F$ is a homotopy pullback. In their original  paper \cite{GG}, Gillet-Grayson simplifies this condition by restricting to the maps $f_0,f_1\colon [0]\to [1]$; compare this with Step 2 of our proof of Theorem~\ref{thm:thmC'}.  Their argument is technical, but reflects the informal intuition: to understand how objects break into finitely many pieces, it suffices to understand how to break a single object into two.
	
	A different proof appears in \cite[Thm 8.2]{GraysonThmC}, where Grayson uses Theorem B' to study {\em dominant} exact functors. By contrast, notice we define dominance for {\em simplicial maps}. This adjustment accounts for the fact that $O|\SC$ does not {\em a priori} correspond to a pCGW category. Further, our definitions of moral subsets and cofinality were designed for our arguments to go through smoothly; no real attempts at generality were made. Whereas Grayson gives a full characterisation of dominant functors \cite[Thm 2.1]{GraysonThmC}, his proof does not translate well to our setting; for our purposes, it suffices to identify a sufficient criterion for dominance, i.e. cofinality.
\end{discussion}

\section{Generators of $K_1(\calC)$}\label{sec:generators}

Having established Theorem~\ref{thm:Gconstruction}, we now begin to deliver on our promise that the $G$-construction leads to an explicit description of $K_1(\calC)$. Section~\ref{sec:Gcons} unpacks the definition of the $G$-construction. Sections~\ref{sec:techprelim} and ~\ref{sec:genK1} work to obtain an increasingly sharp description of the generators of $K_1$; our analysis extends various results from \cite{ShermanAbelian,ShermanExact,NenGen}.

\subsection{Review of $G$-Construction}\label{sec:Gcons} 

\begin{construction}[$G$-Construction]\label{cons:Gcons} An $n$-simplex of $G\calC$ is a pair of flag diagrams of the form
	
	\[\small{\begin{tikzcd}
		P_0\ar[r, >->] & P_1 \ar[r, >->] & P_2 \ar[r, >->] & \cdots\ar[r, >->]  & P_n \\
		&	P_{1/0} \ar[r, >->] \ar[u, {Circle[open]}->]& P_{2/0} \ar[r, >->]\ar[u, {Circle[open]}->]  & \cdots \ar[r,>->] & P_{n/0}  \ar[u, {Circle[open]}->]\\
		&& P_{2/1} \ar[r,>->]  \ar[u, {Circle[open]}->]  & \cdots   \ar[r,>->] & P_{n/1} \ar[u, {Circle[open]}->] \\
		&&& \vdots\ar[u, {Circle[open]}->]  & \vdots \ar[u, {Circle[open]}->] \\
		&&&& P_{n/n-1} \ar[u, {Circle[open]}->] 
		\end{tikzcd}}\quad \small{\begin{tikzcd}
		P'_0\ar[r, >->] & P'_1 \ar[r, >->] & P'_2 \ar[r, >->] & \cdots\ar[r, >->]  & P'_n \\
		&	P_{1/0} \ar[r, >->] \ar[u, {Circle[open]}->]& P_{2/0} \ar[r, >->]\ar[u, {Circle[open]}->]  & \cdots \ar[r,>->] & P_{n/0}  \ar[u, {Circle[open]}->]\\
		&& P_{2/1} \ar[r,>->]  \ar[u, {Circle[open]}->]  & \cdots   \ar[r,>->] & P_{n/1} \ar[u, {Circle[open]}->] \\
		&&& \vdots\ar[u, {Circle[open]}->]  & \vdots \ar[u, {Circle[open]}->] \\
		&&&& P_{n/n-1} \ar[u, {Circle[open]}->] 
		\end{tikzcd}}\]
	subject to the conditions: 
	\begin{enumerate}[label=(\roman*)]
	\item All drawn squares are distinguished. Further, the quotient index square
		\[ \begin{tikzcd}
		P_{j/i} \ar[r,>->] \ar[dr, phantom, "\square"] & P_{k/i} \\
		P_{j/l} \ar[u, {Circle[open]}->]\ar[r,>->]& P_{k/l} \ar[u, {Circle[open]}->]
		\end{tikzcd}\qquad \text{where $i<l<j<k$}\]
is the same in both flag diagrams. 
		\item Every quotient index triangle defines a distinguished square 
		\[\begin{tikzcd}
		P_{j/i} \ar[r,>->] \ar[dr, phantom, "\square"]& P_{k/i} \\
		O \ar[u, {Circle[open]}->]\ar[r,>->]& P_{k/j} \ar[u, {Circle[open]}->]
		\end{tikzcd} \qquad  \text{for any $i<j<k$},\]
		and is the same in both flag diagrams.
		\item Any $P_j\rtail P_k$ and $P'_j\rtail P'_k$ in the filtration can be completed into distinguished squares
		\begin{equation*}\label{eq:triangleDS}
		\begin{tikzcd}
		P_{j} \ar[dr, phantom, "\square"]\ar[r,>->] & P_k \\
		O \ar[u, {Circle[open]}->]\ar[r,>->]& P_{k/j} \ar[u, {Circle[open]}->]
		\end{tikzcd},\quad \begin{tikzcd}
		P'_{j} \ar[r,>->]\ar[dr, phantom, "\square"] & P'_k \\
		O \ar[u, {Circle[open]}->]\ar[r,>->]& P_{k/j} \ar[u, {Circle[open]}->]
		\end{tikzcd}\qquad  \text{for any $j<k$}.
		\end{equation*}
	\end{enumerate}
\end{construction}

\begin{convention}\label{con:vertex} Technically, an $n$-simplex of $G\calC$ is a pair of $(n+1)$-simplices in $\calS\calC$
	$$O\rtail P_0\rtail \dots \rtail P_n\qquad \qquad O\rtail P'_0\rtail \dots \rtail P'_n$$
	such that the $0^\text{th}$ faces (= forgetting $O$ and quotienting by $P_0$) agree. Here we omit the base-point $O$ for simplicity. Therefore, a {\em vertex} of $G\calC$ is a pair $(M,N)\in\calC\times\calC$,
	  and an {\em edge} $(M,N)\to (M',N')$ is a pair of distinguished squares with identical quotient
\[\left(\dsquare{O}{C}{M}{M'}, \dsquare{O}{C}{N}{N'}\right).\]
We sometimes also represent a vertex as $\begin{pmatrix}
M\\N
\end{pmatrix}$.
\end{convention}


\subsection{Sherman Loops \& Splitting}\label{sec:techprelim} Here we extend Sherman's analysis of $K_1$ for exact categories \cite{ShermanAbelian,ShermanExact} to pCGW categories. A guiding principle in his approach is that restricting to split exact sequences helps clarify the general case. This section adapts Sherman's insight to the pCGW setting, culminating in a first characterisation of the generators of $K_1(\calC)$, Theorem~\ref{thm:Sherman}. 


\begin{convention}\label{conv:K1} Denote $G\calC^{\mathsf{o}}$ to be the component of $G\calC$ containing the vertex $(O,O)$; we call this the {\em base-point component}. Hereafter, we take as definition
	$$K_1(\calC):=\pi_1|G\calC^{\mathsf{o}}|.$$
Notice: by Theorem~\ref{thm:GG}, this is isomorphic to the standard definitions of $K_1$ (e.g. via the $S_\bullet$-construction), justifying our choice of definition.
\end{convention}

\begin{construction}[Sherman Loop]\label{cons:eltG}  A \emph{Sherman triple} $(\alpha,\beta,\theta)$  consists of the following data:
\begin{itemize}
	\item Two $\M$-morphisms $A\xrtail{\alpha} B,$ $A'\xrtail{\beta} B'$;
	\item An isomorphism $\theta\colon A\oplus C\oplus B'\rtail A'\oplus C' \oplus B$, where $C$ and $C'$ are specific choices of quotients
	\begin{equation}\label{eq:doublesq}
	\dsquaref{O}{C}{A}{B}{}{}{\alpha}{\delta}\qquad \dsquaref{O}{C'}{A'}{B'}{}{}{\beta}{\gamma}.
	\end{equation}	
\end{itemize}
Its associated \emph{Sherman loop} is the homotopy class $G(\alpha,\beta,\theta)$ in $K_1(\calC)$ represented by the loop 
\begin{equation}\label{eq:shermanloop}
\begin{pmatrix}
O \\ O
\end{pmatrix} \to  \ones{A}{A} \to  \ones{A\oplus C \oplus B'}{B\oplus B'} \to  \ones{A'\oplus C'\oplus B}{B'\oplus B} \leftarrow \ones{A'}{A'} \leftarrow \ones{O}{O},
\end{equation}
where the arrows denote the obvious 1-simplices.\footnote{\label{fn:1simSherLoop}\emph{Details.} The middle arrow in Equation~\eqref{eq:shermanloop} applies $\theta$ on the top row and the canonical permutation isomorphism $B\oplus B'\to B'\oplus B$ on the bottom; the rest of the 1-simplices are defined by applying Lemma~\ref{lem:ADTsquares}. 
}  
\end{construction}

\begin{remark} Fix a pair of $\M$-morphisms $\alpha,\beta$. One easily checks that any two Sherman triples $(\alpha,\beta,\theta)$ and $(\alpha,\beta,\theta')$ define the same Sherman loop in $K_1(\calC)$ up to homotopy -- see e.g. \cite[\S 1]{ShermanAbelian}.
\end{remark}

\begin{definition}[Split] Call a distinguished square of the form 
	\begin{equation*}\label{eq:splitDS}
	\dsquaref{O}{C}{A}{B}{}{}{f}{g}
	\end{equation*} 
	an \emph{exact square}. 
	\begin{enumerate}[label=(\roman*)]
		\item An exact square is called {\em split} if there exists an isomorphism 
		$$\Psi\colon B\rtail A\oplus C$$ 
		such that the following squares commute
		\[\begin{tikzcd}
		A \ar[r,>->,"f"] \ar[d,>->,swap,"1_A"]& B \ar[d,>->,"\Psi"] \\
		A \ar[r,>->,"p_A"] & A\oplus C
		\end{tikzcd}\qquad \begin{tikzcd}
		B \ar[d, {Circle[open]}->,swap,"\varphi(\Psi)"] & C \ar[l, {Circle[open]}->,swap,"g"] \ar[d, {Circle[open]}->,"1_C"]\\
		A\oplus C & C \ar[l, {Circle[open]}->,swap,"q_C"]	 	 
		\end{tikzcd}\]
		in $\M$ and $\E$ respectively, where $p_A$ and $q_C$ are the morphisms from the obvious direct sum square.
		\item  Call an $\M$-morphism $A\rtail B$ is called \emph{split} if its corresponding exact square [obtained from Axiom (K)] is split.
	\end{enumerate}
	
\end{definition}

\begin{remark} A sanity check: suppose $\calC$ is an exact category, and so $\varphi(\Psi)$ corresponds to $$\Psi^{-1}\colon A\oplus C\to B$$
by Axiom (I). A split exact square means there exists an isomorphism $\Psi$ such that $\Psi\circ f = p_A$ and $g\circ \Psi^{-1}=q_C$ in the ambient category. Since $g\circ \Psi^{-1}=q_C\iff g=q_C\circ \Psi$, this recovers the usual definition of a split exact sequence. 
\end{remark}

\begin{theorem}\label{thm:Sherman} $K_1(\calC)$ is generated by Sherman loops $G(\alpha,\beta,\theta)$. 
\end{theorem}
\begin{proof} The proof relies on two helper constructions: (i) $\Cplus$, a pCGW subcategory of $\calC$; and (ii) $\Ghat$, a simplicial subset of $G\calC$. These encode the splitting data of $\calC$ and $G\calC$ respectively, and so admit particularly nice presentations on the level of $\pi_1$. We will relate them to $K_1(\calC)$ to establish the theorem. 
	
Throughout, all $\pi_1$'s are taken at the obvious base-point components (e.g. the one containing $(O,O)$ in $\Ghat$), which we suppress from the notation.

\subsubsection*{Step 0: Setup} This step introduces the two helper constructions, and records some preliminary observations. 

\begin{construction} Define $\calC^{\oplus}$ to be the CGW subcategory of $\calC$ whose exact squares are precisely the split exact squares of $\calC$.
\end{construction}

Since $\calS\Cplus$ is closed under taking direct sums, the simplicial map
$$\calS F\colon \calS\Cplus\hookrightarrow \calS\calC,$$
induced by the inclusion CGW functor $F\colon \Cplus \hookrightarrow \calC$, is a $0$-moral subset. In fact, since $\calS\Cplus[1]=\SC[1]$, $\calS F$ has cofinal image. One can therefore apply Theorem~\ref{thm:thmC'} to construct a homotopy Cartesian diagram
 \begin{equation}\label{eq:SherHCD}
 \begin{tikzcd}
 O|\calS F \ar[d,"r"] \ar[r,"p"] & \calS \Cplus \ar[d,"\calS F"]\\
 O|\calS\calC \ar[r,"q"]& \calS\calC 
 \end{tikzcd}\quad .
 \end{equation}

\begin{construction} Following the convention of Construction~\ref{cons:Hspace}, represent an $n$-simplex in $G\calC$ as
	$$\begin{pmatrix} O \rtail& \doubleunderline{L_0\rtail \dots \rtail L_n}\\ O \rtail& M_0\rtail \dots \rtail M_n \end{pmatrix}$$
	Define $\Ghat$ to be the simplicial subset of $G\calC$ where the $\M$-morphisms in the top row are all split. 
\end{construction}

Some observations. Notice this mirrors $O|\calS F$, except the top row of an $n$-simplex in $\Ghat$ is an $(n+1)$-simplex in $\SC$ whereas the top row in $O|\calS F$ is an $n$-simplex (whose $\M$-morphisms are also all split). In addition, the simplices used to define $G(\alpha,\beta,\theta)$ also belong to $\Ghat$, and so $G(\alpha,\beta,\theta)\in \pi_1|\Ghat|$.

\subsubsection*{Step 1: Relating $\Ghat$ and $\calS\Cplus$} By Corollary~\ref{cor:HtpyFib}, $|O|p|$ and $|O|q|$ define homotopy fibers. We thus work towards extracting a homotopy fiber sequence from Diagram~\eqref{eq:SherHCD}. 

\begin{claim}\label{claim:Sherhtpyfib} $|\Ghat|$ is homotopy equivalent to $|O|p|$ and $|G\calC|$.
\end{claim}
\begin{proof} Take the geometric realisation of Diagram~\eqref{eq:SherHCD} to obtain a homotopy Cartesian square of spaces. As is well-known, the map $r$ induces a homotopy equivalence on the homotopy fibers $|O|p|\to |O|q|$. Next, Observation~\ref{obs:homotopycartesian} (ii) notes that $G\calC \simeq O|q$, essentially by unpacking definitions. One can similarly verify that $\Ghat\simeq O|p$. Assembling all the equivalences, conclude that $|G\calC|\simeq |O|q|\simeq |O|p|\simeq  |\Ghat|$.
\end{proof}

In particular, consider the homotopy fiber sequence associated to $|O|p|\to |\OSF| \xrightarrow{p} |\calS \Cplus|$. Unwinding the proof of Claim~\ref{claim:Sherhtpyfib} yields the following key observation.

\begin{observation}\label{obs:htpyseq} The homotopy equivalence $O|p\simeq \Ghat$ identifies a projection map $v\colon \Ghat\to \OSF$, sending an $n$-simplex to its bottom row. In particular, the exact sequence 
	\begin{equation}
	\begin{tikzcd}
 \pi_1\Omega \gm{\calS\Cplus} \ar[r] &	\pi_1\gm{O|p} \ar[r] & \pi_1\gm{O|\calS F}\ar[r,"p_\ast"] & \pi_1\gm{\calS\Cplus}
	\end{tikzcd}
	\end{equation}
may be reformulated as
	\begin{equation}\label{eq:SherHseq}
	\begin{tikzcd}
\pi_1\Omega\gm{\calS\Cplus} \ar[r] &	\pi_1\gm{\Ghat}\ar[r,"v_\ast"] & \pi_1\gm{O|\calS F}\ar[r,"p_\ast"] & \pi_1\gm{\calS\Cplus}
	\end{tikzcd},
	\end{equation}
where $v_\ast$ is induced by $v$. For further details, see \cite[Cor. 1.8]{GG}. 
\end{observation}

\subsubsection*{Step 2: Generators of $\pi_1|O|\calS F|$} The argument is standard -- no surprises. The 1-simplices of the form
\begin{equation}\label{eq:treeOSF}
\ones{ & O\rtail N}{ O\rtail & \doubleoverline{O\rtail N}}
\end{equation}
yield a maximal tree for the 1-skeleton of $|\OSF|$, connecting the base-point of $\OSF$ to any of its vertices. 
Thus by \cite[Lemma IV.3.4]{WeibelKBook}, the total set of 1-simplices of $\OSF$ generate $\pi_1\gm{\OSF}$. We therefore represent the generators of $\pi_1|\OSF|$ as 
\begin{equation}
\ones{ & O\rtail A}{ O\rtail & \doubleoverline{O\rtail A}} \ones{ & O\rtail C}{ O\rtail & \doubleoverline{A\rtail B}} \ones{ & O\rtail B}{ O\rtail & \doubleoverline{O\rtail B}}^{-1}.
\end{equation}
For simplicity, generators are sometimes represented just by the middle term above, since the other two edges can be recovered from the maximal tree.

\subsubsection*{Step 3: Constructing a Sherman Loop} Review Homotopy Exact Sequence~\eqref{eq:SherHseq}. Given $v_*(x)\in\pi_1|\OSF|$ for any $x\in \pi|\Ghat|$, we can use its presentation from Step 2 to construct a Sherman Loop $G(\alpha,\beta,\theta)$.

\begin{itemize}
	\item {\em The two $\M$-morphisms.} Since $\OSF$ is an $H$-space,  $v_*(x)$ can be expressed as a difference of two 1-simplices, let us say
	\begin{equation}\label{eq:Sher1sim}
	\ones{ & O\rtail C}{ O\rtail & \doubleoverline{A\rtail B}} \quad\text{and}\quad \ones{ & O\rtail C'}{ O\rtail & \doubleoverline{A'\rtail B'}}.
	\end{equation}
	This yields the $\M$-morphisms $\alpha\colon A\rtail B$ and $\beta\colon A'\rtail B'$.
	\item The isomorphism $\theta$. Recall the map $p\colon \OSF\to\calS\Cplus$ acts by projection on the top row. Thus $p_\ast v_\ast(x)\in \pi_1|\calS\Cplus|$ corresponds to the difference between
	\begin{equation}\label{eq:SherSCloops}
	(O\rtail A)(O\rtail C)(O\rtail B')^{-1}\qquad \text{and} \qquad (O\rtail A')(O\rtail C')(O\rtail B')^{-1}. 
	\end{equation}
	By Presentation Theorem~\ref{thm:PresK0}, we know that $\pi_1|\SC^{\oplus}|=K_0(\Cplus)$, so let us rewrite the above as
	\begin{equation}\label{eq:SherDiff}
	[A]+[C]-[B] \qquad \text{and} \qquad [A']+[C']-[B']. 
	\end{equation}
	We don't know if, e.g. $[A]+[C]=[B]$ in $K_0(\Cplus)$ since $A\rtail B$ may not be split. Nonetheless, by exactness of Homotopy Sequence~\eqref{eq:SherHseq}, we can deduce
	 \begin{equation}
	 [A]+[C]-[B] - \left([A']+[C']-[B']\right) = 0,
	 \end{equation}
	 and so
	 \begin{equation}
	 [A]+[C]+[B'] = [A']+[C']+[B].
	 \end{equation}
Since $[M]=[N]$ in $K_0(\Cplus)$ iff $M$ and $N$ are stably isomorphic\footnote{The same argument for exact categories extends to the pCGW setting.},
there exists some $Z\in\calC$ such that 
	 \begin{equation}
	 A\oplus C \oplus B' \oplus Z \cong A'\oplus C' \oplus B \oplus Z. 
	 \end{equation}	 
	Now notice that
	 \begin{equation}\label{eq:Z}
	 \ones{ & O\rtail Z}{ O\rtail & \doubleoverline{O\rtail Z}} \ones{ & O\rtail O}{ O\rtail & \doubleoverline{Z\rtail Z}} \ones{ & O\rtail Z}{ O\rtail & \doubleoverline{O\rtail Z}}^{-1}
	 \end{equation}
	 is null-homotopic, Hence, we can always modify the representation of $v_*(x)$ by adding Loop~\eqref{eq:Z} to Equation~\eqref{eq:Sher1sim} without changing the homotopy class. As such, without loss of generality, assume $Z=O$, giving the isomorphism 
	 \begin{equation}\label{eq:Shertheta}
	 \theta\colon A\oplus C\oplus B'\xrightarrow{\cong} A'\oplus C'\oplus B.
	 \end{equation}
\end{itemize}


\subsubsection*{Step 4: An Equivalence} The following claim tells us that any $x\in \pi_1|\Ghat|$ looks like a Sherman loop when viewed in $\pi_1|\OSF|$.

\begin{claim}\label{claim:SherLookLoop} Given any $x\in \pi_1|\Ghat|$, there exists a Sherman loop $G(\alpha,\beta,\theta)$ such that $v_*(x)$ and $v_*(G(\alpha,\beta,\theta))$ are homotopic.
\end{claim}
\begin{proof} Given $x\in \pi_1|\Ghat|$, construct a Sherman Loop $G(\alpha,\beta,\theta)$ as in Step 3. In particular, $v_*(G(\alpha,\beta,\theta))$ is well-defined since  $G(\alpha,\beta,\theta)\in \pi_1|\Ghat|$.
	
Consider the following diagram in $\pi_1|\OSF|$
\begin{equation}\label{eq:SherLoopHtpy}
\small{\begin{tikzcd}
\ones{ & O}{ O\rtail & \doubleoverline{A}} \ar[d,red,""{name=X1}] \ar[r,blue]& \ones{ & O}{ O\rtail & \doubleoverline{B\oplus B'}} \ar[r,blue]& \ones{ & O}{ O\rtail & \doubleoverline{B'\oplus B}}&  \ones{ & O}{ O\rtail & \doubleoverline{A'}} \ar[l,blue] \ar[dl,red,""{name=X7},{xshift=10pt},{yshift=-5pt}] \\
\ones{ & O}{ O\rtail & \doubleoverline{B}} \ar[ur,""{name=X2}]  \ar[from=X2, to=X1,phantom, "\small{(1)}"] & \ones{&O}{O&\doubleoverline{O}} \ar[ur,""{name=X5}]\ar[l,red,""{name=X3}]  \ar[from=X3, to=X2,phantom, "\small{(2)}",{xshift=20pt},{yshift=5pt}] \ar[r,red]   \ar[u,""{name=X4}]  \ar[from=X4, to=X5,phantom, "\small{(3)}"] & \ones{ & O}{ O\rtail & \doubleoverline{B'}} \ar[u,""{name=X6}]\ar[from=X5, to=X6,phantom, "\small{(4)}",{yshift=-5pt}] \ar[u,""{name=X4}]  \ar[from=X6, to=X7,phantom, "\small{(5)}",{yshift=2.5pt}]
\end{tikzcd}}
\end{equation}
The edges of the diagram are obvious, and record various paths between vertices
\begin{equation}\label{eq:Shervertices}
\ones{ & O}{ O\rtail & \doubleoverline{A}} \dashrightarrow \ones{ & O}{ O\rtail & \doubleoverline{A'}},
\end{equation}
e.g. by concatenating along the blue edges, by concatenating along the red edges, etc.

A couple of key observations. First, notice that all triangles in Diagram~\eqref{eq:SherLoopHtpy} define boundaries of 2-simplices, listed below. 
\begin{align*}
\small{(1) \, \ones{ & O \rtail C \rtail  C\oplus B' }{ O\rtail & \doubleoverline{A\rtail B\rtail B\oplus B'}}, \,\, (2) \, \ones{ & O \rtail B \rtail  B\oplus B' }{ O\rtail & \doubleoverline{O\rtail B\rtail B\oplus B'}},\,\, (3) \, \ones{ & O \rtail B\oplus B' \rtail  B'\oplus B }{ O\rtail & \doubleoverline{O\rtail B\oplus B'\rtail B'\oplus B}} }
\end{align*}
\begin{align*}
\small{(4) \, \ones{ & O \rtail B' \rtail B' \oplus B }{ O\rtail & \doubleoverline{O\rtail B'\rtail B'\oplus B}}, \,\, (5) \, \ones{ & O \rtail C' \rtail  C'\oplus B }{ O\rtail & \doubleoverline{A'\rtail B'\rtail B'\oplus B}}.}
\end{align*}
Hence, the blue and red paths in Diagram~\eqref{eq:SherLoopHtpy} between the two vertices ~\eqref{eq:Shervertices} are homotopic. 

Second, recall from Observation~\ref{obs:htpyseq} that $v\colon \Ghat\to \OSF$ acts by projecting the bottom row. Thus, $v_*(G(\alpha,\beta,\theta))$ corresponds to the loop  
\begin{equation}
\footnotesize{	\ones{ & O\rtail A}{ O\rtail & \doubleoverline{O\rtail A}} \ones{ & O\rtail C\oplus B'}{ O\rtail & \doubleoverline{A\rtail B\oplus B'}} \ones{ & O\rtail O}{ O\rtail & \doubleoverline{B\oplus B'\rtail B'\oplus B}} \ones{ & O\rtail C\oplus B'}{ O\rtail & \doubleoverline{A'\rtail B'\oplus B}}^{-1} \ones{ & O\rtail A'}{ O\rtail & \doubleoverline{O\rtail A'}}^{-1}} \nonumber 
\end{equation}
whereas $v_*(x)$, the difference between two generators, corresponds to the loop 
\begin{equation}
\footnotesize{	\ones{ & O\rtail A}{ O\rtail & \doubleoverline{O\rtail A}} \ones{ & O\rtail C}{ O\rtail & \doubleoverline{A\rtail B}} 
\ones{ & O\rtail B}{ O\rtail & \doubleoverline{O\rtail B}}^{-1} \ones{ & O\rtail B'}{ O\rtail & \doubleoverline{O\rtail B'}} \ones{ & O\rtail C'}{ O\rtail & \doubleoverline{A'\rtail B'}}^{-1} \ones{ & O\rtail A'}{ O\rtail & \doubleoverline{O\rtail A'}}^{-1}}. \nonumber 
\end{equation}
In particular, the loop $v_*(G(\alpha,\beta,\theta))$ corresponds to composing along the blue edges in Diagram~\eqref{eq:SherLoopHtpy} whereas $v_*(x)$ corresponds to composing along the red edges. Hence, they are homotopy equivalent by our previous observation. Conclude that $v_*(G(\alpha,\beta,\theta))$ and $v_*(x)$ have the same homotopy class.
\end{proof}

\subsubsection*{Step 5: Finish} Let  $x$ be an element of $K_1(\calC)$. By Theorem~\ref{thm:Gconstruction} and Claim~\ref{claim:Sherhtpyfib}, we know
$$\Omega |\calS\calC|\simeq |G\calC|\simeq |\Ghat|,$$ 
and so regard $x\in \pi_1|\Ghat|$.\footnote{If $|G\calC|\simeq |\Ghat|$, why not simply regard $x\in \pi_1|G\calC|$ and use $\pi_1|G\calC|$ in Homotopy Sequence~\eqref{eq:SherHseq}? Answer: the map $v\colon \Ghat\to \OSF$ admits a useful explicit description (= projection of the bottom row), which streamlines the proof of Claim~\ref{claim:SherLookLoop}.} In particular, $x$ is an element in Homotopy Sequence~\eqref{eq:SherHseq}. By Claim~\ref{claim:SherLookLoop}, there exists a Sherman loop $G(\alpha,\beta,\theta)$ such that $v_*(G(\alpha,\beta,\theta))=v_*(x)$ in $\pi_1|\OSF|$. In other words, the difference $x-G(\alpha,\beta,\theta)$ vanishes in $\pi_1|\OSF|$, and thus lies in the image of 
$$ K_1(\Cplus)=\pi_1\Omega|\calS\Cplus|\longrightarrow \pi_1|\Ghat|=K_1(\calC).$$ To finish, we quote a couple of technical facts about Sherman Loops whose proof we defer to Appendix~\ref{app:Sherman Loops}. By Lemmas~\ref{lem:SherLoopSplit} and \ref{lem:SherSplitSher}, $K_1(\Cplus)$ is generated by Sherman Loops, and thus so is its image in $K_1(\calC)$. By Lemma~\ref{lem:SherLoopAdd}, the sum of two Sherman Loops is still a Sherman Loop. As such, since
$$x=x-G(\alpha,\beta,\theta)+G(\alpha,\beta,\theta),$$
conclude that $x$ is homotopic to a Sherman Loop in $K_1(\calC)$.
\end{proof}


\begin{discussion} Our proof streamlines Sherman's approach in \cite{ShermanExact}. For instance, Sherman's original argument for Claim~\ref{claim:Sherhtpyfib} proceeds by defining a pair of exact functors 
$$\Delta\colon \calC\to \calC\times \calC$$
$$\Delta'\colon \Cplus\to\Cplus\times \calC,$$
where $\Delta$ is the diagonal, and $\Delta'$ is the diagonal composed with the obvious inclusion map. Sherman then asserts as obvious that the cofiber of $\calS\Delta$ is homotopy equivalent to $|\calS\calC|$, presumably by analogy with short exact sequences of abelian groups. However, there are subtleties. It is not generally true that $\cofib(\Delta)\simeq X$ for a diagonal map of some space $X$ --  e.g. consider $\Delta\colon S^1\rightarrow S^1\times S^1$. 
A potential remedy would be to first prove $\cofib(\calS\Delta)$ and $\calS\calC$ are equivalent as $\mathbb{E}_\infty$-spaces before recovering the desired result, but this is rather involved. Alternatively, one can proceed directly by unpacking definitions, as in our proof of Claim~\ref{claim:Sherhtpyfib}.

\end{discussion}

\subsection{Double Exact Squares}\label{sec:genK1}
Given an exact category, Nenashev \cite{Nen0,NenGen} shows that its $K_1$ is generated by so-called {\em double short exact sequences} -- sharpening Sherman's original result. We extend his analysis to the pCGW setting. 
\begin{definition}[Double Exact Squares]\label{def:DES} A \emph{double exact square} is a pair of exact squares with identical nodes, but possibly different morphisms:
		\[l:=\left(\, \dsquaref{O}{C}{A}{B}{ }{ }{f_1}{g_1} \quad,\quad  \dsquaref{O}{C}{A}{B}{ }{ }{f_2}{g_2} \,\right).\]	
Since this defines an edge $(A,A)\to (B,B)$, any double exact square defines a loop 
\begin{equation}
\begin{tikzcd}
(A,A) \ar[rr,"l"] && (B,B)\\
& (O,O) \ar[ul,"e(A)"] \ar[ur,swap,"e(B)"]
\end{tikzcd}
\end{equation}
where $e(A)$ and $e(B)$ are the obvious edges from the base-point, e.g.	\[e(A):=\left(\, \dsquaref{O}{A}{O}{A}{ }{ }{}{1_A} \quad,\quad  \dsquaref{O}{A}{O}{A}{ }{ }{}{1_A} \,\right).\]	
We call this the \emph{canonical loop of $l$}, and denote it as $\mu(l)$. We denote $\lrangles{l}$ to be its homotopy class in $K_1(\calC)$.
\end{definition}

As the following example illustrates, double exact squares generalise automorphisms in $\calC$. 

\begin{example}[Automorphisms]\label{ex:Aut} If $(A,\alpha)\in \Aut(\calC)$ is an automorphism, we write
	$$l(\alpha)=\left(\, \dsquaref{O}{O}{A}{A}{ }{ }{1_A}{} \quad,\quad  \dsquaref{O}{O}{A}{A}{ }{ }{\alpha}{}  \,\right).$$
\end{example}

To prove that $K_1(\calC)$ is generated by double exact squares, it suffices to show that any Sherman Loop is homotopic to the canonical loop of a double exact square; the rest follows from Theorem~\ref{thm:Sherman}. We first 
set some conventions (which the reader may safely skim on first reading), before stating our main theorem.

\begin{convention}[Coordinate-wise Definition of 1-simplices]\label{conv:coord} Consider two exact squares
	\[h_0:= \left(\dsquaref{O}{M''}{M'}{M}{}{}{m_0}{m_1}\right) \qquad h_1:= \left(\dsquaref{O}{N''}{N'}{N}{}{}{n_0}{n_1}\right).\]	
	\begin{itemize}
		\item If $M''=N''$, we usually denote the corresponding 1-simplex as $(h_0,h_1)\colon (M',N')\to (M,N)$. 
		\item If $M''=O=N''$, we denote the corresponding 1-simplex as $(m_0,n_0)\colon (M',N')\to (M,N)$.
		\item If we wish to take their direct sum (see Lemma~\ref{lem:DirectSum}), this will be denoted
		\[h_0\oplus h_1:= \left(\dsquaref{O}{M''\oplus N''}{M'\oplus N'}{M\oplus N}{}{}{m_0\oplus n_0}{m_1\oplus n_1}\right).\]	
		\item We can also ``add'' an object $C$ to $h_0$ in the obvious way (see Lemma~\ref{lem:ADTsquares}) as follows:
		\[h_0\oplus C:=  \left(\dsquaref{O}{M''\oplus C}{M'}{M\oplus C}{}{}{m_0\oplus C}{m_1\oplus 1}\right).\]	
	\end{itemize}
\end{convention}

\begin{theorem}\label{thm:NenashevDSES} Given any $x\in K_1(\calC)$, there exists a double exact square $l$ such that $x=\mu(l)$.
\end{theorem}
\begin{proof} By Theorem~\ref{thm:Sherman}, we may assume $x$ is a Sherman Loop $G(\alpha,\beta,\theta)$ arising from a pair of exact squares
\begin{equation}
\dsquaref{O}{C}{A}{B}{ }{ }{\alpha}{\delta} \quad,\quad  \dsquaref{O}{C'}{A'}{B'}{ }{ }{\alpha'}{\delta'}
\end{equation}
and an isomorphism $\theta\colon A \oplus C \oplus B' \xrtail{\cong} A'\oplus C'\oplus B$. To turn this into a double exact square, first construct the following distinguished squares
\begin{equation}\label{eq:s0s1}
s_0:=\left(\dsquaref{O}{C\oplus C'}{A\oplus A'}{A\oplus C\oplus B'}{ }{ }{f_0}{g_0}\right) \quad,\quad  s_1:=\left( \dsquaref{O}{C\oplus C'}{A\oplus A'}{A'\oplus C' \oplus B}{ }{ }{f_1}{g_1}\right)
\end{equation}
\begin{equation}\label{eq:coordinateMAPS}
f_0=\begin{pmatrix}
1 & 0\\
0 & 0\\
0 & \alpha'
\end{pmatrix},\,\, f_1=\begin{pmatrix}
0& 1\\
0 & 0\\
\alpha & 0
\end{pmatrix},\,\, g_0=\begin{pmatrix}
0 & 0\\
1 & 0\\
0 & \delta'
\end{pmatrix},\,\, g_1=\begin{pmatrix}
0 & 0\\
0 & 1\\
\delta & 0
\end{pmatrix}.
\end{equation}
The matrix notation indicates, e.g. how the morphism $f_1\colon A\oplus A'\rtail A'\oplus C'\oplus B$ is constructed from $\alpha\colon A\rtail B$ and $1\colon A'\rtail A'$. To see why $s_0$ and $s_1$ are distinguished, apply Lemmas~\ref{lem:ADTsquares} and ~\ref{lem:DirectSum}.\footnote{Note: for $s_1$, one must also permute the direct summands via the obvious permutation isomorphisms.} 

We then apply the isomorphism $\theta$ to define the double exact square
\begin{equation}
l(x):=\left(\dsquaref{O}{C\oplus C'}{A\oplus A'}{A'\oplus C'\oplus B}{ }{}{ \theta\circ f_0}{\varphi(\theta)\circ g_0} \quad,\quad  \dsquaref{O}{C\oplus C'}{A\oplus A'}{A'\oplus C' \oplus B}{ }{ }{f_1}{g_1}\right).
\end{equation}

It thus remains to show that $\mu(l(x))$ is homotopic to $G(\alpha,\beta,\theta)$ in $K_1(\calC)$. In fact, since $K_1(\calC)$ is abelian (see Observation~\ref{obs:K1DirectSum}), it suffices to show that they are freely homotopic. This is accomplished by the following series of lemmas. 

\begin{convention} To ease notation, denote $P:=A\oplus C\oplus B'$ and $Q:=A'\oplus C'\oplus B$. 
\end{convention}

\begin{lemma}\label{lem:Nen1} $\mu(l(x))$ is freely homotopic to the loop
	\begin{equation}\label{eq:NenashevL1}
\begin{tikzcd} (P,Q) \ar[rr,"(\theta\text{,}1)"] && (Q,Q)\\
& (A\oplus A',A\oplus A') \ar[ul,"(s_0\text{,}s_1)"] \ar[ur,swap,"(s_1\text{,}s_1)"]
\end{tikzcd}
	\end{equation}
\end{lemma}
\begin{proof}[Proof of Lemma] Consider the diagram
\begin{equation}
\begin{tikzcd} & (P,Q) \ar[dddr,violet,"(\theta\text{,}1)", ""{name=PD}]\\
 \\
 \\
(A\oplus A',A\oplus A') \ar[uuur,violet,"(s_0\text{,}s_1)",""{name=PU}] \ar[rr,teal,bend left=10,"l(x)"] \ar[rr,violet,bend right =10, "(s_1\text{,}s_1)"] \ar[from=PU, to=PD,phantom,"(1)"] && (Q,Q)\\
\\
& (O,O) \ar[uul,xshift=-1em,teal,"e(A\oplus A')", ""{name=UL}] \ar[uur,swap,teal,"e(Q)", ""{name=UR}] \ar[from=UL, to=UR, phantom, "(2)"]
\end{tikzcd}	
\end{equation}
The statement follows from observing that Triangles (1) and (2) form 2-simplices. Triangle (2) is obvious. Triangle (1) is given by
	
\begin{tikzcd}
A\oplus A'\ar[r, >->,"f_0"] & P \ar[dr,phantom,"\square"] \ar[r, >->,"\theta"] & Q\\
&	C\oplus C' \ar[r, >->,"1"] \ar[u, {Circle[open]}->,"g_0"]& C\oplus C' \ar[u, {Circle[open]}->,swap,"\varphi(\theta)\circ g_0"] \\
&& O \ar[u, {Circle[open]}->] 
\end{tikzcd}\quad \begin{tikzcd}
A\oplus A'\ar[r, >->,"f_1"] & Q \ar[r, >->,"1"]  \ar[dr,phantom,"\square"]  & Q\\
&	C\oplus C' \ar[r, >->,"1"]\ar[u, {Circle[open]}->,"g_1"]& C\oplus C' \ar[u, {Circle[open]}->,"g_1	"] \\
&& O \ar[u, {Circle[open]}->] 
\end{tikzcd}
	
\end{proof}

\begin{lemma}\label{lem:Nen2} The loop $G(\alpha,\beta,\theta)$ is freely homotopic to the loop
	\begin{equation}\label{eq:NenashevL2}
\begin{tikzcd} (P,B\oplus B') \ar[rr,"(\theta\text{,}1)"] && (Q,B\oplus B')\\
& (A\oplus A',A\oplus A') \ar[ul,"(s_0\text{,}s)"] \ar[ur,swap,"(s_1\text{,}s)"]
\end{tikzcd}
	\end{equation}
where $s$ is the distinguished square
\begin{equation}\label{eq:s}
s:=\left(\dsquaref{O}{C\oplus C'}{A\oplus A'}{B\oplus B'}{}{}{\alpha\oplus \alpha'}{\delta\oplus \delta'}\right)\qquad .
\end{equation}
\end{lemma}
\begin{proof}[Proof of Lemma] Consider the diagram
	\begin{equation}\label{eq:NenSherLoop}
	\begin{tikzcd}
(P,B\oplus B') \ar[rr,red,"(\theta\text{,}1)"]&& (Q,B\oplus B')\\
\\
(A,A) \ar[r] \ar[uu,teal,"(a_0\text{,} b_0)"]& (A\oplus A',A\oplus A') \ar[uur,red,swap,"(s_1\text{,}s)",""{name=R}] \ar[uul,red,"(s_0\text{,}s)",""{name=L}] \ar[from=L, to=R, phantom, "(\star)"]& (A',A') \ar[l] \ar[uu,swap,teal,"(a_1\text{,}b_1)"]\\
\\
& (O,O) \ar[uul,teal,"e(A)"] \ar[uur,teal,swap,"e(A')"] \ar[uu,"e(A\oplus A')"]
	\end{tikzcd}
	\end{equation}
where the 1-simplices $(a_0,b_0)$ and $(a_1,b_1)$ are defined by the obvious data.
Notice: the outer loop of Diagram~\eqref{eq:NenSherLoop} is equivalent\footnote{A small difference: the top edge here is $(\theta,1)\colon (P,B\oplus B')\to (Q,B\oplus B')$, as opposed to $(\theta,\tau)\colon (P,B\oplus B')\to (Q,B'\oplus B)$, but the twist is already encoded in $(a_1,b_1)\colon (A',A')\to (Q,B\oplus B')$.}  to the Sherman Loop $G(\alpha,\beta,\theta)$ while the triangle $(\star)$ is Loop~\eqref{eq:NenashevL2}. 

It remains to check that all other triangles in Diagram~\eqref{eq:NenSherLoop} bound 2-simplices. The bottom two triangles are immediate. The top left triangle bounds
$$\begin{tikzcd}
A\ar[r, >->] & A\oplus A' \ar[dr,phantom,"\square"] \ar[r, >->,"f_0"] & P\\
&	A' \ar[r, >->,"C\oplus \alpha' "] \ar[u, {Circle[open]}->]& C\oplus B' \ar[u, {Circle[open]}->,swap,"A\oplus 1"] \\
&O \ar[r,>->] \ar[u,{Circle[open]->}] \ar[ur,phantom,"\square"]& C\oplus C' \ar[u, {Circle[open]}->,swap,"1\oplus \delta'"] 
\end{tikzcd}\quad \begin{tikzcd}
A\ar[r, >->] & A\oplus A' \ar[dr,phantom,"\square"] \ar[r, >->,"\alpha\oplus \alpha'"] & B\oplus B'\\
&	A' \ar[r, >->,"C\oplus \alpha'"] \ar[u, {Circle[open]}->, ]& C\oplus B' \ar[u, {Circle[open]}->,swap," \delta\oplus 1 "] \\
&O \ar[r,>->] \ar[u,{Circle[open]->}] \ar[ur,phantom,"\square"]& C\oplus C' \ar[u, {Circle[open]}->,swap,"1\oplus \delta'"] 
\end{tikzcd}$$
That this defines a 2-simplex follows from Lemmas~\ref{lem:ADTsquares} and \ref{lem:DirectSum}. The top right triangle is analogous. 
\end{proof}

\begin{lemma}\label{lem:Nen3} Denote $V:=\left(B\oplus B'\right)\star_{(A\oplus A')} Q$ and $W:=C\oplus C'\oplus C\oplus C'$. \underline{Then}, there exists distinguished squares of the form
	\begin{equation}
	t:=\left(\dsquaref{C\oplus C'}{W}{Q}{V}{}{g_1}{h_t}{j}\right) \qquad t':=\left(\dsquaref{O}{C\oplus C'}{Q}{V}{}{}{h_t}{k_t}\right)
	\end{equation}
		\begin{equation}\label{eq:u-u'}
	u:=\left(\dsquaref{C'\oplus C'}{W}{B\oplus B'}{V}{}{\delta\oplus\delta'}{h_u}{j}\right) \qquad u':=\left(\dsquaref{O}{C\oplus C'}{B\oplus B'}{V}{}{}{h_u}{k_u}\right).
	\end{equation}
\end{lemma}
\begin{proof} Applying Axiom (DS), Definition~\ref{def:pCGW} to the diagram 
	\[\begin{tikzcd}
	C\oplus C' \ar[d, {Circle[open]}->,"\delta\oplus \delta'"] \ar[dr, phantom, "\square"]& O \ar[dr, phantom, "\square"]\ar[d, {Circle[open]}->]  \ar[l, >->] \ar[r, >->] & C\oplus C' \ar[d, {Circle[open]}->,"g_1"] \\
	B\oplus B' & A\oplus A' \ar[l, >->,swap,"\alpha\oplus \alpha'"] \ar[r, >->,"f_1"] & Q 
	\end{tikzcd},\]
conclude that there exists a distinguished square of the form $t$. 
Since $t$ is distinguished, deduce that $\frac{V}{Q}\cong C\oplus C'$ (Remark~\ref{rem:INV}), and thus there exists a distinguished square of the form $t'$. The same argument shows there exists distinguished squares of the form $u$ and $u'$.
\end{proof}

\begin{lemma}\label{lem:Nen4} Loops~\eqref{eq:NenashevL1} and \eqref{eq:NenashevL2} are homotopic.
\end{lemma}
\begin{proof}[Proof of Lemma] Let $V:=\left(B\oplus B'\right)\star_{(A\oplus A')} Q$ and $t$ as in Lemma~\ref{lem:Nen3}, and $s_0,s_1$ as in Equation~\eqref{eq:s0s1}. Define $\s$ as the horizontal composition of $s_1$ and $t$
	\[\s:=\left(\begin{tikzcd}
O \ar[dr, phantom, "\square"]\ar[d, {Circle[open]}->]   \ar[r, >->] & C\oplus C' \ar[d, {Circle[open]}->,"g_1"] \ar[r,>->] \ar[dr, phantom, "\square"] & W \ar[d, {Circle[open]}->,"j"] \\
 A\oplus A' \ar[r, >->,"f_1"] & Q  \ar[r,>->,"h_t"] & V
\end{tikzcd}\right).\]	
Now let $s,u$ as in Equations~\eqref{eq:s} and~\eqref{eq:u-u'}. By construction of $V$, the $\M$-morphism $h_t\circ f_1$ equals 
$$h_u\circ (\alpha\oplus\alpha')\colon A\oplus A' \rtail B\oplus B'\rtail V.$$
Hence, $\tilde{s}$ also equals the horizontal composition of $s$ and $u$.  This will allow us to construct a loop $L$ that Loops~\eqref{eq:NenashevL1} and \eqref{eq:NenashevL2} are both homotopic to. 

Consider the diagram
\begin{equation}\label{eq:NenBigL1}
\!\!\!\!\!\!\!\!\!\!\!\!\!\!\!\!\!\!\!\!\!\begin{tikzcd}
(P\oplus C\oplus C',V) \ar[rr,"(\theta\oplus 1\text{,}1)" {xshift=-30pt},blue]&& (Q\oplus C\oplus C',V)\\
\\
(P,Q) \ar[rr,swap,"(\theta\text{,}1)",""{name=X},violet] \ar[uu,"(1\oplus C\oplus C'\text{,} t')",""{name=Z}] \ar[uurr,"(\theta\oplus C\oplus C'\text{,}\,t')" {yshift=10pt,xshift=50pt},""{name=Y}]  \ar[from=X, to=Y,phantom, "(4)"] \ar[from=Z, to=Y,phantom, "(3)"] && (Q,Q) \ar[uu,swap,"(1\oplus C\oplus  C'\text{,}t')"]\\
\\
& (A\oplus A',A\oplus A') \ar[uul,swap,"(s_0\text{,}s_1)",""{name=X2},violet] \ar[uur,"(s_1\text{,}s_1)",""{name=Y2},violet] 
\ar[uuuul,bend left=120, xshift=-0.5em,""{name=X1},"(s_0\oplus C\oplus C'\text{,}\,\,\s)",blue]  \ar[uuuur,bend right=120, xshift=0.5em,""{name=Y1},swap,"(s_1\oplus C\oplus C'\text{,}\s)",blue] 
\ar[from=X1, to=X2,phantom, "(1)"] 
\ar[from=Y1, to=Y2,phantom, "(2)"]
\end{tikzcd}\qquad
\end{equation}
The purple edges are Loop~\eqref{eq:NenashevL1}, the blue edges form an outer loop, which we denote $L$. To show that the two loops are homotopic, it suffices to check that Triangles (1) - (4) are boundaries of 2-simplices -- this is worked out explicitly in Claim~\ref{claim:NenL1}. Analogously, one can construct the diagram
\begin{equation}\label{eq:NenBigL2}
\!\!\!\!\!\!\!\!\!\!\!\!\!\begin{tikzcd}
(P\oplus C\oplus  C',V) \ar[rr,"(\theta\oplus 1\text{,}1)" {xshift=-30pt},blue]&& (Q\oplus C\oplus C',V)\\
\\
(P,B\oplus B') \ar[rr,swap,"(\theta\text{,}1)",""{name=X},red] \ar[uu,"(1\oplus C\oplus C'\text{,} u')",""{name=Z}] \ar[uurr,"(\theta\oplus C\oplus  C'\text{,}\,u')" {yshift=10pt,xshift=50pt},""{name=Y}]  \ar[from=X, to=Y,phantom, "(4')"] \ar[from=Z, to=Y,phantom, "(3')"] && (Q,B\oplus B') \ar[uu,swap,"(1\oplus C\oplus C'\text{,} u')"]\\
\\
& (A\oplus A',A\oplus A') \ar[uul,swap,"(s_0\text{,}s)",""{name=X2},red] \ar[uur,"(s_1\text{,}s)",""{name=Y2},red] 
\ar[uuuul,bend left=120, xshift=-0.7em,""{name=X1},"(s_0\oplus C\oplus C'\text{,}\,\,\s)",blue]  \ar[uuuur,bend right=120, xshift=0.7em,""{name=Y1},swap,"(s_1\oplus C\oplus C'\text{,}\s)",blue] 
\ar[from=X1, to=X2,phantom, "(1')"] 
\ar[from=Y1, to=Y2,phantom, "(2')"]
\end{tikzcd}
\end{equation}
where the red edges are Loop~\eqref{eq:NenashevL2} and the blue edges are loop $L$. A similar check shows Triangles (1') - (4') are also boundaries of 2-simplices (details in Claim~\ref{claim:NenL2}). Conclude that Loop~\eqref{eq:NenashevL1} and  Loop~\eqref{eq:NenashevL2} are both homotopic to loop $L$, and thus homotopic to each other as well. 
\end{proof}

\subsubsection*{Proof of Theorem~\ref{thm:NenashevDSES}} Given any Sherman Loop $G(\alpha,\beta,\theta)\in K_1(\calC)$, we construct another loop $\mu(l(x))$ where $l(x)$ is a double exact square. Reviewing our work:
\begin{itemize}
	\item Lemma~\ref{lem:Nen1} shows $\mu(l(x))$ is freely homotopic to Loop~\eqref{eq:NenashevL1}.
	\item Lemma~\ref{lem:Nen2} shows $G(\alpha,\beta,\theta)$ is freely homotopic to Loop~\eqref{eq:NenashevL2}. \item Lemma~\ref{lem:Nen4} shows  Loops~\eqref{eq:NenashevL1} and ~\eqref{eq:NenashevL2} are homotopic.
\end{itemize}
Since $K_1(\calC)$ is an abelian group, deduce that $G(\alpha,\beta,\theta)$ is homotopic to $\mu (l(x))$. Since $K_1(\calC)$ is generated by Sherman Loops (Theorem~\ref{thm:Sherman}), conclude that $K_1(\calC)$ is generated by double exact squares.
\end{proof}

\begin{discussion}\label{dis:NenPushout} Although we were guided by Nenashev's proof in \cite{NenGen}, a direct translation to our setting does not work. The key issue is the isomorphism 
	$$B  \oplus \frac{B}{A} \cong  B\star_A B,$$
which holds in exact categories and is central to \cite[Lemma 2.6]{NenGen}, but (as Inna Zakharevich pointed out to the author) fails in $\Var_k$ -- see Figure~\ref{fig:Inna}. To get around this, we construct the homotopies explicitly, verifying by hand that the chosen diagrams define valid 2-simplices. The methodological upshot: while pushouts in exact categories are arguably better behaved, restricted pushouts still retain enough formal properties to make the analysis go through (cf. Lemmas~\ref{lem:Nen3} and \ref{lem:DirectSum}). 
\end{discussion}
\vspace{-1.5em}
\begin{figure}[H]
	\centering
	\includegraphics[scale=0.5]{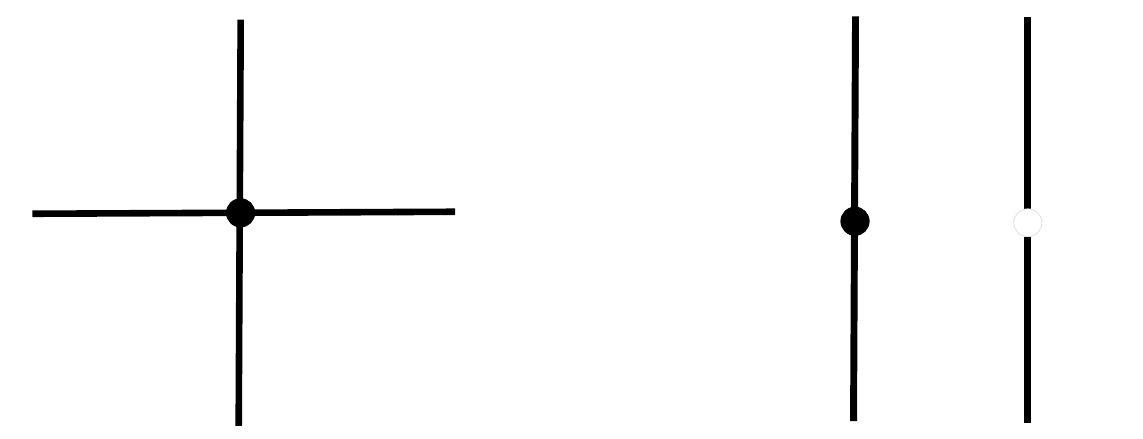}
	\caption{$\mathbb{A}^1\star_{\{\ast\}}\mathbb{A}^1$ is not isomorphic to $\mathbb{A}^1\coprod \left(\mathbb{A}^1\setminus \{\ast\} \right)$.
	}\label{fig:Inna}
\end{figure}

\section{Relations of $K_1(\calC)$}\label{sec:Nenashev}

Having characterised the generators of $K_1(\calC)$ for pCGW categories, we now work to determine its relations. We first give a baseline characterisation in Proposition~\ref{prop:baseline}. We then sharpen our understanding by comparing this to other descriptions of $K_1$ by Nenashev \cite{Nen0,Nen1} (for exact categories) and Zakharevich \cite{ZakhK1} (for Assemblers). A guiding observation is Warning~\ref{warning:compose}, which highlights a technical subtlety regarding the composition of 1-simplices in $K_1$. This brings into focus an apparent discrepancy between our account and Zakharevich's, which we investigate.

\subsection{A Baseline Argument} Recall from Convention~\ref{conv:K1} that $K_1(\calC)=\pi_1|G\calC^{\mathsf{o}}|$. Hence, applying the standard presentation of the fundamental group of a connected simplicial space, we get the following:

\begin{proposition}\label{prop:baseline} $K_1(\calC)$ is generated by symbols $\lrangles{f}$, for isomorphism classes of double exact squares $f$ in $G\calC$, modulo the relations:
\begin{itemize}
	\item[(B1)] Given any $A\in\calC$, the standard edge $e(A)\colon (O,O)\to (A,A)$ of $G\calC$ vanishes. That is,
	\[\left\langle\left(\dsquaref{O}{A}{O}{A}{}{}{}{1}\,,\,\dsquaref{O}{A}{O}{A}{}{}{}{1}\right)\right\rangle=0.\]
	\item[(B2)]  Given any $A\in\calC$, the degenerate 1-simplex $\id_A\colon (A,A)\to (A,A)$ of $G\calC$ vanishes. That is, 
	\[\left\langle\left(\dsquaref{O}{O}{A}{A}{}{}{1}{}\,,\,\dsquaref{O}{O}{A}{A}{}{}{1}{}\right)\right\rangle=0.\]
	\item[(B3)] Given double exact squares of the form
	\begin{equation}\label{eq:baseline1SIMPa}
	\begin{small}
	l_A:=\left(\dsquaref{O}{X}{A}{B}{}{}{f_0}{g_0} \,,\, \dsquaref{O}{X}{A}{B}{}{}{f'_0}{g'_0} \right) \quad 	l_B:=\left(\dsquaref{O}{Y}{B}{C}{}{}{f_1}{g_1} \,,\, \dsquaref{O}{Y}{B}{C}{}{}{f'_1}{g'_1} \right)
	\end{small}
	\end{equation}
\begin{equation}\label{eq:baseline1SIMPb}
\begin{small}
l_C:=\left(\dsquaref{O}{Z}{A}{C}{}{}{f_2}{g_2} \,,\, \dsquaref{O}{Z}{A}{C}{}{}{f'_2}{g'_2} \right) 
\end{small}
\end{equation}
that assemble into a 2-simplex in $G\calC$
	\begin{equation}\label{eq:baseline2SIMP}
\begin{tikzcd}
A\ar[r, >->," f_0"] & B  \ar[dr,phantom,"\square"] \ar[r, >->,"f_1"] & C\\
& X \ar[dr,phantom,"\square"] \ar[dr,phantom,"\square"]\ar[r, >->,swap,"h_1"] \ar[u, {Circle[open]}->,"g_0"]& Z \ar[u, {Circle[open]}->, swap, "g_2"] \\
& O \ar[r,>->] \ar[u, {Circle[open]}->]& Y\ar[u, {Circle[open]}->,swap,"h_2"] 
\end{tikzcd}\quad \begin{tikzcd}
A\ar[r, >->," f'_0"] & B  \ar[dr,phantom,"\square"] \ar[r, >->,"f'_1"] & C\\
& X \ar[dr,phantom,"\square"] \ar[dr,phantom,"\square"]\ar[r, >->,swap,"h_1"] \ar[u, {Circle[open]}->,"g'_0"]& Z \ar[u, {Circle[open]}->, swap, "g'_2"] \\
& O \ar[r,>->] \ar[u, {Circle[open]}->]& Y\ar[u, {Circle[open]}->,swap,"h_2"] 
\end{tikzcd}
\end{equation}
we have that 
$$ \lrangles{l_A} + \lrangles{l_B} = \lrangles{l_C}.$$
\end{itemize}	
\end{proposition}	
\begin{proof} Let $\iota \colon X\hookrightarrow G\calC^{\mathsf{o}}$ be the simplicial subset generated by vertices $(A,A)$ and double exact squares. By Theorem~\ref{thm:NenashevDSES}, the inclusion map induces an isomorphism $\iota_*\colon \pi_1|X|\to \pi_1|G\calC^{\mathsf{o}}|$, so it suffices to describe $\pi_1|X|$.\footnote{{\em Details.} Surjectivity follows directly from Theorem~\ref{thm:NenashevDSES}. For injectivity, any 2-simplex relation in $G\calC^{\mathsf{o}}$ can be refined, up to homotopy, to lie in $X$: each of its edges represents a canonical loop in $\pi_1|G\calC^{\mathsf{o}}|$ (up to choice of maximal spanning tree -- see Remark~\ref{rem:piGC}), which by Theorem~\ref{thm:NenashevDSES} is homotopic the canonical loop of some double exact square.}
Define $\Gamma$ as the set of 1-simplices $\{(O,O)\to (A,A)\}$; this is clearly a maximal tree spanning the 1-skeleton of $X$. We therefore obtain the standard presentation:
	$$\pi_1|X|:={{\bigslant{\pi_0(X[1])}{\stackanchor{$\langle t\rangle $ = 0 if $t\in\Gamma$, and $\lrangles{\id_A}=0$ for any degenerate 1-simplex}{$d_1(x)=d_2(x) + d_0(x),\quad \forall x\in \pi_0(X[2])$}}}}\qquad .$$
$\pi_0 (X[1])$ corresponds to the isomorphism classes of double exact squares. Further, any $x\in \pi_0 (X[2])$ can be represented as Diagram~\ref{eq:baseline2SIMP}, where $d_0(x)=l_B$, $d_2(x)=l_A$ and $d_1=l_C$. Putting everything together yields the stated presentation of $K_1$ in the proposition.
\end{proof}

\begin{remark}\label{rem:piGC} Was appealing to Theorem~\ref{thm:NenashevDSES} necessary? In principle, one could apply the above argument to $GC^{\mathsf{o}}$ directly, but there are complications. Notice the obvious choice set of 1-simplices $$\{(O,O)\to (A,A')\}$$ 
only connects to vertices where $A\cong A'$, and so fails to form a maximal tree in general. Of course, this can always be upgraded to some maximal tree $T$ by Zorn's Lemma, but we would no longer have an explicit description of $T$ -- and thus no explicit description of the relations for $\pi_1|G\calC^{\mathsf{o}}|$ either.
\end{remark}

\begin{warning}[Composition of 1-simplices]\label{warning:compose} Proposition~\ref{prop:baseline} does {\em not} assert that any three double exact squares $l_A,l_B,l_C$ give the identity
	$$ \lrangles{l_A} + \lrangles{l_B} = \lrangles{l_C},\qquad \text{whenever\,\, $f_0\circ f_1 = f_2$ \,\,and\,\, $f'_0\circ f'_1 = f'_2$.}$$
Why? While composition in $\calS\calC$ yields a pair of flag diagrams, notice these define a 2-simplex in $G\calC$ just in case the quotient index triangles in Equation~\eqref{eq:baseline2SIMP} 
\[\left(\dsquaref{O}{Y}{X}{Z}{ }{ }{h_1}{h_2 }\quad ,\quad\dsquaref{O}{Y}{X}{Z}{ }{ }{h_1}{h_2 }\right)\]
are two copies of the same square.
\end{warning}

Keeping Warning~\ref{warning:compose} in mind will help us appreciate the work done in the subsequent sections. On one level, we extend Nenashev's work on exact categories \cite{Nen0,Nen1} to our setting, providing yet another characterisation of $K_1(\calC)$. On another level, Nenashev's presentation clarifies {\em how} composition of 1-simplices splits in $K_1$, illuminating the difference between our approach and Zakharevich's.

\subsection{Admissible Triples} 
 A {\em triangle contour} $\mathcal{T}$ in $G\calC$ 
\begin{equation}\label{eq:tricontour}
\begin{tikzcd}
&(P_1,P'_1) \ar[dr,"e_1"]\\
(P_0,P'_0) \ar[ur,"e_0"] \ar[rr,"e_2"] && (P_2,P'_2)
\end{tikzcd} 
\end{equation}
is given by three pairs of distinguished squares of the form 
\begin{equation*}
e_0:=\small{\left(\dsquaref{O}{P_{1/0}}{P_0}{P_1}{}{}{\alpha_{0,1}}{\alpha_{1/0,1}} \quad,\quad \dsquaref{O}{P_{1/0}}{P'_0}{P'_1}{}{}{\alpha'_{0,1}}{\alpha'_{1/0,1}}\right)}\quad e_1:=\small{\left(\dsquaref{O}{P_{2/1}}{P_1}{P_2}{}{}{\alpha_{1,2}}{\alpha_{2/1,2}} \quad,\quad \dsquaref{O}{P_{2/1}}{P'_1}{P'_2}{}{}{\alpha'_{1,2}}{\alpha'_{2/1,2}}\right)} 
\end{equation*}

\begin{equation*}
e_2:=\small{\left(\dsquaref{O}{P_{2/0}}{P_0}{P_2}{}{}{\alpha_{0,2}}{\alpha_{2/0,2}} \quad,\quad \dsquaref{O}{P_{2/0}}{P'_0}{P'_2}{}{}{\alpha'_{0,2}}{\alpha'_{2/0,2}}\right)}.
\end{equation*}

\begin{definition}[Admissible Triple]\label{def:AdmTrip} We call a triple $\calT=(e_0,e_1,e_2)$ of the above form \emph{admissible} if $$\alpha_{1,2}\circ \alpha_{0,1}=\alpha_{0,2}\qquad \text{and}\qquad \alpha'_{1,2}\circ \alpha'_{0,1}=\alpha'_{0,2}\,\,.$$ 
In which case, one can apply Lemma~\ref{lem:quotFilt} and complete $\calT$ to the following pair of diagrams:
	\begin{equation}\label{eq:admtripDEF}
\begin{tikzcd}
P_0\ar[r, >->," \alpha_{0,1} "] & P_1  \ar[dr,phantom,"\square"] \ar[r, >->,"\alpha_{1,2}"] & P_2\\
&	P_{1/0} \ar[dr,phantom,"\square"] \ar[dr,phantom,"\square"]\ar[r, >->,swap,"\alpha_{1/0,2/0}"] \ar[u, {Circle[open]}->,"\alpha_{1/0,1}"]& P_{2/0} \ar[u, {Circle[open]}->, swap, "\alpha_{2/0,2}"] \\
& O \ar[r,>->] \ar[u, {Circle[open]}->]& P_{2/1} \ar[u, {Circle[open]}->,swap,"\alpha_{2/1,2/0}"] 
\end{tikzcd}\quad \begin{tikzcd}
P'_0\ar[r, >->," \alpha'_{0,1} "]  & P'_1 \ar[dr,phantom,"\square"] \ar[r, >->,"\alpha'_{1,2}"] & P'_2\\
 &	P_{1/0} \ar[dr,phantom,"\square"] \ar[r, >->,swap,"\alpha'_{1/0,2/0}"] \ar[u, {Circle[open]}->,"\alpha'_{1/0,1}"]& P_{2/0} \ar[u, {Circle[open]}->, swap, "\alpha'_{2/0,2}"] \\
& O \ar[r,>->] \ar[u, {Circle[open]}->]& P_{2/1} \ar[u, {Circle[open]}->,swap,"\alpha'_{2/1,2/0}"] 
\end{tikzcd}
\end{equation}
 In particular, define 
\begin{equation}\label{eq:admiAssocDS}
l(\calT):=\left(\dsquaref{O}{P_{2/1}}{P_{1/0}}{P_{2/0}}{}{}{\alpha_{1/0,2/0}}{\alpha_{2/1,2/0}} \quad \, \quad \dsquaref{O}{P_{2/1}}{P_{1/0}}{P_{2/0}}{}{}{\alpha'_{1/0,2/0}}{\alpha'_{2/1,2/0}}\right)
\end{equation}
to be the \emph{double exact square associated to admissible triple $\calT$}. Following Definition~\ref{def:DES}, denote $\mu (l(\calT))$ to be the canonical loop of $l(\calT)$.
\end{definition}

Any admissible triple $\calT=(e_0,e_1,e_{2})$ defines a loop $e_0e_1e^{-1}_2$, which we also denote using $\calT$. Notice: if $$\alpha_{1/0,2/0}=\alpha'_{1/0,2/0}\qquad 
\text{and}\qquad \alpha_{2/1,2/0}=\alpha'_{2/1,2/0}$$ 
in Diagram~\eqref{eq:admtripDEF}, then the loop $\calT$ bounds a 2-simplex in $G\calC$. However, even when the condition does not hold, we can still say something meaningful about the (free) homotopy class of $\calT$:

\begin{lemma}\label{lem:admtriple} Let $\calT$ be an admissible triple, with $l(\calT)$ its associated double exact square. If the 1-simplices of $\calT$ are all double exact squares, then the loop $\calT$ is freely homotopic to $\mu(l(\calT))$.
\end{lemma}

This lemma extends \cite[Corollary 4.3]{Nen0}. The proof is more involved than Nenashev's original argument (for similar reasons as in Discussion~\ref{dis:NenPushout}), and will be deferred to Section~\ref{sec:admtriples}.

\subsection{Nenashev Relations} To define Nenashev's relations on $K_1(\calC)$, we shall need the following generalisation of double exact squares. A \emph{$3 \times 3$ diagram} in a pCGW category $\calC$ is a pair of diagrams
\begin{equation}\label{eq:txDiagram}
\left(\begin{tikzcd}
X_{00}\ar[r, >->, "f_{0}"] 
\ar[d,>->,swap,"h_{0}"] & X_{01} \ar[d, >->,"h_{1}"] & X_{02}  
\ar[d,>->, "h_{2}"] \ar[l, {Circle[open]}->,swap,"g_{0}"]\\
X_{10}  \ar[r,>->,"f_{1}"]  &	X_{11}  \ar[dr,phantom,"\circlearrowleft"] & X_{12} \ar[l, {Circle[open]}->,swap,"g_1"] \\
X_{20} \ar[u, {Circle[open]}->,"j_{0}"] \ar[r,>->,"f_2"]& X_{21} \ar[u, {Circle[open]}->,"j_1"]& X_{22}  \ar[u, {Circle[open]}->,swap,"j_2"] \ar[l, {Circle[open]}->,swap,"g_{2}"]
\end{tikzcd} \qquad\text{,}\qquad  \begin{tikzcd}
X_{00}\ar[r, >->, "f'_{0}"] 
\ar[d,>->,swap,"h'_{0}"] & X_{01} \ar[d, >->,"h'_{1}"] & X_{02}  
\ar[d,>->, "h'_{2}"] \ar[l, {Circle[open]}->,swap,"g'_{0}"]\\
X_{10}  \ar[r,>->,"f'_{1}"]  &	X_{11}  \ar[dr,phantom,"\circlearrowleft"] & X_{12} \ar[l, {Circle[open]}->,swap,"g'_1"] \\
X_{20} \ar[u, {Circle[open]}->,"j'_{0}"] \ar[r,>->,"f'_2"]& X_{21} \ar[u, {Circle[open]}->,"j'_1"]& X_{22}  \ar[u, {Circle[open]}->,swap,"j'_2"] \ar[l, {Circle[open]}->,swap,"g'_{2}"]
\end{tikzcd}\right)
\end{equation}
on the same objects subject to the following conditions:
\begin{itemize} 
	\item The horizontal and vertical rows of each diagram define exact squares. Explicitly, a $\tx$ diagram is defined by 6 double exact squares
	\begin{equation*}
l_{i}:= \left(	\dsquaref{O}{X_{i2}}{X_{i0}}{X_{i1}}{ }{ }{ f_{i}}{g_{i} } \, \,,\,\,	\dsquaref{O}{X_{i2}}{X_{i0}}{X_{i1}}{ }{ }{ f'_{i}}{g'_{i} }\right) \qquad	l^{i}:= \left(\dsquaref{O}{X_{2i}}{X_{0i}}{X_{1i}}{ }{}{h_{i} }{ j_{i}} \, \,,\,\, \dsquaref{O}{X_{2i}}{X_{0i}}{X_{1i}}{ }{}{h'_{i} }{ j'_{i}} \right),
\end{equation*}
for all $i\in \{0,1,2\}$.
	\item The bottom right $\E$-square commutes. No conditions are imposed on the other squares. 
\end{itemize}

These $\tx$ diagrams were originally defined as commutative diagrams in exact categories. In this setting, one can form the pushout $Z:=X_{01}\star_{X_{00}} X_{10}$ and check that the induced map $v\colon Z\rtail X_{11}$ is an admissible mono -- this plays a key role in the proof of \cite[Prop. 5.1]{Nen0}. However, this argument breaks down in general pCGW categories: restricted pushouts may fail to yield an $\M$-morphism $v\colon Z\rtail X_{11}$ unless the original $\M$-square in Diagram~\eqref{eq:txDiagram} is {\em optimal} (in the sense of Definition~\ref{def:restPO}). We circumvent this problem by hardcoding the desired pushout behaviour into the following definition.

\begin{definition}[Optimal $\tx$ Diagrams]\label{def:OPTtx} Call a $\tx$ diagram \emph{optimal} if there exists an object $Z$ along with three pairs of $\M$-morphisms 
	$$v,v'\colon Z\rtail X_{11},\quad u,u'\colon X_{01}\rtail Z \quad \text{and}\quad w,w'\colon X_{10}\rtail Z.$$
These are required to satisfy the identities
\begin{align}\label{eq:OPTID}
&u\circ f_0=w\circ h_0  &  v\circ u = h_1  &&  v\circ w = f_1 \\
& u'\circ f'_0=w'\circ h'_0 & v'\circ u' = h'_1 && v'\circ w '= f'_1\,\,, \nonumber 
\end{align}
and assemble into the following diagrams 
\begin{enumerate}[label=(P\arabic*)]
	\item \[
\begin{tikzcd}
X_{01}\ar[r, >->,"u"] \ar[dr,phantom,"\square"] & Z \ar[dr,phantom,"\square"] \ar[r, >->,"v"] & X_{11} \\
O \ar[r,>->] \ar[u,{Circle[open]}->]&	X_{20} \ar[r, >->,"f_2"] \ar[u, {Circle[open]}->] \ar[dr,phantom,"\square"]  & X_{21}\ar[u, {Circle[open]}->,"j_1"] \\
& O \ar[r,>->] \ar[u,{Circle[open]}->]& X_{22} \ar[u, {Circle[open]}->,"g_2"] 
\end{tikzcd}	 \quad 	\begin{tikzcd}
	X_{01}\ar[r, >->,"u'"] \ar[dr,phantom,"\square"] \ar[dr,phantom,"\square"] & Z \ar[dr,phantom,"\square"] \ar[r, >->,"v'"] & X_{11} \\
	O \ar[r,>->] \ar[u,{Circle[open]}->]&	X_{20} \ar[dr,phantom,"\square"] \ar[r, >->,"f'_2"] \ar[u, {Circle[open]}->]& X_{21}\ar[u, {Circle[open]}->,"j'_1"] \\
	& O \ar[r,>->] \ar[u,{Circle[open]}->]& X_{22} \ar[u, {Circle[open]}->,"g_2'"] 
	\end{tikzcd}	
	\]
		\item \[
	\begin{tikzcd}
	X_{10}\ar[r, >->,"w"]\ar[dr,phantom,"\square"]  & Z \ar[dr,phantom,"\square"] \ar[r, >->,"v"] & X_{11} \\
	O \ar[r,>->] \ar[u,{Circle[open]}->]&	X_{02} \ar[dr,phantom,"\square"]  \ar[r, >->,"h_2"] \ar[u, {Circle[open]}->]& X_{12}\ar[u, {Circle[open]}->,"g_1"] \\
	&O \ar[r,>->] \ar[u,{Circle[open]}->]& X_{22} \ar[u, {Circle[open]}->,"j_2"] 
	\end{tikzcd}	 \quad 	\begin{tikzcd}
	X_{10}\ar[r, >->,"w'"] \ar[dr,phantom,"\square"] & Z \ar[dr,phantom,"\square"] \ar[r, >->,"v'"] & X_{11} \\
O \ar[r,>->] \ar[u,{Circle[open]}->]	&	X_{02} \ar[dr,phantom,"\square"]  \ar[r, >->,"h'_2"] \ar[u, {Circle[open]}->]& X_{12}\ar[u, {Circle[open]}->,"g'_1"] \\
	&O \ar[r,>->] \ar[u,{Circle[open]}->]& X_{22} \ar[u, {Circle[open]}->,"j'_2"] 
	\end{tikzcd}	
	\]
	\item \[
	\begin{tikzcd}
	X_{00}\ar[r, >->,"f_0"] & X_{01} \ar[dr,phantom,"\square"] \ar[r, >->,"u"] & Z \\
	&	X_{02} \ar[r, >->] \ar[u, {Circle[open]}->,"g_0"]& X_{02}\oplus X_{20}\ar[u, {Circle[open]}->] \\
	&& X_{20} \ar[u, {Circle[open]}->] 
	\end{tikzcd}	 \quad 		\begin{tikzcd}
	X_{00}\ar[r, >->,"f'_0"] & X_{01} \ar[dr,phantom,"\square"] \ar[r, >->,"u'"] & Z \\
	&	X_{02} \ar[r, >->] \ar[u, {Circle[open]}->,"g'_0"]& X_{02}\oplus X_{20}\ar[u, {Circle[open]}->] \\
	&& X_{20} \ar[u, {Circle[open]}->] 
	\end{tikzcd}	
	\]
	
		\item \[
	\begin{tikzcd}
	X_{00}\ar[r, >->,"h_0"] & X_{10} \ar[dr,phantom,"\square"] \ar[r, >->,"w"] & Z \\
	&	X_{20} \ar[r, >->] \ar[u, {Circle[open]}->,"j_0"]& X_{02}\oplus X_{20}\ar[u, {Circle[open]}->] \\
	&& X_{02} \ar[u, {Circle[open]}->,] 
	\end{tikzcd}	 \quad 		\begin{tikzcd}
	X_{00}\ar[r, >->,"h'_0"] & X_{10} \ar[dr,phantom,"\square"] \ar[r, >->,"w'"] & Z \\
	&	X_{20} \ar[r, >->] \ar[u, {Circle[open]}->,"j'_0"]& X_{02}\oplus X_{20}\ar[u, {Circle[open]}->] \\
	&& X_{02} \ar[u, {Circle[open]}->] 
	\end{tikzcd}	
	\]
\end{enumerate}	
\end{definition}

\begin{remark} A sanity check: notice the Optimality Identities~\eqref{eq:OPTID} imply that the $\M$-squares in Diagram~\eqref{eq:txDiagram} commute, e.g. $h_1\circ f_0=f_1\circ h_0$. 
\end{remark}

\begin{definition}\label{def:NenDC} Define $\calD(\calC)$ to be the abelian group with generators $\lrangles{l}$ for all double exact squares $l$ in $\calC$ subject to the following relations. 
	\begin{enumerate}[label=(N\arabic*)]
	\item $\lrangles{l}=0$ if 
\begin{equation}\label{eq:NenR1}
l=\left(\dsquaref{O}{C}{A}{B}{}{}{f}{g} \quad ,\quad \dsquaref{O}{C}{A}{B}{}{}{f}{g}\right)
\end{equation}
Any identical pair of exact squares will be called a \emph{diagonal} square.
	\item Given an optimal $\tx$ diagram
	\begin{equation}\label{eq:3x3}
\left(\begin{tikzcd}
X_{00}\ar[r, >->, "f_{0}"] 
\ar[d,>->,swap,"h_{0}"]  & X_{01} \ar[d, >->,"h_{1}"] & X_{02}  
\ar[d,>->, "h_{2}"] \ar[l, {Circle[open]}->,swap,"g_{0}"]\\
X_{10}  \ar[r,>->,"f_{1}"]  &	X_{11}  \ar[dr,phantom,"\circlearrowleft"] & X_{12} \ar[l, {Circle[open]}->,swap,"g_1"] \\
X_{20} \ar[u, {Circle[open]}->,"j_{0}"] \ar[r,>->,"f_2"]& X_{21} \ar[u, {Circle[open]}->,"j_1"]& X_{22}  \ar[u, {Circle[open]}->,swap,"j_2"] \ar[l, {Circle[open]}->,swap,"g_{2}"]
\end{tikzcd} \qquad\text{,}\qquad  \begin{tikzcd}
X_{00}\ar[r, >->, "f'_{0}"] 
\ar[d,>->,swap,"h'_{0}"] & X_{01} \ar[d, >->,"h'_{1}"] & X_{02}  
\ar[d,>->, "h'_{2}"] \ar[l, {Circle[open]}->,swap,"g'_{0}"]\\
X_{10}  \ar[r,>->,"f'_{1}"]  &	X_{11} \ar[dr,phantom,"\circlearrowleft"]  & X_{12} \ar[l, {Circle[open]}->,swap,"g'_1"] \\
X_{20} \ar[u, {Circle[open]}->,"j'_{0}"] \ar[r,>->,"f'_2"]& X_{21} \ar[u, {Circle[open]}->,"j'_1"]& X_{22}  \ar[u, {Circle[open]}->,swap,"j'_2"] \ar[l, {Circle[open]}->,swap,"g'_{2}"]
\end{tikzcd}\right)
\end{equation}
defined by the following 6 double exact squares
	\begin{equation*}
l_{i}:= \left(	\dsquaref{O}{X_{i2}}{X_{i0}}{X_{i1}}{ }{ }{ f_{i}}{g_{i} } \, \,,\,\,	\dsquaref{O}{X_{i2}}{X_{i0}}{X_{i1}}{ }{ }{ f'_{i}}{g'_{i} }\right) \qquad	l^{i}:= \left(\dsquaref{O}{X_{2i}}{X_{0i}}{X_{1i}}{ }{}{h_{i} }{ j_{i}} \, \,,\,\, \dsquaref{O}{X_{2i}}{X_{0i}}{X_{1i}}{ }{}{h'_{i} }{ j'_{i}} \right)
\end{equation*}
for all $i\in \{0,1,2\}$, the following 6-term relation holds
\begin{equation}
\lrangles{l_0} + \lrangles{l_2} - \lrangles{l_1} = \lrangles{l^0} + \lrangles{l^2} - \lrangles{l^1}. 
\end{equation}
\end{enumerate}
\end{definition}
\begin{theorem}\label{thm:NenSurj} There exists a well-defined homomorphism 
\begin{equation}\label{eq:NenSurj}
m\colon \calD(\calC)\longrightarrow K_1(\calC)
\end{equation}
that is surjective. In other words, the two relations of $\calD(\calC)$ also hold in $K_1(\calC)$.	
\end{theorem}
\begin{proof} By Theorem~\ref{thm:NenashevDSES}, we know that $K_1(\calC)$ is generated by double exact squares so both groups have the same generators. It remains to check the relations.
\begin{enumerate}[label=(N\arabic*):]
	\item Let $l$ be as in Equation~\eqref{eq:NenR1}. The corresponding loop $\mu(l)$ bounds the 2-simplex 
\[	\begin{tikzcd}
O\ar[r, >->] & A \ar[dr,phantom,"\square"] \ar[r, >->,"f"] & B \\
	&	A \ar[r, >->,"f"] \ar[u, {Circle[open]}->]& B \ar[u, {Circle[open]}->,""] \\
	&& C \ar[u, {Circle[open]}->,"g"] 
	\end{tikzcd}	 \quad 	\begin{tikzcd}
	O\ar[r, >->] & A \ar[dr,phantom,"\square"] \ar[r, >->,"f"] & B \\
	&	A \ar[r, >->,"f"] \ar[u, {Circle[open]}->]& B \ar[u, {Circle[open]}->,""] \\
	&& C \ar[u, {Circle[open]}->,"g"] 
	\end{tikzcd}
	\]
in $G\calC$, and so $\lrangles{l}=0$.
	\item 
Leveraging the fact that the $\tx$ diagram is optimal, construct the diagram
	\begin{equation}
\begin{tikzcd}
(X_{00}\tk X_{00}) \ar[rr,"l_0",blue] \ar[dd,"l^{0}",blue] \ar[dr,"\alpha_0"]&& (X_{01}\tk  X_{01}) \ar[dd,"l^1",blue] \ar[dl,swap,"\alpha_3"]\\
& (Z \tk Z) \ar[dr,"\alpha_2"]  \\
(X_{10}\tk X_{10}) \ar[rr,"l_1",blue] \ar[ur,"\alpha_1"] && (X_{11} \tk X_{11})
\end{tikzcd}
\end{equation}
where outer blue edges $\alpha:=l_0l^1(l_1)^{-1}(l^0)^{-1}$ form a loop, and the inner edges are given by 
\[\small{\alpha_0:= \left(	\dsquaref{O}{X_{02}\oplus X_{20}}{X_{00}}{Z}{ }{ }{u\circ f_0}{} \, ,\, 	\dsquaref{O}{X_{02}\oplus X_{20}}{X_{00}}{Z}{ }{ }{u'\circ f'_0}{}\right) \quad \alpha_1:= \left(	\dsquaref{O}{X_{02}}{X_{10}}{Z}{ }{ }{w}{} \, ,\, 	\dsquaref{O}{X_{02}}{X_{10}}{Z}{ }{ }{w'}{}\right)}  \]
\[\small{\alpha_2:= \left(	\dsquaref{O}{X_{22}}{Z}{X_{11}}{ }{ }{v}{} \, ,\, 	\dsquaref{O}{X_{22}}{Z'}{X_{11}}{ }{ }{v'}{}\right)  \quad  \alpha_3:= \left(	\dsquaref{O}{X_{20}}{X_{01}}{Z}{ }{ }{u}{} \, ,\, 	\dsquaref{O}{X_{20}}{X_{01}}{Z}{ }{ }{u'}{}\right)}  \] 
For orientation, we start with a basic observation.
 
 \begin{lemma}\label{lem:Nenclosedloop} 
 	Given any closed loop $l=e_0\dots e_n$ whose edges are all double exact squares, 
 	$$ \lrangles{l}= \sum^{n}_{i=0} (-1)^{\epsilon_i}\lrangles{e_i} $$
 	in $K_1(\calC)$, where the coefficient $(-1)^{\epsilon_i}$ reflects the orientation of edge $e_i$ in $l$. 
 \end{lemma}
\begin{proof}[Proof of Lemma] We spell this out for the case of $\alpha$; the general case is completely analogous. Consider the diagram
	\begin{equation}
	\begin{tikzcd}
	(X_{00}\tk X_{00}) \ar[rr,"l_0",blue] \ar[dd,"l^{0}",blue]&& (X_{01}\tk  X_{01}) \ar[dd,"l^1",blue] \\
	& (O \tk O) \ar[ur] \ar[dr] \ar[ul] \ar[dl]  \\
	(X_{10}\tk X_{10}) \ar[rr,"l_1",blue] && (X_{11} \tk X_{11})
	\end{tikzcd}
	\end{equation}
	featuring $\alpha$ as the outer loop but now with the base-point $(O,O)$ at the center. Tracing along the loop $\mu(l_0)\mu(l^1) \mu(l_1)^{-1}\mu(l^0)^{-1}$, each edge $(O,O)\to (X_{ij},X_{ij})$ is traversed exactly once in both directions, and so their contributions cancel. Hence, the loop is freely homotopic to the square boundary, and so
	\begin{equation*}
	\lrangles{\alpha} = \lrangles{l_0} + \lrangles{l^1} - \lrangles{l_1} - \lrangles{l^0}.
	\end{equation*}
\end{proof}
 
Next, appealing once more to the optimality of the $\tx$ diagram, notice:
\begin{itemize}

	\item Diagrams (P3) and (P4) are 2-simplices. By the Optimality Identities~\eqref{eq:OPTID}, they are bounded by the loops $l_0 \alpha_3 \alpha_{0}^{-1}$ and $l^0 \alpha_1 \alpha_{0}^{-1}$ respectively. Therefore, deduce that
\begin{align}\label{eq:3x3r1}
\lrangles{l_0} + \lrangles{\alpha_3} - \lrangles{\alpha_0}=0\\
\lrangles{l^0} + \lrangles{\alpha_1} - \lrangles{\alpha_0}=0. \nonumber
\end{align}
\item Diagrams (P1) and (P2) are admissible triples whose associated exact squares are $l_2$ and $l^2$ respectively. We may represent these admissible triples as $\alpha_{3}\alpha_2 (l^1)^{-1}$ and $\alpha_{1}\alpha_2 (l_1)^{-1}$ respectively; this follows from the Optimality Identities~\eqref{eq:OPTID} and the fact the $\E$-squares in the $\tx$-diagram commute. Thus, apply Lemma~\ref{lem:admtriple} to deduce
\begin{align}\label{eq:3x3r2}
\lrangles{\alpha_3} + \lrangles{\alpha_2} - \lrangles{l^1}=\lrangles{l_2}  \\
\lrangles{\alpha_1} + \lrangles{\alpha_2} - \lrangles{l_1}= \lrangles{l^2}. \nonumber
\end{align}
\end{itemize}
Combining Equations~\eqref{eq:3x3r1} and \eqref{eq:3x3r2}, 
\begin{align*}
\lrangles{l_2} - \lrangles{l^2} & =  \lrangles{\alpha_3} + \lrangles{\alpha_2} - \lrangles{l^1} - \lrangles{\alpha_1} - \lrangles{\alpha_2} + \lrangles{l_1} \\
&= \lrangles{\alpha_3} - \lrangles{\alpha_0} - \lrangles{l^1} - \lrangles{\alpha_1} + \lrangles{\alpha_0} + \lrangles{l_1} \\
& = -\lrangles{l_0} - \lrangles{l^1} + \lrangles{l^0} + \lrangles{l_1},
\end{align*}
and so by rearranging terms, conclude
\begin{equation}
\lrangles{l_0} + \lrangles{l_2} - \lrangles{l_1} = \lrangles{l^0} + \lrangles{l^2} - \lrangles{l^1}.
\end{equation}
\end{enumerate}	
\end{proof}

\subsection{Stratified Birational Equivalences}\label{sec:stratified} Here, we apply Theorem~\ref{thm:NenSurj} to establish the connection between the generators of $K_1(\Var_k)$ and birational maps. We begin with the case of
irreducible varieties.

\smallskip

Recall: a birational equivalence $f\colon X\dashrightarrow Y$ between irreducible varieties is a rational map defining an isomorphism $U\cong f(U)$, where $U\subseteq X$ and $f(U)\subseteq Y$ are non-empty opens. We call such an $f$ {\em stratified} if, in addition, $$X\setminus U\cong Y\setminus f(U).$$ This resembles a double exact square in $\Var_k$, 
except that exact squares allow the open subvariety to be empty. The following corollary shows this extra generality is not essential for $K_1$.

\begin{corollary}\label{cor:BirIso} Let $\alpha,\beta\colon A\to A$ be any pair of automorphisms in a pCGW category $\calC$. By Axiom (I), we may either regard them as $\M$-morphisms or $\E$-morphisms. Hence, define
	$$\widetilde{l}(\alpha):=\left(\dsquaref{O}{A}{O}{A}{ }{ }{ }{\varphi(\alpha)} \quad,\quad \dsquaref{O}{A}{O}{A}{ }{ }{ }{1_A} \right).$$
Recall the definition of $l(\alpha)$ from Example~\ref{ex:Aut}. One thus obtains the identities:
	\begin{enumerate}[label=(\roman*)]
		\item $\lrangles{l(\alpha)}=\lrangles{\widetilde{l}(\alpha^{-1})}.$
		\item $\lrangles{l(\alpha\beta)}=\lrangles{l(\alpha)}+\lrangles{l(\beta)}$.
		\item  Define
	$$l(\alpha,\beta):=\left(\, \dsquaref{O}{O}{A}{A}{ }{ }{\alpha}{} \quad,\quad  \dsquaref{O}{O}{A}{A}{ }{ }{\beta}{}  \,\right),$$
	not to be confused with $l(\alpha\beta)$, which composes $\alpha$ with $\beta$. Then $$\lrangles{l(\alpha,\beta)}=\lrangles{l(\beta\alpha^{-1})}=\lrangles{l(\beta)}-\lrangles{l(\alpha)} \quad \text{in $K_1(\calC)$}.$$

	\end{enumerate}
\end{corollary}

\begin{remark} Suppose $\alpha,\beta\colon A\to A$ are automorphisms of a variety. By items (i) and (iii) above, $l(\alpha,\beta)=\lrangles{\widetilde{l}(\alpha\beta^{-1})}$, and so $l(\alpha,\beta)\in K_1(\Var_k)$ corresponds to an honest birational automorphism. One may also interpret $\lrangles{l(\alpha,\beta)}$ as the formal difference between automorphisms $\beta$ and $\alpha$.
\end{remark}

\begin{proof}[Proof of Corollary] Recall that $\varphi(\alpha)\colon A\otail A$ defines $\alpha\colon A\to A$ in $\calC$ if $\E\subseteq \calC$, and $\alpha^{-1}\colon A\to A$ if $\E^\opp\subseteq \calC$.
	\begin{enumerate}[label=(\roman*)] 	\item Consider the following $\tx$ diagram
		 \begin{equation*} 
	\small{	\left(\begin{tikzcd}
		O\ar[r, >->] 
		\ar[d,>-> ] & O \ar[d, >->] & O
		\ar[d,>->] \ar[l, {Circle[open]}->]\\
	A  \ar[r,>->,"1_A"]  &A  \ar[dr,phantom,"\circlearrowleft"] & O \ar[l, {Circle[open]}->] \\
		A \ar[u, {Circle[open]}->,"1_A"] \ar[r,>->,"\alpha"]& A \ar[u, {Circle[open]}->,"\varphi(\alpha^{-1})"]& O  \ar[u, {Circle[open]}->] \ar[l, {Circle[open]}->]
		\end{tikzcd} \qquad\text{,}\qquad  \begin{tikzcd}
		O\ar[r, >->] 
	\ar[d,>-> ] & O \ar[d, >->] & O
\ar[d,>->] \ar[l, {Circle[open]}->]\\
A  \ar[r,>->,"\alpha"]  &A \ar[dr,phantom,"\circlearrowleft"] & O \ar[l, {Circle[open]}->] \\
A \ar[u, {Circle[open]}->,"1_A"] \ar[r,>->,"\alpha"]& A \ar[u, {Circle[open]}->,"1_A"]& O  \ar[u, {Circle[open]}->] \ar[l, {Circle[open]}->]
\end{tikzcd}\right).}
		\end{equation*}
To show optimality, pick the following pairs of $\M$-morphisms:
		$$v:=1_A,\,v':= \alpha,\qquad u,u':= O\rtail A \qquad \text{and}\quad w,w':=1_A.$$
The claimed identity then follows by Relations (N1) and (N2).
\item Immediate from Relation (B3) of Proposition~\ref{prop:baseline} and the fact that $K_1$ is abelian.
		\item The following (optimal) $\tx$ diagram
		\begin{equation*} 
		\small{	\left(\begin{tikzcd}
			A\ar[r, >->,"\alpha"] \ar[d,>->, swap,"\alpha"]  & A \ar[d, >->,"1_A"] & O
			\ar[d,>->] \ar[l, {Circle[open]}->]\\
			A  \ar[r,>->,"1_A"]  &A \ar[dr,phantom,"\circlearrowleft"] & O \ar[l, {Circle[open]}->] \\
			O \ar[u, {Circle[open]}->] \ar[r,>->]& O \ar[u, {Circle[open]}->] & O  \ar[u, {Circle[open]}->] \ar[l, {Circle[open]}->]
			\end{tikzcd} \qquad\text{,}\qquad  \begin{tikzcd}
			A\ar[r, >->,"\beta"] 
			\ar[d,>->,swap,"\alpha"] & A \ar[d, >->,"1_A"] & O
			\ar[d,>->] \ar[l, {Circle[open]}->]\\
			A  \ar[r,>->,"\beta\alpha^{-1}"]  &A \ar[dr,phantom,"\circlearrowleft"]  & O \ar[l, {Circle[open]}->] \\
			O \ar[u, {Circle[open]}->] \ar[r,>->]& O \ar[u, {Circle[open]}->] & O  \ar[u, {Circle[open]}->] \ar[l, {Circle[open]}->]
			\end{tikzcd} \right)}		
		\end{equation*}
		yields the relation $\lrangles{l(\alpha,\beta)}=\lrangles{l(\beta\alpha^{-1})}.$
 By item (ii), deduce
		 $$\lrangles{l(\beta\alpha^{-1})} = \lrangles{l(\beta)}+\lrangles{l(\alpha^{-1})}\quad\text{and}\quad \lrangles{l(
	\alpha^{-1})}=-\lrangles{l(\alpha)}.$$
This assembles to give $\lrangles{l(\alpha,\beta)}=\lrangles{l(\beta\alpha^{-1})}=\lrangles{l(\beta)}-\lrangles{l(\alpha)}$, as desired. 
	\end{enumerate}	
\end{proof}

We now extend the notion of stratified birational equivalence from
irreducible to arbitrary varieties.

\begin{definition}\label{def:strat-bir-eq}
	A {\em stratified birational equivalence} between varieties $X$ and $Y$ consists of a pair
	of open immersions of non-empty opens
$$U \hookrightarrow X, \qquad V \hookrightarrow Y,$$
 together with isomorphisms
$$U\cong V\,\,\quad\text{and}\quad X\setminus U \cong Y\setminus V \,.$$
This clearly defines a 1-simplex in $G\Var_k^\mathsf{o}$, and includes all double exact squares. As such, by Theorem~\ref{thm:NenSurj} and Corollary~\ref{cor:BirIso}, stratified birational equivalences generate $K_1(\Var_k).$
\end{definition}

\begin{discussion} The definition above agrees with the irreducible case. For reducible varieties, $U\cong V$ induces finitely many birational equivalences between the irreducible components of $X$ and $Y$ that intersect the open subsets, while the remaining components lie entirely in the closed complements (which organise into isomorphic pairs since $X\setminus U\cong Y\setminus V$). In this sense, the geometric equivalence between $X$ and $Y$ is assembled from finitely many pieces, motivating the terminology {\em stratified} birational equivalence.
\end{discussion}

\begin{discussion}[Generalising stratifiability] Any birational equivalence defines a 1-simplex in the $G$-construction, but not necessarily an element of $K_1(\Var_k)$. For this to happen, the associated 1-simplex must lie in the base-point component [since $K_1(\Var_k)=\pi_1|G\Var_k^\mathsf{o}|$]. More concretely, a birational map $f\colon X\dashrightarrow Y$ defines an element in $K_1$ just in case 
	$$[X]=[Y]\quad \text{in}\, K_0(\Var_k).$$
It is clear stratified birational equivalences of irreducible varieties satisfy this condition, but not every birational equivalence does -- e.g. the birational equivalence $\mathbb{P}^2\dashrightarrow\mathrm{Bl}_{[0:0:1]}\mathbb{P}^2$.\footnote{To be safe, assume the base field $k$ has characteristic 0, but the example should work more generally. The argument is straightforward once we know dimension is an invariant of $K_0(\Var_k)$ [and so $[\ast]\neq [\mathbb{A}^1]$], but there does not seem to be a published reference showing this for arbitrary fields. Nonetheless, see this argument \cite{SawinField} by Will Sawin.}
\end{discussion}

\subsection{Assembler Relations}\label{sec:Ass} We now apply Theorem~\ref{thm:NenSurj} to compare our relations with Zakharevich's $K_1$ of an Assembler (Proposition~\ref{prop:ass}). As a corollary, we prove that $\calD(\calC)\cong K_1(\calC)$, and so $\calD(\calC)$ gives an alternative presentation of $K_1(\calC)$ for pCGW categories (Corollary~\ref{cor:Nen}).
 
We start with an informal overview. An {\em Assembler} is a Grothendieck site $\calA$ whose topology encodes how an object $A$ may be covered by a finite set of disjoint subobjects $\{A_i\}_{i\in I}$. To define the $K$-theory of an Assembler $\calA$, one constructs its associated {\em category of covers} $\calW(\calA)$, defined as follows:
\begin{itemize}
	\item[] \textbf{Objects:} Finite sets of objects $\{A_i\}_{i\in I}$ in $\calA$;
	\item[] \textbf{Morphisms:} Piecewise automorphisms in $\calA$. Explicitly, a morphism  $f\colon \{A_i\}_{i\in I}\to \{B_j\}_{j\in J}$ in $\calW(\calA)$ is a tuple of morphisms $f_i\colon A_i\to B_{f(i)}$ such that $\{f_i\colon A_i\to B_j\}_{i\in f^{-1}(j)}$ is a finite disjoint covering family.
\end{itemize}
The $K$-theory spectrum $K(\calA)$ is then defined as the symmetric spectrum of the $\Gamma$-space induced by $\calW(\calA)$ -- see \cite[\S 2]{ZakhAss} for details. Its relevance is that it gives an alternative construction of $K\Var_k$, which is equivalent to the CGW construction \cite[Theorems 7.8 and 9.1]{CGW}. Further, extending Muro-Tonks' model of $K_1$ of a Waldhausen Category \cite{MuroTonks}, Zakharevich gave the following $K_1$ presentation.

\begin{theorem}[{{\cite[Theorem B]{ZakhK1}}}]\label{thm:ZakhB} Let $\calA$ be an Assembler whose morphisms are closed under pullback. Then $K_1(\calA)$ is generated by a pair of morphisms 
	$$A\xrightrightarrows[g]{f} B$$
	in $\calW(\calA)$. These satisfy the relations
	\begin{enumerate}[label=(Z\arabic*)]
		\item $\lrangles{A\xrightrightarrows[f]{f} B}=0$;
		\item $\lrangles{A\xrightrightarrows[f_2]{f_1} B} + \lrangles{C\xrightrightarrows[g_2]{g_1} D}=\lrangles{A\coprod C\xrightrightarrows[f_2\coprod g_2]{f_1\coprod g_1} B\coprod D}$;
		\item $\lrangles{B\xrightrightarrows[g_2]{g_1} C} +\lrangles{A\xrightrightarrows[f_2]{f_1} B}=\lrangles{A\xrightrightarrows[g_2f_2]{g_1f_1} B}$.
	\end{enumerate}
\end{theorem}

A couple remarks are in order. First, \cite[Theorem B]{ZakhK1} leaves open the possibility that there may be more relations on $K_1$ to be identified. This incompleteness is inherited from Muro-Tonks' original model of $K_1$: although \cite[Prop 6.3]{MuroTonks} shows that their model coincides with Nenashev's model for exact categories (and is thus complete), they were unable to show the same for all Waldhausen categories. Second, Relation (Z1) clearly corresponds to the diagonal relation (N1) of $\calD(\calC)$. It remains to investigate Relations (Z2) and (Z3) in our context, which we work out below.

\begin{proposition}[Assembler Relations]\label{prop:ass} Let $f$ and $g$ be two double exact squares
		\begin{equation*}\label{eq:assf}
f:= \left(	\dsquaref{O}{X}{A}{B}{ }{ }{ f_{1}}{f_{2} } \quad ,\quad 	\dsquaref{O}{X}{A}{B}{ }{ }{ f'_{1}}{f'_{2} }\right)  \qquad\quad g:= \left(	\dsquaref{O}{Y}{C}{D}{ }{ }{ g_{1}}{g_{2} } \quad ,\quad 	\dsquaref{O}{Y}{C}{D}{ }{ }{ g'_{1}}{g'_{2} }\right). 
	\end{equation*}
Then the following relations hold in $\calD(\calC)$:	
\begin{enumerate}[label=(A\arabic*)]
	\item \emph{(Formal Direct Sums).} $\lrangles{f}+\lrangles{g}=\lrangles{f\oplus g}$.
\item \emph{(Restricted Composition).} Suppose $B=C$, and so $f$ and $g$ define an admissible triple. Then
$$\lrangles{f} + \lrangles{g}=\lrangles{g\circ f} + \lrangles{l_2},$$ 
where $l_2$ is the associated double exact square 
\[l_2:=\left( \dsquaref{O}{Y}{X}{\frac{D}{A}}{ }{ }{h_1}{j_1} \quad,\quad  \dsquaref{O}{Y}{X}{\frac{D}{A}}{ }{}{h'_1}{j'_1}\right), \qquad \text{and} \qquad \lrangles{g\circ f}:= \left(	\dsquaref{O}{\frac{D}{A}}{A}{D}{ }{ }{ g_{1}f_1}{ } \quad ,\quad 	\dsquaref{O}{\frac{D}{A}}{A}{D}{ }{ }{ g'_{1}f'_1}{}\right).\]
\end{enumerate}	
\end{proposition}
\begin{proof} We derive Relations (A1) and (A2) from Relations (N1) and (N2).
\begin{itemize}
	\item[(A1):] We claim the following is an optimal $\tx$ diagram:
	\begin{equation}\label{eq:asscoprod}
	\left(\begin{tikzcd}
	A \ar[r, >->, "f_{1}"] 
	\ar[d,>->]  & B \ar[d, >->] & X
	\ar[d,>->] \ar[l, {Circle[open]}->,swap,"f_{2}"]\\
	A\oplus C \ar[r,>->,"f_{1}\oplus g_1"]  &	B\oplus D \ar[dr,phantom,"\circlearrowleft"]  & X\oplus Y\ar[l, {Circle[open]}->,swap,"f_2\oplus g_2"] \\
	C \ar[u, {Circle[open]}->] \ar[r,>->,"g_1"]& D\ar[u, {Circle[open]}-> ] & Y \ar[u, {Circle[open]}->,swap]\ar[l, {Circle[open]}->,swap,"g_{2}"]
	\end{tikzcd} \qquad\text{,}\qquad  \begin{tikzcd}
	A \ar[r, >->, "f'_{1}"] 
	\ar[d,>->]  & B \ar[d, >->] & X
	\ar[d,>->] \ar[l, {Circle[open]}->,swap,"f'_{2}"]\\
	A\oplus C \ar[r,>->,"f'_{1}\oplus g'_1"]  &	B\oplus D  \ar[dr,phantom,"\circlearrowleft"] & X\oplus Y \ar[l, {Circle[open]}->,swap,"f'_2\oplus g'_2"] \\
	C \ar[u, {Circle[open]}->] \ar[r,>->,"g'_1"]& D \ar[u, {Circle[open]}-> ]& Y \ar[u, {Circle[open]}->] \ar[l, {Circle[open]}->,swap,"g'_{2}"]
	\end{tikzcd}\right)
	\end{equation}		
	The vertical columns -- denoted $l^0,l^1,l^2$ -- define the direct sum squares from Axiom (A), Definition~\ref{def:pCGW}. The top and bottom rows correspond to $f$ and $g$. The middle rows correspond to $f\oplus g$, as defined in Lemma~\ref{lem:DirectSum}. To define the required $\M$-morphisms for optimality, take repeated restricted pushouts of 
	$$ D\xltail{g_1} C\ltail O \rtail A\xrtail{f_1} B \quad \text{and}\quad  D\xltail{g'_1} C\ltail O \rtail A\xrtail{f'_1} B;$$
the obvious choices of $\M$-morphisms are easily seen to satisfy Optimality Identities~\eqref{eq:OPTID}. Lemma~\ref{lem:DirectSum} then  supplies Diagrams (P1) – (P4) by “adding” the obvious filtrations. Finally, to see why the $\E$-squares commute, use the fact that the direct sum quotients are defined coordinate-wise, that is $f_2\oplus g_2:=(1\oplus g_2)\circ (f_2\oplus 1)$ and $f'_2\oplus g'_2:=(1\oplus g'_2)\circ (f'_2\oplus 1)$.\footnote{{\em Details.} The validity of these quotient choices is established in the proof of Lemma~\ref{lem:DirectSum}, Step~2. Using this, one verifies that the $\E$-squares are quotients of the same $\M$-morphism and hence commute up to canonical isomorphism on $Y$. By Axiom (PQ), this isomorphism lifts to $X\oplus Y$ and thus may be absorbed into the choice of direct sum quotients.}

	This sets up the final move. Since Diagram~\eqref{eq:asscoprod} is optimal, apply Relation (N2) to get 
	\[\lrangles{f}+\lrangles{g}-\lrangles{f\oplus g} = \lrangles{l^0} + \lrangles{l^2} - \lrangles{l^1}.\]
	Since $l^0, l^1$ and $l^2$ are all diagonal, deduce from (N1) that $\lrangles{l^0} = \lrangles{l^1} =\lrangles{l^2} = 0$, and so conclude 
	\[\lrangles{f}+\lrangles{g}=\lrangles{f\oplus g}. \]
	\item[(A2):] Given that $B=C$, construct the following diagram
	\begin{equation}\label{eq:asscompose}
	\left(\begin{tikzcd}
	A \ar[r, >->, "="] 
	\ar[d,>->,swap,"f_1"]  & A \ar[d, >->,"g_1f_1"] & O
	\ar[d,>->] \ar[l, {Circle[open]}->]\\ B \ar[r,>->,"g_1"]  &	D \ar[dr,phantom,"\circlearrowleft"]  & Y \ar[l, {Circle[open]}->,swap,"g_2"] \\
X \ar[u, {Circle[open]}->,"f_2"] \ar[r,>->,"h_1"]& \frac{D}{A} \ar[u, {Circle[open]}->]& Y \ar[u, {Circle[open]}->,swap,"="] \ar[l, {Circle[open]}->,swap,"j_1"]
	\end{tikzcd} \qquad\text{,}\qquad  \begin{tikzcd}
	A \ar[r, >->, "="] 
	\ar[d,>->,swap,"f'_1"]  & A \ar[d, >->,"g'_1f'_1"] & O
	\ar[d,>->] \ar[l, {Circle[open]}->]\\ B \ar[r,>->,"g'_1"]  &	D \ar[dr,phantom,"\circlearrowleft"]  & Y \ar[l, {Circle[open]}->,swap,"g'_2"] \\
X \ar[u, {Circle[open]}->,"f'_2"] \ar[r,>->,"h'_1"]& \frac{D}{A} \ar[u, {Circle[open]}->]& Y \ar[u, {Circle[open]}->,swap,"="] \ar[l, {Circle[open]}->,swap,"j'_1"]
	\end{tikzcd}  \right).
	\end{equation}		
As before, label the horizontal rows as $l_0$, $l_1$ and $l_2$ and the vertical columns as $l^0, l^1$ and $l^2$, where $l_1=g$, $l^0=f$ and $l^1=g\circ f$. It is obvious most of these define double exact squares; the only non-trivial case is $l_2$, but this is directly constructed by Lemma~\ref{lem:quotFilt}. In fact, the $\E$-morphisms $(j_1,j'_1)$ are specifically chosen by Lemma~\ref{lem:quotFilt} to ensure that the $\E$-squares commute, and so Diagram~\eqref{eq:asscompose} is a $\tx$ diagram. Finally, to check optimality, pick the following pairs of $\M$-morphisms:
	$$v:=g_1,\,v':= g'_1,\qquad u:=f_1,\,u':=f'_1 \qquad \text{and}\quad w,\, w':=1_B.$$ 
	The Optimality Identities are trivially satisfied; it is also easy to construct the required Diagrams (P1) – (P4). 
	Now apply Relation (N2) to get 
	\[\lrangles{l_0}+\lrangles{l_2}-\lrangles{g} = \lrangles{f} + \lrangles{l^2} - \lrangles{g\circ f}.\]
	Since $l_0$ and $l^2$ are diagonal, apply (N1) to conclude
	\[\lrangles{g\circ f} + \lrangles{l_2}= \lrangles{f} + \lrangles{g}.\] 
\end{itemize}
\end{proof}

Relation (A2) improves upon Proposition~\ref{prop:baseline} by determining what happens to {\em any} admissible triple in $K_1$, clarifying Warning~\ref{warning:compose}. Namely, given an admissible triple $\calT=(f,g,g\circ f)$, we now know that
$$\lrangles{g\circ f} + \lrangles{l_2} = \lrangles{f}+\lrangles{g} \qquad \text{in}\, K_1.$$
Unlike Theorem~\ref{thm:ZakhB}, where composition always splits in $K_1$, here an obstruction term $\lrangles{l_2}$ appears. Of course, if $\lrangles{l_2}=0$ then Relations (A2) and (Z3) coincide, but this is not true in general, as below.  

\begin{example}\label{ex:KeyExample} Given any automorphism $\alpha\colon B\rtail B$, construct the diagram pair
	\[\begin{tikzcd}
	A \ar[r, >->] \ar[dr,phantom,"\square"] & A\oplus B  \ar[dr,phantom,"\square"] \ar[r, >->,"1_A\oplus 1_B"] & A\oplus B \\
	O \ar[r,>->] \ar[u,{Circle[open]}->]&	B \ar[r, >->,"1_B"] \ar[u, {Circle[open]}->] \ar[dr,phantom,"\square"]  & B \ar[u, {Circle[open]}->] \\
	& O \ar[r,>->] \ar[u,{Circle[open]}->]& O \ar[u, {Circle[open]}->] 
	\end{tikzcd}  \quad \begin{tikzcd}
	A \ar[r, >->] \ar[dr,phantom,"\square"] & A\oplus B  \ar[dr,phantom,"\square"] \ar[r, >->,"1_A\oplus \alpha "] & A\oplus B \\
	O \ar[r,>->] \ar[u,{Circle[open]}->]&	B \ar[r, >->,"\alpha"] \ar[u, {Circle[open]}->] \ar[dr,phantom,"\square"]  & B \ar[u, {Circle[open]}->] \\
	& O \ar[r,>->] \ar[u,{Circle[open]}->]& O \ar[u, {Circle[open]}->] 
	\end{tikzcd}	\qquad.
	\]	
Notice the obstruction term here is $\lrangles{l_2}=l(\alpha)$, the double exact square associated to $\alpha$ (Example~\ref{ex:Aut}). It is known that $l(\alpha)\neq 0$ in general, including the case of varieties -- see e.g. \cite{ZakhPtCount}.  
\end{example}

\begin{discussion}\label{dis:Comp} We suspect this discrepancy with Zakharevich's account stems from \cite[Theorem 2.1]{ZakhK1}, a key ingredient in Zakharevich's proof of Theorem~\ref{thm:ZakhB}, which models the $K$-theory of (nice) Assemblers using Waldhausen categories whose cofibration sequences all split (up to weak equivalence). We emphasise that this relies on a non-standard notion of weak equivalence \cite[Def. 1.7]{ZakhK1}. In the CGW setting, it is clearly false that all exact squares split in $\Var_k$ -- consider e.g.	\[\dsquare{O}{\mathbb{A}^1}{\{\ast\} }{ \mathbb{P}^1}.\]
\end{discussion}

\begin{remark} Proposition~\ref{prop:ass} identifies certain relations of the abelian group $\calD(\calC)$, and so naturally the proof was algebraic. However, since we are ultimately interested in $K_1(\calC)=\pi_1|G\calC^{\mathsf{o}}|$, one may also appeal to the homotopical perspective to establish the relations more directly. Relation (A1) can be deduced from the $H$-space structure of $G\calC$ and applying the standard Eckmann-Hilton argument; Relation (A2) translates Lemma~\ref{lem:admtriple}, though the lemma's proof is comparatively involved (see Section~\ref{sec:admtriples}).
\end{remark}


We end with one final application. In order to show that $\calD(\calC)\cong K_1(\calC)$ for exact categories, Nenashev \cite{Nen1} constructs a homomorphism
$$b\colon K_1(\calC)\to \calD(\calC)$$
and shows that it is inverse to the map $m\colon \calD(\calC)\to K_1(\calC)$ from Equation~\eqref{eq:NenSurj}. Notice the naive map sending $\lrangles{f}\mapsto \lrangles{f}$ is not \emph{a priori} well-defined since $K_1(\calC)$ may impose relations beyond those of $\calD(\calC)$. In \cite[\S 4]{Nen1}, well-definedness of $b$ is verified through substantial combinatorial bookkeeping, before the final proof that it yields the desired isomorphism.\footnote{{\em Commentary.} With hindsight, one can interpret Nenashev’s proof as explicitly deducing an Eckmann–Hilton result from Relations (N1) and (N2) of $K_1$; compare \cite[\S 4]{Nen1} with our Observation~\ref{obs:K1DirectSum}. Notice the only place where homotopical input appears is \cite[Lemma 4.9]{Nen1}, which verifies invariance under elementary 2-simplex (triangle) moves.} In our framework, Propositions~\ref{prop:baseline} and ~\ref{prop:ass} combine to give a more direct proof.

\begin{corollary}\label{cor:Nen} Given any pCGW category $\calC$, there is an isomorphism
	$$\calD(\calC)\cong K_1(\calC).$$	
\end{corollary}
\begin{proof} It suffices to check that $\calD(\calC)$ implies the relations in Proposition~\ref{prop:baseline}, which we know to be complete. Relations (B1) and (B2) are diagonal relations, and follow immediately from Relation (N1). Relation (B3) is a special case of Relation (A2) in Proposition~\ref{prop:ass}, where $\lrangles{l_2}=0$ since $l_2$ is diagonal by assumption. 
\end{proof}

\begin{remark} In fact, the proof of Corollary~\ref{cor:Nen} shows $K_1(\calC)$ can be presented as the free abelian group generated by double exact squares, modulo the Diagonal Relation (N1) (Definition~\ref{def:NenDC}) and Restricted Composition Relation (A2) (Proposition~\ref{prop:ass})	
\end{remark}

\section{Some Test Problems}

This paper began with Question~\ref{qn:key}: what information do the higher $K$-groups of varieties encode? Progress on this question requires advances on two fronts: developing the framework of non-additive $K$-theory on the one hand, and concrete applications of the newly-developed ($K$-theory) tools on the other. This paper is of the first kind, with a view towards laying the groundwork for future theorems. We conclude with some test problems and discussion.

\subsection{Non-Additive $K$-theory}\label{sec:NAKtheory}

Having characterised $K_1$, the natural next step is the following problem.

\begin{problem} Characterise $K_n(\calC)$ for $n>1$ for pCGW categories. 
\end{problem}

\begin{discussion} In a substantial generalisation of Nenashev's work, Grayson gave a complete characterisation of $K_n$ for all $n$ in the setting of exact categories \cite{GraysonBinary}. Encouraged by the results in this paper, the natural proof strategy would be to extend Grayson's argument to the pCGW setting. 
	
	However, there is an obvious barrier. Grayson characterises the $K$-groups of exact categories via binary chain complexes, and so invokes the Gillet-Waldhausen Theorem. Although an analogue of this result has been shown for extensive categories (e.g. $\mathrm{FinSet}$) \cite{SarazolaShapiroGW}, a Gillet-Waldhausen Theorem for $\Var_k$ has not been worked out yet. Nonetheless, even the case of $\mathrm{FinSet}$ is interesting. By Barratt-Priddy-Quillen, the $K$-groups of $\mathrm{FinSet}$ correspond to the stable homotopy groups of spheres, whose complete description is a longstanding open problem in homotopy theory. What might this presentation of $K_n(\mathrm{Finset})$ tell us about the stable homotopy groups of spheres? About the $J$-homomorphism and Adams $e$-invariant?
\end{discussion}

In light of Proposition~\ref{prop:ass}, perhaps a more urgent problem is the following.

\begin{problem} Resolve the discrepancy between our presentation of $K_1(\Var_k)$ and Zakharevich's.
\end{problem}


\begin{discussion} Here is one way to make precise Discussion~\ref{dis:Comp} regarding the choice of weak equivalences. Our proof of Theorem~\ref{thm:Sherman} makes crucial use of Lemma~\ref{lem:SherLoopSplit}, which states: if $\calC$ is a pCGW category whose exact squares all split, then $K_1(\calC)$ is generated by automorphisms. In which case, by Corollary~\ref{cor:BirIso}, composition splits in $K_1$ with no obstruction term. In light of \cite[Theorem 2.1]{ZakhK1}, can we extend Lemma~\ref{lem:SherLoopSplit} to Waldhausen categories whose cofibration sequences all split? 
\end{discussion}

\begin{discussion} Example~\ref{ex:KeyExample} shows how the obstruction term may fail to vanish in $K_1(\Var_k)$. So what happens in the Assemblers setting?

Recall from \cite[\S 5.1]{ZakhAss} the assembler $\calV_k$: its objects are $k$-varieties, its morphisms locally closed immersions, and its Grothendieck topology is generated by coverages of the form $\{Y\to X,X\setminus Y\to X\}$ for closed embeddings $Y\to X$. In Zakharevich's framework, the generators of $K_1(\calV_k)$ are finite tuples of morphisms from $\calV_k$.

We now translate Example~\ref{ex:KeyExample} to Assemblers (with coproducts). Take two multimorphisms $$f, f' \colon \{\{A\}, \{B\}\} \to A \oplus B$$
given by coproduct inclusions, followed by automorphisms $1 \oplus 1$ and $1 \oplus \alpha$. These correspond to the data of Example~\ref{ex:KeyExample} in a natural way, yet their composition yields a different admissible triple: 

	\[  \begin{tikzcd}
A \ar[r, >->] \ar[dr,phantom,"\square"] & A\oplus B  \ar[dr,phantom,"\square"] \ar[r, >->,"1_A\oplus 1_B"] & A\oplus B \\
O \ar[r,>->] \ar[u,{Circle[open]}->]&	B \ar[r, >->,"1_B"] \ar[u, {Circle[open]}->] \ar[dr,phantom,"\square"]  & B \ar[u, {Circle[open]}->,swap,"A\oplus 1_B"] \\
& O \ar[r,>->] \ar[u,{Circle[open]}->]& O \ar[u, {Circle[open]}->] 
\end{tikzcd} \quad	\begin{tikzcd}
A \ar[r, >->] \ar[dr,phantom,"\square"] & A\oplus B  \ar[dr,phantom,"\square"] \ar[r, >->,"1_A\oplus \alpha"] & A\oplus B \\
O \ar[r,>->] \ar[u,{Circle[open]}->]&	B \ar[r, >->,"1_B"] \ar[u, {Circle[open]}->] \ar[dr,phantom,"\square"]  & B \ar[u, {Circle[open]}->,swap,"A\oplus\alpha"] \\
& O \ar[r,>->] \ar[u,{Circle[open]}->]& O \ar[u, {Circle[open]}->] 
\end{tikzcd}
\]	
In this case, the obstruction term is now $\lrangles{1_B,1_B}$, which vanishes. That is, the corresponding composition in our presentation of $K_1(\Var_k)$ also splits. It is presently unclear how this argument extends to arbitrary pairs of piecewise automorphisms, but it offers a template for investigating the general case.\footnote{Where is the difficulty? Unlike the pCGW category $\Var_k$, which treats open and closed immersions separately, the assembler $\calV_k$ freely combines them. This wasn't a problem in the present example (which only involved coproduct inclusions and automorphisms), but will require careful handling in the general case. Arguments from \cite[\S 9]{CGW} may be helpful.}
\end{discussion}

\begin{remark} In \cite{CGW}, the authors introduce the notion of an ACGW category, which essentially formalises how $\Var_k$ resembles an abelian category.\footnote{Strictly speaking, only the category of reduced schemes of finite type is an ACGW category, not $\Var_k$. However, \cite[Example 6.3]{CGW} shows the two categories have equivalent $K$-theory.} As such, it is perhaps worth recalling that every abelian category $\calA$ supports a split exact structure, which in general defines a different $K$-theory from the natural exact structure on $\calA$ -- see the warning after \cite[Example II.7.1.2]{WeibelKBook}.
\end{remark}

Another natural question, posed to the author by Emanuele Dotto, is the following.

\begin{problem} Does the $K$-theory of pCGW categories commute with infinite products?
\end{problem}

\begin{discussion} Relevantly: Zakharevich conjectures in \cite[Remark 2.4]{ZakhPtCount} that the $K$-theory of Assemblers commutes with infinite products. However, she notes that such this result seems presently out of reach since previous results of this form were worked out for Waldhausen categories with cylinder functors (which Assemblers do not have) and exact categories.
\end{discussion}

\subsection{The Motivic Euler Characteristic} There is a well-known enrichment of the Euler Characteristic, known as the {\em motivic Euler Characteristic} or {\em compactly-supported $\mathbb{A}^1$-Euler Characteristic}, defined as a ring homomorphism
$$\chi^{\mathrm{mot}}\colon K_0(\Var_k) \to \mathrm{GW}(k)$$
valued in $GW(k)$, the Grothendieck-Witt ring of quadratic forms over field $k$. An exposition of its construction can be found in \cite{5authorZakhWick}. It is natural to ask if one can lift this to the level of $K$-theory spectra, which was proved in the affirmative by Nanavaty \cite[Theorem 1.1]{Nanavaty}. Explicitly, he constructs a map of spectra
$$K\Var_k \to \End(\sk)$$
where $\End(\sk)$ is the {\em Endomorphism Spectrum of the unit object in the motivic stable homotopy category}, recovering $\cmot$ on $\pi_0$. This sets up the problem:

\begin{problem}\label{prob:K1} Define a natural map $K_1(\Var_k)\to \pi_{1,0}(\sk)$. What geometric information does it encode?
\end{problem}

\begin{discussion} The homotopy groups of $\End(\sk)$ are defined as $\pi_{\ast,0}$. A foundational result, due to Morel \cite{Morel}, shows that $\pi_{0,0}(\sk)\cong \mathrm{GW}(k)$, so one may regard the higher homotopy groups as defining higher Grothendieck-Witt Groups. What geometric information is detected at the higher levels? Recent work by \cite{1lineMOTIVIC} tells us 
	\begin{equation}\label{eq:pi1MOREL}
	0\to K^M_2(k)/24\to \pi_{1,0}(\sk)\to k^\times/2\oplus \Z/2\to 0,
	\end{equation}
where $K^M_\ast(k)$ denotes the Milnor $K$-theory of $k$.
Combined with Theorems~\ref{thm:K1} and/or \ref{thm:Nenashev}, we now have an explicit description of both groups in Problem~\ref{prob:K1}. It remains to map the double exact squares in $\Var_k$ to $\pi_{1,0}(\sk)$ in a natural way, but it is presently unclear, e.g. how these generalised automorphisms interact with the Milnor $K$-theory term, and what this means geometrically. 
\end{discussion}

Note: the canonical unit map $\End(\sk)\to KQ$ (i.e. the Hermitian $K$-theory spectrum) induces an isomorphism on the level of $\pi_0$.  We may therefore also think of  $\cmot$ as a map on $\pi_0$ of $K\Var_k\to KQ$, which may be more tractable than the full motivic-homotopy formulation. For those interested in power structures on $K_0(\Var_k)$, here is a natural question:

\begin{problem}\label{prob:ExP} Can we lift the symmetric power structure on $K_0(\Var_k)$ to all $K$-groups
	$$\Sym^{(m)}\colon K_n(\Var_k)\longrightarrow K_n(\Var_k)\qquad?$$
Some natural follow-ups:
\begin{enumerate}[label=(\roman*)]
	\item Can the compatibility of $\cmot$ with symmetric powers on $\pi_0$ be checked (or deduced) from the simplicial presentation given by the $G$-construction?
	\item If a direct lift fails, is there a natural modification of the construction that {\em does} lift -- e.g. reduced symmetric powers with the fat diagonals removed? 
\end{enumerate}	
\end{problem}

\begin{discussion} Notice Problem~\ref{prob:ExP} asks about operations on $K$-groups rather than the underlying $K$-theory spaces. This is a subtle distinction since space-level definitions can be problematic. To illustrate: given the $Q$-construction $\calQ\M$ of an exact category with exterior powers, one might try to define $\lambda$-operations on the $K$-groups via a self-map $\lambda^m\colon |\calQ\calM|\to|\calQ\calM|$. But this cannot work: such a map would induce a group homomorphism on $\pi_1$, whereas the usual $\lambda^m$ is not additive on $K_0$. 

Interestingly, these are precisely the kind of issues the $G$-construction was originally designed to resolve -- at least in algebraic $K$-theory. 
Indeed, in contrast with other models of $K$-theory (e.g. the $Q$-construction, the $S_\bullet$-construction etc.), Grayson \cite{GraysonExterior,GraysonAdams} showed that the $G$-construction can be leveraged to give combinatorial descriptions of $\lambda$– and Adams operations on higher $K$-groups for suitable exact categories.
	
Returning to our context: the compatibility of $\cmot$ with power structures on $K_0$ and $GW(k)$ has so far only been established in special cases, typically via deep arithmetic (e.g. \cite{PajPalPower,PajwaniCellular}). Problem~\ref{prob:ExP} calls for a shift in perspective: what if we approach the problem simplicially instead? The jury is still out, but Theorem~\ref{thm:GG} gives a promising first clue: the $G$-construction behaves as expected on $\Var_k$. We plan to explore this direction in future work.
\end{discussion}

\begin{remark} Lifting the symmetric power structure on $K_0(\Var_k)$ ought to have useful implications for lifting Kapranov's motivic zeta function to a map of $K$-theory spectra -- see \cite[Question 7.6]{CWZ}. 
\end{remark}

\subsection{Combinatorics of Definable Sets} Theorem~\ref{thm:NenashevDSES} showed that $K_1(\calC)$ of a pCGW category is generated by double exact squares; Example~\ref{ex:Aut} showed that these generalise automorphisms. Automorphisms play a key role in many different areas of mathematics -- can this be extended to these double exact squares in a productive way?

For the model theorist, the automorphism group $\Aut(M)$ of a countable first-order structure $M$ encodes important information about $M$ -- e.g. $\Aut(M)$ measures the homogeneity of $M$, but there are other insights \cite{AutKayeMcPherson,EvansAut}. Some interesting questions:

\begin{problem} A countable structure $M$ is said to be {\em homogeneous} just in case every finite partial isomorphism extends to a global automorphism of $M$. Can we adapt 
	the obstruction theory developed in \cite{ZakhLefschetz} to analyse, e.g. barriers to homogeneity or amalgamation?  What about uncountable structures? 
\end{problem}

\begin{problem} To what extent is $K_1(M)$ be useful for studying non-homogeneous structures? 
\end{problem}


\begin{problem} In the converse direction, by viewing $\Aut(M)$ as a topological group, model theorists were able to analyse its structure productively using a wide range of tools. Can we apply the same analysis to $K_1 (M)$? What new insights does this give? 
\end{problem}

It is also worth revisiting the original papers \cite{Krajicek,KrajicekScanlon} where the Grothendieck ring of Definable Sets was first investigated. In particular, \cite{KrajicekScanlon} introduces the so-called strong and weak Euler Characteristics on first-order structures before asking which fields admit a non-trivial strong Euler Characteristic. In light of our present work, a natural problem may be:

\begin{problem} Lift the weak/strong Euler Characteristic on first-order structures to the level of spectra. Analyse what happens on $K_1$ -- what information does it detect? 	 In addition, are there examples of fields with strong Euler characteristics that are trivial on $K_0$ but non-trivial on $K_1$?  
\end{problem}

\subsection{Matroids}\label{sec:ProbMatroids} Similar questions about generalised automorphisms may be posed regarding matroids. However, in light of Example~\ref{ex:MatroidspCGW}, a more urgent problem is the following:

\begin{problem} What is the right notion of restricted pushouts for matroids?
\end{problem}

\begin{discussion} Unlike the other pCGW categories, the restricted pushout of $M_0\ltail N\rtail M_1$ cannot be the pushout in the ambient category $\mathrm{Mat}_\bullet$ since this may not exist. To see why, consider the following example suggested to the author by Chris Eppolito: let $N=\{a,b,c,\bullet\}, M_0=\{a,b,c,d,\bullet\}$ and $M_1=\{a,b,c,e,\bullet\}$, all endowed with the uniform matroid structure of rank 2, with point $\bullet$. 
\end{discussion}

\begin{discussion}[Independent Squares] A related problem was considered by the model theorists. Let $\calD$ be a category whose morphisms are all monic, and suppose $\calD$ admits a {\em stable independence notion}, i.e. a class of so-called {\em independent squares}
	$$\begin{tikzcd}
	N \ar[r,>->] \ar[d,>->]  \ar[dr,phantom,"\forkindep"]& M_1 \ar[d,>->]\\
	M_0  \ar[r,>->]& M
	\end{tikzcd}$$  
satisfying the axioms listed in \cite[Definition 5.5]{Vasey}. This includes the condition that any span in $\calD$ $$M_0\ltail N\rtail M_1$$
can be completed into an independent square.  Examples include $\mathrm{FinSet}$, the category of vector spaces over a fixed fields (with injective linear transformations), and the category of graphs (with subgraph embeddings). 

Here is the key insight from \cite[\S 5.1]{Vasey}. Under certain technical conditions, the category of independent squares associated to a given span has a weakly initial object, which is unique up to (not necessarily unique) isomorphism. Model theorists call this the {\em prime object} over the span. Some natural questions:
\begin{enumerate}
	\item  Can the definition of restricted pushouts be (meaningfully) weakened to include prime objects? 
	\item Can specific classes of matroids (e.g. vectorial, or graphic matroids) admit such a structure? All classes?
	\item How do independent squares align with matroid-theoretic notions like rank and flats?
\end{enumerate}
This will be investigated in future work.
\end{discussion}

\appendix
\section{Properties of Restricted Pushouts}

As explained in Section~\ref{sec:pCGW}, it is unreasonable to ask for $\M$-morphisms of CGW categories to be closed under pushouts, so we instead ask for them to be closed under {\em restricted pushouts}, a weaker notion. This section collects various technical facts about them, which hold in any pCGW category $\calC$. Finally, whenever relevant, we follow Convention~\ref{conv:coord} when labelling morphisms. 

\begin{lemma}\label{lem:Hspace1}\hfill 
	\begin{enumerate}[label=(\roman*)]
		\item Given a span $B\leftarrowtail A \rtail C$, 
		$$\frac{B}{A}\oplus \frac{C}{A}\cong \frac{B\star_A C}{A}.$$
		\item Given a span $B\leftarrowtail A \rtail C$, along with $\M$-morphisms $B\rtail B'$ and $C\rtail C'$,
		$$\frac{B'\star_A C'}{B\star_A C} \cong \frac{B'}{B}\oplus \frac{C'}{C}.$$
	\end{enumerate}
\end{lemma}

\begin{proof}\hfill 
\begin{enumerate}[label=(\roman*):]
	\item Consider the diagram
	\[\begin{tikzcd}
	\frac{C}{A}\ar[d, {Circle[open]}->] \ar[dr, phantom, "\square"]& O \ar[dr, phantom, "\square"]\ar[d, {Circle[open]}->]  \ar[l, >->] \ar[r, >->] & \frac{B}{A} \ar[d, {Circle[open]}->] \\
	C & A \ar[l, >->] \ar[r, >->] & B
	\end{tikzcd}\]
	Apply Axiom (DS) to obtain the left diagram below:
	\begin{equation}\label{eq:PPiso}
	\begin{tikzcd}
	O \ar[dr, phantom, "\square"]\ar[d, {Circle[open]}->]  \ar[r, >->] & \frac{B}{A} \ar[d, {Circle[open]}->] \ar[dr, phantom, "\square"] \ar[r,>->] &   \frac{B}{A}\oplus \frac{C}{A} \ar[d,{Circle[open]}->]\\
	A \ar[r, >->,"f"] & B \ar[r,>->,"f'"]  & B\star_A C
	\end{tikzcd}\qquad\qquad  \dsquaref{O}{\frac{B\star_AC}{A}}{A}{B\star_AC}{}{}{f'f}{} 
	\end{equation}
	Since distinguished squares compose, the outermost rectangle of the left diagram also defines a distinguished square. On the other hand, by Axiom (K), the RHS exact square exists. Since formal cokernels are unique up to isomorphism, conclude that $\frac{B}{A}\oplus \frac{C}{A}\cong \frac{B\star_A C}{A}$.
	\item Repeated applications of Fact~\ref{facts:restrictedpushouts} yields 
	\begin{equation}
	\begin{tikzcd}
	A \ar[r,>->] \ar[d,>->] & B \ar[r,>->] \ar[d,>->] & B' \ar[d,>->]\\
	C \ar[r,>->] \ar[d,>->] & B\star_A C  \ar[d,>->] \ar[r,>->] &  B'\star_A C \ar[d,>->]\\
	C'\ar[r,>->]& B\star_AC' \ar[r,>->]& B'\star_AC'
	\end{tikzcd}
	\end{equation}
	In particular, there exists an $\M$-morphism $B\star_AC\rtail B'\star_A C'$, so $\frac{B'\star_AC'}{B\star_AC}$ is well-defined and exists by Axiom (K). By Axiom (PQ), restricted pushouts preserve quotients, and so
	\begin{equation}\label{eq:iso(iii)}
	\frac{B'}{B} \cong \frac{B'\star_AC}{B\star_A C}\quad\text{and} \quad \frac{C'}{C}\cong  \frac{B\star_AC'}{B\star_A C}\,\,.
	\end{equation}
	Now consider the diagram
	\begin{equation}\label{eq:PP(iii)}
	\begin{tikzcd}
	\frac{B\star_AC'}{B\star_A C} \ar[d, {Circle[open]}->] \ar[dr, phantom, "\square"]& O \ar[dr, phantom, "\square"]\ar[d, {Circle[open]}->]  \ar[l, >->] \ar[r, >->] & \frac{B'\star_AC}{B\star_A C} \ar[d, {Circle[open]}->] \\
	B\star_A C' & B\star_A C \ar[l, >->] \ar[r, >->] & B'\star_AC
	\end{tikzcd}
	\end{equation}
	Applying the same argument from (i) to Equations~\eqref{eq:iso(iii)} and \eqref{eq:PP(iii)}, deduce the desired isomorphism.
\end{enumerate}	
\end{proof}

\begin{corollary}[Pushouts create 2-simplices]\label{cor:RPtrick} Given any span of 1-simplices in $G\calC$
\begin{equation}\label{eq:RPtrick1}
\ones{M'}{L'}\leftarrow \ones{M}{L} \to \ones{M''}{L''},
\end{equation}
taking restricted pushouts yields the following diagram
\begin{equation}\label{eq:RPtrickDiag}
\begin{tikzcd}
\ones{M}{L} \ar[rr] \ar[d] \ar[drr] && \ones{M''}{L''} \ar[d]\\
\ones{M'}{L'} \ar[rr]&& \ones{M'\star_MM''}{L'\star_L L''}
\end{tikzcd},
\end{equation}
where both triangles bound 2-simplices. In particular, this defines a homotopy between Diagram~\eqref{eq:RPtrick1} and 
\begin{equation}
\ones{M'}{L'} \to \ones{M'\star_M M''}{L'\star_L L''} \leftarrow \ones{M''}{L''}.
\end{equation}
\end{corollary}
\begin{proof} Since Diagram~\eqref{eq:RPtrick1} is a span of 1-simplices, their choice of formal quotients agree. By a harmless abuse of notation, 
denote these quotients as:
	$$\frac{M'}{M}=\frac{L'}{L}\quad\text{and}\quad \frac{M''}{M}=\frac{L''}{L}.$$
 Now form restricted pushouts $P:=M'\star_M M''$ and $Q:=L'\star_L L''$. Apply Lemma~\ref{lem:Hspace1} to deduce
	$$ \small{\frac{P}{M}\cong \frac{M'}{M}\oplus \frac{M''}{M} \qquad \text{and} \qquad \frac{Q}{L}\cong \frac{L'}{L}\oplus \frac{L''}{L}=\frac{M'}{M}\oplus \frac{M''}{M}.}$$
	In fact, Equation~\eqref{eq:PPiso} of the proof tells us the data assembles into diagram pairs, such as
	\[\begin{tikzcd}
	M \ar[r, >->] \ar[dr,phantom,"\square"]&  M' \ar[dr,phantom,"\square"] \ar[r, >->] & P\\
	O \ar[u,{Circle[open]->}] \ar[r,>->] & \frac{M'}{M}  \ar[u,{Circle[open]->}] \ar[r,>->]	\ar[dr,phantom,"\square"]& \frac{M'}{M}\oplus \frac{M''}{M}\ar[u, {Circle[open]}->] \\
	& O \ar[u,{Circle[open]}->] \ar[r,>->] & \frac{M''}{M} \ar[u, {Circle[open]}->] 
	\end{tikzcd}\quad  \begin{tikzcd}
	L \ar[r, >->] \ar[dr,phantom,"\square"]&  L' \ar[dr,phantom,"\square"] \ar[r, >->] & Q \\
	O \ar[u,{Circle[open]->}] \ar[r,>->] & \frac{M'}{M}  \ar[u,{Circle[open]->}] \ar[r,>->]	\ar[dr,phantom,"\square"]& \frac{M'}{M}\oplus \frac{M''}{M}\ar[u, {Circle[open]}->] \\
	& O \ar[u,{Circle[open]}->] \ar[r,>->] & \frac{M''}{M} \ar[u, {Circle[open]}->] 
	\end{tikzcd}\quad.	\]
This defines 2-simplices in $G\calC$
	$$\begin{pmatrix} & O\rtail  &  \doubleunderline{M\rtail M'\rtail P}\\ &O\rtail &L\rtail L'\rtail Q \end{pmatrix}\quad\text{and}\quad \begin{pmatrix} & O\rtail  &  \doubleunderline{M\rtail M''\rtail P}\\ &O\rtail &L\rtail L''\rtail Q \end{pmatrix},$$
which assembles into Diagram~\eqref{eq:RPtrickDiag}.
\end{proof}

Next, examining the examples of pCGW categories in Section~\ref{sec:pCGW}, notice the formal direct sums are all disjoint unions (= the coproduct injections are monic + their pullback is the initial object), and the distinguished squares are all pullbacks (possibly satisfying additional hypotheses). One may therefore wonder if formal direct sums commute with distinguished squares in some natural sense. The following lemmas make this intuition precise.

\begin{convention}[Permutation Isomorphisms]\label{conv:perm-iso}  For any objects $A,B$ in $\calC$, the universal property of restricted pushouts defines a canonical {\em permutation isomorphism} $\tau\colon A\oplus B \rtail B\oplus A$. Since $\tau$ is an isomorphism, Axiom (I) defines a corresponding $\E$-morphism $\varphi(\tau)\colon A\oplus B \otail B\oplus A$. 
\end{convention}

\begin{lemma}\label{lem:ADTsquares} Given any distinguished square
	\begin{equation}\phi:=\left(\dsquaref{O}{C}{A}{B}{}{}{f}{g}\right),
	\end{equation}
the following is true:
	\begin{enumerate}[label=(\roman*)]
		\item Given any object $D$ in $\calC$, we can construct the following distinguished squares: \begin{equation}\label{eq:sumdistSQ}
		\dsquaref{O}{C\oplus D}{A}{B\oplus D}{}{}{f\oplus D}{g\oplus 1} \qquad \text{and} \qquad 
		\dsquaref{O}{C}{A\oplus D}{B\oplus D}{}{}{f\oplus 1}{g\oplus D}; 
		\end{equation} 
		\begin{equation}
		\dsquaref{C}{C\oplus D}{B}{B\oplus D}{1\oplus D}{g}{1\oplus D}{g\oplus 1}\qquad \text{and} \qquad \dsquaref{D}{C\oplus D}{A\oplus D}{B\oplus D}{C\oplus 1}{q_D}{f\oplus 1}{g\oplus 1}.
		\end{equation}
		All newly featured morphisms will be defined in the course of the proof.
		\item The distinguished squares in (i) can be permuted in the obvious way, e.g. 
		\begin{equation}\label{eq:sumdistSQ-perm}
		\dsquaref{O}{D\oplus C}{A}{D\oplus B}{}{}{D\oplus f}{1\oplus g} \qquad \text{and} \qquad 
		\dsquaref{O}{C}{D\oplus A}{D\oplus B}{}{}{1\oplus f}{D\oplus g}; 
		\end{equation} 
		\item Given $A\xrtail{f\oplus D}B\oplus D$ or $A\xrtail{D\oplus f} D\oplus B$ (as defined above), there exists isomorphisms
		$$ \frac{B\oplus D}{A}\cong \frac{B}{A}\oplus D \cong C\oplus D\qquad \text{and} \qquad \frac{D\oplus B}{A}\cong D\oplus \frac{B}{A} \cong D\oplus C \qquad.$$
	\end{enumerate}
\end{lemma}
\begin{proof} 
	\begin{enumerate}[label=(\roman*):]
		\item First, apply Fact~\ref{facts:restrictedpushouts} to
		$$D\leftarrowtail O\rtail A\rtail B,$$
		and obtain the isomorphism $B\oplus D\cong (A\oplus D)\star_A B$. Next, consider the diagram
			\[\begin{tikzcd}
		D \ar[d, {Circle[open]}->,"q_D"] \ar[dr, phantom, "\square"]& O \ar[dr, phantom, "\square"]\ar[d, {Circle[open]}->]  \ar[l, >->] \ar[r, >->] & C \ar[d, {Circle[open]}->,"g"] \\
		A\oplus D & A \ar[l, >->] \ar[r, >->,"f"] & B
		\end{tikzcd}\]
		where $q_D$ is the canonical $\E$-morphism from a direct sum square. Applying Axiom (DS), we obtain the following distinguished squares
		\[\psi:=\left(\dsquaref{C}{C\oplus D}{B}{B\oplus D}{1\oplus D}{g}{1\oplus D}{g\oplus 1}\right)\qquad \psi':=\left(\dsquaref{D}{C\oplus D}{A\oplus D}{B\oplus D}{C\oplus 1}{q_D}{f\oplus 1}{g\oplus 1}\right).\]
		Next, obtain Equation~\eqref{eq:sumdistSQ} by horizontal composition of $\psi$ with $\phi$, and vertical composition of $\psi'$ with the formal sum square 
		\[\dsquare{O}{C}{D}{C\oplus D}.\]
		\item Apply the same analysis in part (i) to the diagram
			\[\begin{tikzcd}
		C \ar[d, {Circle[open]}->,swap,"g"]  \ar[dr,phantom,"\square"]& O 
	\ar[r,>->] \ar[dr,phantom,"\square"]\ar[l,>->,swap]\ar[d, {Circle[open]}->] & D\ar[d, {Circle[open]}->,"q_D"] \ar[dr,phantom,"\square"] 	\ar[r,>->,"1"] & D \ar[d, {Circle[open]}->,"\varphi(\tau)\circ q_D"] \\
		B & A \ar[l,>->,swap,"f"] 	\ar[r,>->] & A\oplus D 	\ar[r,>->,"\tau"]& D\oplus A
		\end{tikzcd}\qquad,\]
	where $\tau\colon A\oplus D\to D\oplus A$ denotes the permutation isomorphism. 
			\item Apply Axiom (K) to the LHS exact squares of Diagrams~\eqref{eq:sumdistSQ} and \eqref{eq:sumdistSQ-perm}.
	\end{enumerate}
\end{proof}

\begin{lemma}[Direct Sums]\label{lem:DirectSum} \hfill
	\begin{enumerate}[label=(\roman*)]
		\item Given any pair of exact squares
\begin{equation}\label{eq:dir-sum-exact-square}
\dsquaref{O}{V_0}{A}{B}{ }{ }{ f_{0}}{f_{1} } \qquad ,\qquad 	\dsquaref{O}{V_1}{C}{D}{ }{ }{ g_{0}}{g_{1} } 
\end{equation}
		we can construct an exact square
		\[\dsquaref{O}{V_0\oplus V_1}{A\oplus C}{B\oplus D}{ }{ }{ f_{0}\oplus g_0 }{f_{1} \oplus g_1}\quad .\]
In particular, this yields an isomorphism $\frac{B\oplus D}{A\oplus C} \cong \frac{B}{A}\oplus \frac{D}{C}$ by Axiom (K).
		\item Given objects $A,B$ in $\calC$, the induced $\M$-morphism $$1_A\oplus 1_B\colon A\oplus B\rtail A\oplus B$$ is the identity $1_{A\oplus B}$. In particular, direct sums commute with the degeneracy maps of $G\calC$.
		\item Continuing with Diagram~\eqref{eq:dir-sum-exact-square} as our setup, the following $\M$-square commutes
	\begin{equation}\label{eq:M-and-E}
	\begin{tikzcd}
	A\oplus C \ar[r,>->,"f_0\oplus g_0"] \ar[d,>->,swap,"\tau"] & B\oplus D \ar[d,>->,"\tau"]\\
	C\oplus A \ar[r,>->,"g_0\oplus f_0"]& D\oplus B
	\end{tikzcd}
	\end{equation}
		where the $\tau$'s denote the permutation isomorphisms.
		\item  Given a pair of flag diagrams
		\begin{equation}\label{eq:dirsum-flag}
		\begin{tikzcd}
		A \ar[r, >->,"f_0"] \ar[dr,phantom,"\square"]&  B  \ar[dr,phantom,"\square"] \ar[r, >->,"g_0"] & C\\
		O \ar[u,{Circle[open]->}] \ar[r,>->] & V_0  \ar[u,{Circle[open]->},"f_1"] \ar[r,>->,"h_0"]	\ar[dr,phantom,"\square"]& V_1 \ar[u, {Circle[open]}->,"g_1"] \\
		& O \ar[u,{Circle[open]}->] \ar[r,>->] & V_2 \ar[u, {Circle[open]}->,"h_1"] 
		\end{tikzcd}\quad  \begin{tikzcd}
		A' \ar[r, >->,"f'_0"] \ar[dr,phantom,"\square"]&  B'  \ar[dr,phantom,"\square"] \ar[r, >->,"g'_0"] & C'\\
		O \ar[u,{Circle[open]->}] \ar[r,>->] & V'_0  \ar[u,{Circle[open]->},"f'_1"] \ar[r,>->,"h'_0"]	\ar[dr,phantom,"\square"]& V'_1\ar[u, {Circle[open]}->,"g'_1"] \\
		& O \ar[u,{Circle[open]}->] \ar[r,>->] & V'_2\ar[u, {Circle[open]}->,"h'_1"] 
		\end{tikzcd} \quad,
		\end{equation}
	direct sums commute with the composition of $\M$-morphisms as follows:
			\begin{equation}\label{eq:dirsum-comp-M}
		(g_0\circ f_0)\oplus (g'_0\circ f'_0)= (g_0\oplus g'_0) \circ  (f_0\oplus  f'_0).
		\end{equation}	
In particular, direct sums commute with the face maps of $G\calC$.
		\item Continuing with Diagram~\ref{eq:dirsum-flag} as our setup, we can construct a distinguished square
		\begin{equation}\label{eq:dirsum-v}
		\dsquaref{V_0\oplus V'_0}{V_1\oplus V'_1}{B\oplus B'}{C\oplus C'}{h_0\oplus h'_0}{f_1\oplus f'_1}{g_0\oplus g'_0}{g_1\oplus g'_1}
		\end{equation}
	\end{enumerate}	
	
\end{lemma}
\begin{proof} \begin{enumerate}[label=(\roman*):]
		\item Apply Lemma~\ref{lem:ADTsquares} to construct the diagram
		\[\begin{tikzcd}
	V_1 \ar[d, {Circle[open]}->,swap,"A\oplus g_1"] \ar[dr, phantom, "\square"]& O \ar[dr, phantom, "\square"]\ar[d, {Circle[open]}->,]  \ar[l, >->] \ar[r, >->] & V_0 \ar[d, {Circle[open]}->,"f_1\oplus C"] \\
		A\oplus D & A\oplus C \ar[l, >->,"1\oplus g_0"] \ar[r, >->,swap,"f_0\oplus 1"] & B\oplus C
		\end{tikzcd}\qquad .\]
	By Fact~\ref{facts:restrictedpushouts}, $(B\oplus C)\star_{A\oplus C} (A\oplus D )\cong B\oplus D$. The rest follows from applying Axiom (DS) and the fact that distinguished squares compose horizontally. This defines Exact Square~\eqref{eq:dir-sum-exact-square}, with the $\M$-morphism:
\begin{equation}\label{eq:dirsum-M-morphism}
f_0\oplus g_0:=(f_0\oplus 1)\circ (1\oplus g_0)=(1\oplus g_0)\circ (f_0\oplus 1).
\end{equation}
		\item This follows from the definition of $1_A\oplus 1_B$ (Equation~\eqref{eq:dirsum-M-morphism}), and the universal property of restricted pushouts. 
		\item As a warm-up, examine the definition of $f_0\oplus 1$ and $1\oplus f_0$ in Lemma~\ref{lem:ADTsquares}, constructed using restricted pushouts. Assemble the data into the diagram below
	\begin{equation}\label{eq:perm-iso-arg}
	\begin{tikzcd}
	O \ar[r,>->] \ar[d,>->] & A \ar[r,>->,"f_0"] \ar[d,>->] & B \ar[d,>->] \\
	C \ar[r,>->] \ar[d,>->, "1"]& A\oplus C \ar[r,>->,"f_0\oplus 1"] \ar[d,>->,"\tau"]& B\oplus C\ar[d,>->,dotted,"\tau" ]\\
	C \ar[r,>->]& C\oplus A \ar[r,>->,"1\oplus f_0"] & C\oplus B
	\end{tikzcd}\quad,
	\end{equation}
		where the permutation isomorphism $\tau\colon B\oplus C\rtail C\oplus B$ is induced by the universal property of restricted pushouts (cf. Convention~\ref{conv:perm-iso}); this yields the identity $\tau\circ (1\oplus f_0)=(f_0\oplus 1)\circ\tau$. Similarly, deduce that  $(1\oplus f_0)\circ\tau  =\tau \circ (f_0\oplus 1)$, with the obvious permutations.

Now compare the definition of $f_0\oplus g_0$ vs. $g_0\oplus f_0$ from Equation~\eqref{eq:dirsum-M-morphism}, modified with the relevant permutations (greyed out below):
\begin{equation}\label{eq:twist-M-and-E}
\begin{tikzcd}
A\oplus C \ar[rr,>->,"f_0\oplus 1 "]  \ar[drr,>->,"f_0\oplus g_0"]\ar[d,>->,swap,"1\oplus g_0"]&& B\oplus C   \ar[d,>->,"1\oplus g_0"] \\
A\oplus D \ar[d,>->,gray,swap,"\tau"] \ar[rr,>->,swap,"f_0\oplus 1"]&& B\oplus D \ar[d,>->,gray, "\tau"]\\
\color{gray} D\oplus A \ar[rr,>->,gray, swap,"1\oplus f_0"]&& \color{gray}{D\oplus B}
\end{tikzcd} \qquad \qquad \begin{tikzcd} \color{gray}{A\oplus C} \ar[rr,>->,"1\oplus g_0",gray]  \ar[d,>->,swap,"\tau", gray] && \color{gray}{A\oplus D} \ar[d,>->,"\tau",gray] \\
C\oplus A \ar[rr,>->,"g_0\oplus 1 "]  \ar[drr,>->,"g_0\oplus f_0"]\ar[d,>->,swap,"1\oplus f_0"]&& D\oplus A   \ar[d,>->,"1\oplus f_0"] \\
C\oplus B \ar[rr,>->,swap,"g_0\oplus 1"]&& D\oplus B
\end{tikzcd}\quad,
\end{equation}
Plugging in the obvious permutations, compute:
$$\tau\circ (f_0\oplus g_0)=\left(\tau\circ (f_0\oplus 1)\right)\circ (1\oplus g_0)= (1\oplus f_0) \circ \left(\tau\circ  (1\oplus g_0)\right)= (g_0\oplus f_0)\circ \tau\,.$$
		\item  Examining part (i), notice the relevant $\M$-morphisms were all constructed via restricted pushouts. The claim then essentially follows from Fact~\ref{facts:restrictedpushouts}.
\item Proceed in stages. 
\subsubsection*{Step 0} Apply Lemma~\ref{lem:ADTsquares} to construct the LHS diagram below
\begin{equation}\label{eq:dirsum-step0}
\begin{tikzcd}
V'_0 \ar[d,{Circle[open]}->,swap,"q_{V'_{0}}"] \ar[dr,phantom,"\square"]& O \ar[l,>->]\ar[d,{Circle[open]}->] \ar[r,>->] \ar[dr,phantom,"\square"]& V_2 \ar[d,{Circle[open]}->,"h_1"] \\
V_0\oplus V'_0 \ar[d,{Circle[open]}->,swap,"f_1\oplus 1"]  \ar[dr,phantom,"\square"] & V_0 \ar[l,>->,"1\oplus V'_0",swap]\ar[d,{Circle[open]}->,"f_1"] \ar[r,>->,"h_0"] \ar[dr,phantom,"\square"] & V_1 \ar[d,{Circle[open]}->,"g_1"] \\
B\oplus V'_0 & B  \ar[l,>->,"1\oplus V'_0"] \ar[r,>->,swap,"g_0"]& C
\end{tikzcd}\quad\qquad \dsquaref{V_0\oplus V'_0}{V_1\oplus V'_0}{B\oplus V'_0}{C\oplus V'_0}{h_0\oplus 1}{f_1\oplus 1}{g_0\oplus 1}{\varv}\quad.
\end{equation}
By Axiom~(DS) we obtain the RHS distinguished square, with $\M$-morphisms labelled as in Lemma~\ref{lem:ADTsquares}. In addition, we may identify $\varv=g_1\oplus 1$ without loss of generality. 
[Why? Applying functor $k\colon\Ar_{\square}\E\to \Ar_{\triangle}\M$ (Helper Definition~\ref{def:help}) to this square yields 
$$\begin{tikzcd}
O \ar[d,{Circle[open]->}]\ar[rr,>->] \ar[drr,phantom,"\square"]&& V_1\oplus V'_0\ar[d,{Circle[open]->},"\varv"] \\
A \ar[rr,>->,swap,"(g_0\oplus 1)\circ (f_0\oplus V'_0)"] && C\oplus V'_0
\end{tikzcd}.$$
By appealing to either part (iv) of this Lemma or Fact~\ref{facts:restrictedpushouts} directly, notice that
$$(g_0\oplus 1)\circ (f_0\oplus V'_0)= (g_0\circ f_0)\oplus V'_0,$$
which corresponds to the kernel of $g_1\oplus 1$. Hence, by Axiom (K), $g_1\oplus 1$ and $\varv$ are equivalent up to codomain-preserving isomorphism, and the two morphisms can be identified (cf. Footnote~\ref{fn:INV}).] 

\subsubsection*{Step 1} Consider the diagram
\begin{equation}\label{eq:dirsum-step1}
\begin{tikzcd}
B\oplus V_0' \ar[d,{Circle[open]->},"1\oplus f'_1",swap] \ar[dr,phantom, "\square"] & B \ar[dr,phantom, "\square"]  \ar[d,{Circle[open]->},"q_B"] \ar[r,>->,"g_0"] \ar[l,>->,"1\oplus V'_0",swap] & C \ar[d,{Circle[open]->},"q_C"]\\
B\oplus B' & B\oplus A' \ar[r,>->,swap,"g_0\oplus 1"] \ar[l,>->, "1\oplus f'_0"] & C\oplus A'
\end{tikzcd}
\end{equation}
The LHS square is distinguished by Lemma~\ref{lem:ADTsquares}. For the RHS square, construct 
\begin{equation}
\begin{tikzcd}
& B \ar[d,{Circle[open]}->,swap,"q_B"] \ar[r,>->, dashed,blue,"u"] \ar[dr,phantom,xshift=0.5em, blue,"\square"]& C\ar[d,{Circle[open]}->,"q_C"]\\
A' \ar[r,>->,"p_A"] & B\oplus A'  \ar[r,>->,"g_0\oplus 1"]& C\oplus A'\\
O \ar[u,>->] \ar[r,>->]& B \ar[u,>->,"p_B"]  \ar[r,>->,"g_0"] &  C \ar[u,>->,swap,"p_C"] 
\end{tikzcd}\quad,
\end{equation}
with the obvious restricted pushouts in the bottom rectangle. Since $g_0\oplus 1$ induces a morphism 
	\[\left( A'\rtail B\oplus A'\right)\rtail \left(  A'\rtail C\oplus A'\right)\in \Ar_\Delta \M,\]
apply the functor $k^{-1}\colon \Ar_\triangle \M\to \Ar_{\square}\E$ to obtain the blue indicated distinguished square. Since distinguished squares commute (Definition~\ref{def:goodDC}), deduce from Convention~\ref{conv:restrPO-quotient} that $u=g_0$.

We can therefore apply Axiom (DS) to Diagram~\eqref{eq:dirsum-step1}, which yields the distinguished square
\begin{equation}
\dsquaref{B\oplus V'_0}{C\oplus V'_0}{B\oplus B'}{C\oplus B'}{g_0\oplus 1}{1\oplus f'_1}{w_0}{w_1}\quad.
\end{equation}
In particular, Fact~\ref{facts:restrictedpushouts} shows $w_0=g_0\oplus 1$; extending this, a similar argument as in Step 0 shows $w_1=1\oplus f'_1$.

\subsubsection*{Step 2} Assemble the distinguished squares from Steps 0 and 1 into the diagram:
\begin{equation}\label{eq:dirsum-step2a}
\begin{tikzcd}
O \ar[d,{Circle[open]->}] \ar[r,>->] \ar[dr,phantom,"\square"] & V_0 \ar[dr,phantom,"\square"] \ar[d,{Circle[open]->},"f_1",swap]\ar[r,>->] & V_0\oplus V'_0 \ar[dr,phantom,"\square"] \ar[d,{Circle[open]->},"f_1\oplus 1",swap] \ar[r,>->,,"h_0\oplus 1"] & V_1\oplus V'_0 \ar[d,{Circle[open]->}, "g_1\oplus 1"]\\
A \ar[d,{Circle[open]->}]\ar[r,>->,"f_0"] \ar[dr,phantom,blue,xshift=0.4em,"\square"] & B\ar[d,{Circle[open]->}] \ar[r,>->] \ar[dr,phantom,"\square"]  & B\oplus V'_0 \ar[d,{Circle[open]->},"1\oplus f'_1",swap]\ar[r,>->,"g_0\oplus 1"] \ar[dr,phantom,"\square"]  & C\oplus V'_0 \ar[d,{Circle[open]->},"1\oplus f'_1"]\\
A\oplus A' \ar[r,>->,swap,"f_0\oplus 1"] & B\oplus A' \ar[r,>->,swap,"1\oplus f'_0"] & B\oplus B' \ar[r,>->,swap,"g_0\oplus 1"] & C\oplus B'
\end{tikzcd}\qquad .
\end{equation}
All squares are distinguished: the blue square by the same reasoning as in Step~1, the others by Steps~0 – 1 and the setup.

Now recall from Equation~\eqref{eq:dirsum-M-morphism} that $f_0\oplus f'_0=(1\oplus f'_0)\circ(f_0\oplus 1)$. Since distinguished squares compose, Axiom~(K) identifies 
\[
(1\oplus f'_1)\circ(f_1\oplus 1)\;\;\sim\;\; f_1\oplus f'_1,
\]
up to codomain-preserving isomorphism. Thus, taking this to be the definition of $f_1\oplus f'_1$, we may replace the middle vertical in Diagram~\eqref{eq:dirsum-step2a} and construct the obvious diagram
\begin{equation}\label{eq:dirsum-step2b}
\begin{tikzcd}
V_0\oplus V_1' \ar[d,{Circle[open]->}] \ar[dr,phantom, "\square"] 
& V_0\oplus V_0' \ar[d,{Circle[open]->},"f_1\oplus f'_1"] \ar[dr,phantom,"\square",xshift=0.5em] 
\ar[l,>->,"1\oplus h'_0",swap] \ar[r,>->,"h_0\oplus 1"] 
& V_1\oplus V'_0 \ar[d,{Circle[open]->}]\\
B\oplus C' 
& B\oplus B' \ar[l,>->,"1\oplus g'_0"] \ar[r,>->,swap,"g_0\oplus 1"] 
& C\oplus B'
\end{tikzcd},
\end{equation}
to which Axiom (DS) applies. 
We thus obtain the desired distinguished square (in solid arrows):
\begin{equation}\label{eq:dirsum-step2c}
\begin{tikzcd}
\color{gray} O \ar[r,>->,gray,dashed]  \ar[d,{Circle[open]}->,dashed,gray] \ar[dr,phantom,yshift=-0.25em, xshift=0.3em,gray,"\square"] & V_0\oplus V'_0 \ar[d,{Circle[open]}->,"f_1\oplus f'_1"] \ar[r,>->,"h_0\oplus h'_0"] \ar[dr,phantom,"\square"]& V_1\oplus V'_1 \ar[d,{Circle[open]}->,"w_2"]\\
\color{gray} A\oplus A'  \ar[r,>->,gray,swap,dashed,"f_0\oplus f'_0"] & B\oplus B'  \ar[r,>->,"g_0\oplus g'_0",swap] & C\oplus C'
\end{tikzcd}\quad.
\end{equation}
It remains to determine $w_2$. Diagram~\eqref{eq:dirsum-step2c} shows $w_2$ is the quotient of 
$(g_0\oplus g'_0)\circ (f_0\oplus f'_0)$. But direct sums commute with composition of $\M$-morphisms, so
\[ (g_0\oplus g'_0)\circ(f_0\oplus f'_0) \;=\; (g_0\circ f_0)\oplus(g'_0\circ f'_0),\]
which corresponds to the formal kernel of $g_1\oplus g'_1$. In other words, $w_2$ and $g_1\oplus g'_1$ are equivalent up to codomain-preserving isomorphism, and so we may identify $w_2=g_1\oplus g'_1$.
	\end{enumerate}
\end{proof}

\section{Technical Lemmas \& Some 2-Simplices}

\subsection{Technical Facts about Sherman Loops}\label{app:Sherman Loops} To finish the proof of Theorem~\ref{thm:Sherman}, we will require the following three technical lemmas. The results here extend the arguments from \cite[\S 1 - 2]{ShermanAbelian} to the setting of pCGW categories.

\begin{lemma}\label{lem:SherLoopAdd} The sum of two Sherman loops is equivalent to a Sherman loop. Explicitly, consider two pairs of $\M$-morphisms 
	$$\alpha_i\colon A_i\to X_i\qquad ,\qquad \beta_i\colon B_i\to Y_i,\qquad \text{for $i=1,2$};$$
and isomorphisms 
$$\theta_i\colon A_i\oplus \frac{X_i}{A_i}\oplus Y_i\longrightarrow B_i\oplus \frac{Y_i}{B_i}\oplus X_i ,\qquad \text{for $i=1,2$}.$$
\underline{Then} 
$$G(\alpha_1,\beta_1,\theta_1)+ G(\alpha_2,\beta_2,\theta_2)=G\left(\alpha_1\oplus \alpha_2,\beta_1\oplus \beta_2,T_2(\theta_1\oplus \theta_2)T_1^{-1}\right)$$ 
in $K_1(\calC)$, whereby $T_1$ and $T_2$ are the canonical permutation isomorphisms
$$T_1\colon A_1\oplus \frac{X_1}{A_1}\oplus Y_1\oplus A_2\oplus \frac{X_2}{A_2}\oplus Y_2\to A_1\oplus A_2\oplus \frac{X_1}{A_1}\oplus \frac{X_2}{A_2}\oplus Y_1\oplus Y_2$$
$$$$
$$T_2\colon B_1\oplus \frac{Y_1}{B_1}\oplus X_1\oplus B_2\oplus \frac{Y_2}{B_2}\oplus X_2\to B_1\oplus B_2\oplus \frac{Y_1}{B_1}\oplus \frac{Y_2}{B_2}\oplus X_1\oplus X_2.$$
$$$$
\end{lemma}
\begin{proof} There are no surprises -- the argument is the same as in \cite[Prop. 1]{ShermanAbelian}. Apply the $H$-space structure of $G\calC$ to construct the loop corresponding to $G(\alpha_1,\beta_1,\theta_1)+G(\alpha_2,\beta_2,\theta_2)$. Notice it is almost identical to $G\left(\alpha_1\oplus \alpha_2,\beta_1\oplus \beta_2,T_2(\theta_1\oplus \theta_2)T_1^{-1}\right)$ -- the only difference being that the direct summands of certain vertices are arranged in a different order. To establish homotopy equivalence, use the obvious isomorphisms to permute these summands and construct a sequence of 2-simplices connecting the two loops. [One will need to justify why the obvious diagrams do in fact define 2-simplices, but this follows from distinguished squares being closed under isomorphisms (Definition~\ref{def:goodDC}).]
\end{proof}

\begin{lemma}\label{lem:SherLoopSplit} Let $\calC$ be a pCGW category whose exact squares all split. \underline{Then}, every element of $K_1(\calC)$ is equivalent to a loop of the form
\begin{equation}
G(A,\alpha):=\left(\begin{tikzcd}
(A,A) \ar[rr,"l(\alpha)"] && (A,A)\\
& (O,O) \ar[ul] \ar[ur]
\end{tikzcd}\right)
\end{equation}
where $l(\alpha)$ is the 1-simplex
\[l(\alpha):=\left(\, \dsquaref{O}{O}{A}{A}{ }{ }{1}{} \quad,\quad  \dsquaref{O}{O}{A}{A}{ }{ }{\alpha}{} \,\right)\]	
for some automorphism $(A,\alpha)\in\Aut(\calC)$.
\end{lemma}
\begin{proof} The proof combines an argument from \cite[\S 5]{GG} and \cite[Prop. 2]{ShermanAbelian}. Proceed in stages.
\subsubsection*{Step 1: Combinatorial Loops in $K_1(\calC)$} Suppose $z\in K_1(\calC)=\pi_1|G\calC^{\mathsf{o}}|$. By the simplicial approximation theorem, $z$ can be represented by a loop formed combinatorially from 1-simplices of $G\calC$
\begin{equation}\label{eq:z1stloop}
\ones{O}{O} \to \bullet \leftarrow \bullet \to \dots \leftarrow \bullet\to \bullet \leftarrow \ones{O}{O}
\end{equation}
where we draw the 1-simplices as arrows. 
Consider one of the configurations in Diagram~\eqref{eq:z1stloop}, e.g. 
$$\ones{M'}{L'}\leftarrow \ones{M}{L} \to \ones{M''}{L''}.$$
By Corollary~\ref{cor:RPtrick}, this is homotopic to 
$$\ones{M'}{L'}\to \ones{M'\star_M M''}{L'\star_L L''} \leftarrow \ones{M''}{L''}.$$
Applying this trick multiple times, we can deform Loop~\eqref{eq:z1stloop} into one of the form
\begin{equation}\label{eq:niceloop}
\begin{tikzcd}
\ones{O}{O} \ar[r] \ar[d] & \ones{M_0}{L_0} \ar[r]& \dots\ar[r] & \ones{M_{q-1}}{L_{q-1}} \ar[d]\\
\ones{M'_0}{L'_0} \ar[r]& \dots \ar[r] & \ones{M'_{q-1}}{L'_{q-1}} \ar[r] & \ones{M}{L} 
\end{tikzcd}
\end{equation}
	\subsubsection*{Step 2: The Base Case} Start by analysing the component
\begin{equation}\label{eq:BaseCase1Simp}
\ones{O}{O}\xrightarrow{l_0} \ones{M_0}{L_0}\xrightarrow{l_1} \ones{M_1}{L_1}
\end{equation}
	 of Loop~\eqref{eq:niceloop} in $K_1(\calC)$. Suppose $l_0$ is defined by the following pair of exact squares
\begin{equation}\label{eq:Loop1stC}
l_0:=\left(\dsquaref{O}{\widehat{M_0}}{O}{M_0}{ }{ }{ }{\eta_0}\quad,\quad \dsquaref{O}{\widehat{M_0}}{O}{L_0}{ }{ }{ }{\mu_0}\right),
\end{equation}
	with isomorphisms $\eta_0$ and $\mu_0$. Recall that all isomorphisms in $\calC$ can be regarded as either an $\M$ or an $\E$-morphism, related by the functor $$\varphi\colon \mathrm{iso}\M\to \mathrm{iso}\E,$$ 
	an isomorphism of categories. We can therefore define two 1-simplices
			\[l_0':=\left(\dsquaref{O}{\widehat{M_0}}{O}{\widehat{M_0}}{ }{ }{ }{1}\quad,\quad \dsquaref{O}{\widehat{M_0}}{O}{ \widehat{M_0}}{ }{ }{ }{1}\right) \quad \text{and} \quad l_0'':=\left(
\begin{tikzcd}
O \ar[rr, >->] \ar[d,{Circle[open]->}] \ar[drr,phantom,"\square"]&&  O \ar[d,{Circle[open]->}]   \\
\widehat{M_0} \ar[rr,>->,"\varphi^{-1}(\eta_0)"] && M_0
\end{tikzcd} \,,\,       
\begin{tikzcd}
O \ar[rr, >->] \ar[d,{Circle[open]->}] \ar[drr,phantom,"\square"]&&  O \ar[d,{Circle[open]->}]   \\
\widehat{M_0} \ar[rr,>->,"\varphi^{-1}(\mu_0)"] && L_0
\end{tikzcd}
\right),\]
which assemble into the following 2-simplex
\[\begin{tikzcd}
O \ar[r, >->] \ar[dr,phantom,"\square"]&  \widehat{M_0} \ar[dr,phantom,"\square"] \ar[r, >->,"\varphi^{-1}(\eta_0) "] & M_0\\
O \ar[u,{Circle[open]->}] \ar[r,>->] & \widehat{M_0}\ar[u,{Circle[open]->},"1"] \ar[r,>->,"1"]	\ar[dr,phantom,"\square"]& \widehat{M_0}\ar[u, {Circle[open]}->,swap,"\eta_0"] \\
& O \ar[u,{Circle[open]}->] \ar[r,>->] & O  \ar[u, {Circle[open]}->] 
\end{tikzcd}\quad 
\begin{tikzcd}
	O \ar[r, >->] \ar[dr,phantom,"\square"]&  \widehat{M_0} \ar[dr,phantom,"\square"] \ar[r, >->,"\varphi^{-1}(\mu_0) "] & L_0\\
	O \ar[u,{Circle[open]->}] \ar[r,>->] & \widehat{M_0}\ar[u,{Circle[open]->},"1"] \ar[r,>->,"1"]	\ar[dr,phantom,"\square"]& \widehat{M_0}\ar[u, {Circle[open]}->,swap,"\mu_0"] \\
	& O \ar[u,{Circle[open]}->] \ar[r,>->] & O  \ar[u, {Circle[open]}->] 
	\end{tikzcd}.\]
Notice the top right squares are distinguished since distinguished squares are closed under isomorphisms. We then assemble the following diagram
\begin{equation}
\begin{tikzcd}
&& \ones{ \widehat{M_0}}{ \widehat{M_0}} \ar[d,"l''_0"] \ar[drr,blue,"l_1\circ l''_0"]\\
\ones{O}{O} \ar[urr,blue,"l'_0"] \ar[rr,red,"l_0"] && \ones{M_0}{L_0} \ar[rr,red,"l_1"]&& \ones{M_1}{L_1}
\end{tikzcd}.
\end{equation}
One easily checks the RHS triangle also bounds a 2-simplex, which implies the red path of 1-simplices is homotopic to the blue path. Hence, assume without loss of generality that both $\mu_0$ and $\eta_0$ of Equation~\eqref{eq:Loop1stC} are the identity $1\colon M_0\to M_0$.

\subsubsection*{Step 3: The Inductive Step} Proceeding along Loop~\eqref{eq:BaseCase1Simp}, consider $l_1$, to be represented as\footnote{To improve readability, quotients here are denoted suggestively as e.g. $\frac{M_1}{M_0}$ but the proof works for any quotient representative (not necessarily the canonical one from Axiom (K)). This is the same abuse of notation made in Corollary~\ref{cor:RPtrick}.}
\[l_1:=\left(\dsquaref{O}{\frac{M_1}{M_0}}{M_0}{M_1}{ }{ }{ \eta_1}{\eta'_1}\quad,\quad \dsquaref{O}{\frac{M_1}{M_0}}{M_0}{L_1}{ }{ }{\mu_1}{\mu'_1}\right).\]
Since all exact squares split by hypothesis, this yields isomorphisms
$$\psi_M\colon   M_1\rtail M_0\oplus \frac{M_1}{M_0}\qquad \text{and} \qquad \psi_L\colon L_1\rtail M_0\oplus \frac{M_1}{M_0},$$
which we shall consider as $\M$-morphisms. This defines the following distinguished squares
\[\begin{tikzcd}
O \ar[r, >->] \ar[dr,phantom,"\square"]\ar[d,{Circle[open]->}] &  \frac{M_1}{M_0} \ar[dr,phantom,"\square"] \ar[r, >->,"1"] \ar[d,{Circle[open]->}," \eta'_1"]&  \frac{M_1}{M_0}\ar[d, {Circle[open]}->,"\varphi(\psi_M)\circ \eta'_1"]   \\
M_0 \ar[r,>->,"\eta_1"] &  M_1  \ar[r,>->,"\psi_M"]&   M_0\oplus \frac{M_1}{M_0} 
\end{tikzcd}\qquad \begin{tikzcd}
O \ar[r, >->] \ar[dr,phantom,"\square"]\ar[d,{Circle[open]->}] &  \frac{M_1}{M_0} \ar[dr,phantom,"\square"] \ar[r, >->,"1"] \ar[d,{Circle[open]->}," \mu'_1"]&   \frac{M_1}{M_0}\ar[d, {Circle[open]}->,"\varphi(\psi_L)\circ \mu'_1"]   \\
M_0 \ar[r,>->,"\mu_1"] &  L_1  \ar[r,>->,"\psi_L"]&   M_0\oplus \frac{M_1}{M_0} 
\end{tikzcd}\]
By Axiom (I), the RHS squares commute and are thus distinguished. Hence, the horizontal compositions define two exact squares. In fact, we can say more. To ease notation, denote 
$$v:=\psi_M\circ\eta_1 \qquad \text{and}\qquad v':=\varphi(\psi_M)\circ \eta'_1,$$
$$w:=\psi_L\circ\mu_1 \qquad \text{and}\qquad w':=\varphi(\psi_L)\circ \mu'_1,$$
and denote
\[\dsquaref{O}{\frac{M_1}{M_0}}{M_0}{M_0\oplus \frac{M_1}{M_0}}{}{ }{ p}{q}\]
to be the canonical direct sum square of $M_0\oplus \frac{M_1}{M_0}$.  Since both exact squares in $l_1$ are split, we know that
\begin{equation}\label{eq:IndCaseIdentities}
v=  p =w
\end{equation}
$$ v'= q = w'.$$
Leverage these identities to construct the following 2-simplex
\begin{equation}\label{eq:IndCase1}
\begin{tikzcd}
O \ar[r, >->] \ar[dr,phantom,"\square"]&  M_0 \ar[dr,phantom,blue,"\square"] \ar[r, >->,"\eta_1"] & M_1\\
O \ar[u,{Circle[open]->}] \ar[r,>->] &  M_0\ar[u,{Circle[open]->}," 1"] \ar[r,>->,"v"] \ar[dr,phantom,"\square"]&   M_0\oplus \frac{M_1}{M_0} \ar[u, {Circle[open]}->,swap,"\varphi(\psi^{-1}_M)"]  \\
&O \ar[r,>->] \ar[u, {Circle[open]}->] & \frac{M_1}{M_0} \ar[u, {Circle[open]}->,swap,"q"]
\end{tikzcd} \qquad \begin{tikzcd}
O \ar[r, >->] \ar[dr,phantom,"\square"]&  M_0 \ar[dr,phantom,red,"\square"] \ar[r, >->,"\mu_1"] & L_1\\
O \ar[u,{Circle[open]->}] \ar[r,>->] &  M_0\ar[u,{Circle[open]->}," 1"] \ar[r,>->,"v"] \ar[dr,phantom,"\square"]&   M_0\oplus \frac{M_1}{M_0} \ar[u, {Circle[open]}->,swap,"\varphi(\psi_L^{-1})"]  \\
&O \ar[r,>->] \ar[u, {Circle[open]}->] & \frac{M_1}{M_0} \ar[u, {Circle[open]}->,swap,"q"]
\end{tikzcd} \quad. 
\end{equation} 
To see why the red-indicated square is distinguished, notice:
\begin{itemize}
	\item\textbf{Case 1.} Suppose $\E$ is a subcategory of $\calC$. Then 
$$	\varphi(\psi_L^{-1} )\circ v=\psi^{-1}_L\circ v=\mu_1 \qquad \text{in $\calC $}$$
	if and only if
	$$ v = \psi_L\circ \mu_1,$$
which holds by Identity~\eqref{eq:IndCaseIdentities}.
	\item \textbf{Case 2.} Suppose  $\E^{\opp}$ is a subcategory of $\calC$, and so the arrows reverse. Applying Identity~\eqref{eq:IndCaseIdentities} once more, deduce
	$$	v= \psi_L\circ \mu_1 = \varphi(\psi_L^{-1})\circ \mu_1 .$$

\end{itemize}
The claim then follows from distinguished squares being closed under isomorphisms. The case for the blue-indicated square is analogous. 

Having checked that Diagram~\eqref{eq:IndCase1} defines a 2-simplex, the gears line up and the inductive argument falls into place. First note that Diagram~\eqref{eq:IndCase1} defines a homotopy between 
$$\ones{O}{O}\xrightarrow{l_0} \ones{M_0}{L_0}\xrightarrow{l_1} \ones{M_1}{L_1}$$
and a 1-simplex
$$l'_1\colon \ones{O}{O}\to \ones{M_1}{L_1}$$
whereby
\[l'_1:=\left(\dsquaref{O}{M_0\oplus \frac{M_1}{M_0}}{O}{M_1}{ }{ }{ }{}\quad,\quad \dsquaref{O}{M_0\oplus \frac{M_1}{M_0}}{O}{L_1}{ }{ }{}{}\right).\]
Hence, the initial segment of Loop~\eqref{eq:niceloop} is homotopic to
$$\ones{O}{O}\xrightarrow{l'_1}\ones{M_1}{L_1}\xrightarrow{l_2}\ones{M_2}{L_2},$$
deleting a term. Then, apply the Base Case argument (Step 2) to justify presenting $l'_1$ as
\[l'_1:=\left(\dsquaref{O}{M_1}{O}{M_1}{ }{ }{ }{1}\quad,\quad \dsquaref{O}{\ M_1}{O}{M_1}{ }{ }{}{1}\right),\]
which sets up our inductive step again. Keep going for the rest of Loop~\eqref{eq:niceloop} on both sides, until we finally obtain a loop of the form
\begin{equation}\label{eq:IndCaseFinalLoop}
\ones{O}{O} \xrightarrow{\kappa} (M,L) \xleftarrow{\gamma}  \ones{O}{O},
\end{equation}
where 
\[\kappa:=\left(\dsquaref{O}{\widehat{M}}{O}{M}{ }{ }{ }{\kappa_0}\quad,\quad \dsquaref{O}{\widehat{M}}{O}{L}{ }{ }{}{\kappa_1}\right) \qquad\text{and}\qquad \gamma:=\left(\dsquaref{O}{\widehat{N}}{O}{M}{ }{ }{ }{\gamma_0}\quad,\quad \dsquaref{O}{\widehat{N}}{O}{L}{ }{ }{}{\gamma_1}\right)\]
Notice we can no longer apply the Base Case argument to simplify $\kappa$ or $\gamma$ since the arrows of Diagram~\eqref{eq:IndCaseFinalLoop} are in the wrong direction.
	\subsubsection*{Step 4: Finish}	A technical observation: both $\kappa$ and $\gamma$ define isomorphisms $L\xrightarrow{\cong} M$  in $\calC$ but the presentation will differ depending on whether $\E^\opp$ or $\E$ is a subcategory of $\calC$.
	\begin{itemize}
		\item[] \textbf{Case 1:} {\em $\E\subseteq \calC$.}  In which case, define $\omega:= \kappa_0\circ \kappa_1^{-1}$  and $ \lambda:=\gamma_0\circ \gamma_1^{-1}$ in $\calC$
		\item[] \textbf{Case 2:} {\em $\E^\opp \subseteq \calC$.} In which case, define $\omega:= \kappa_0^{-1}\circ \kappa_1$ and $\lambda:=\gamma^{-1}_0\circ \gamma_1$ in $\calC$.
	\end{itemize}
By Axiom (I), we may regard $\omega$ and $\lambda$ as $\M$-morphisms as well. We now construct the obvious diagram
\begin{equation}\label{eq:IndFINAL}
\begin{tikzcd}
&\ones{M}{M } \ar[r,blue,"g"]& \ones{M}{M}  \\
\ones{O}{O} \ar[ur,"f_0",blue]\ar[r,"\kappa",red] & \ones{M}{L} \ar[u,"f_1"]\ar[ur,"f_2"] & \ar[l,swap,"\gamma",red] \ones{O}{O}\ar[u, swap,"f_0",blue] 
\end{tikzcd}
\end{equation}	
whereby

\[g:= \left(\dsquaref{O}{O}{M}{M}{}{}{1}{} \quad,\quad \dsquaref{O}{O}{M}{M}{}{}{\lambda\circ\omega^{-1}}{ } \right)\qquad f_0:=\left(\dsquaref{O}{M}{O}{M}{}{}{}{1} \quad,\quad \dsquaref{O}{M}{O}{M}{}{}{}{1} \right).\]
\[f_1:=\left(\dsquaref{O}{O}{M}{M}{}{}{1}{} \quad,\quad \dsquaref{O}{O}{L}{M}{}{}{\omega}{} \right)\qquad f_2:=\left(\dsquaref{O}{O}{M}{M}{}{}{1}{} \quad,\quad \dsquaref{O}{O}{L}{M}{}{}{\lambda}{} \right) . \]
In both cases ($\E$ or $\E^\opp\subseteq \calC$), it is easy to check that the triangles of Diagram~\eqref{eq:IndFINAL} bound 2-simplices, e.g. 
\[f_1\kappa = f_0 \qquad \small{\left(\begin{tikzcd}
O \ar[r, >->," "] & M \ar[dr,phantom,"\square"] \ar[r,>->,"1"]& M \\
& \widehat{M}\ar[u,{Circle[open]->},"\kappa_0"] \ar[r,>->,swap,"\varphi^{-1}(\kappa_0)"]	 & M\ar[u, {Circle[open]}->,swap,"1"]  \\
&   & O \ar[u, {Circle[open]}->]  
\end{tikzcd}\qquad 
\begin{tikzcd}
O \ar[r, >->," "] & L \ar[dr,phantom,"\square"] \ar[r,>->,"\omega"]& M \\
& \widehat{M}\ar[u,{Circle[open]->},"\kappa_1"] \ar[r,>->,swap,"\varphi^{-1}(\kappa_0)"]	 & M \ar[u, {Circle[open]}->,swap,"1"]  \\
&   & O \ar[u, {Circle[open]}->]  
\end{tikzcd}\right)}.\]
Conclude that the red loop in Diagram~\eqref{eq:IndFINAL} is homotopic to the blue loop. Notice the blue loop is precisely of the form $G(A,\alpha)$ as claimed in lemma statement, with $A:=M$ and $\alpha:=\lambda\circ \omega^{-1}$.
\end{proof}

\begin{lemma}\label{lem:SherSplitSher} The automorphism loop $G(A,\alpha)$ in Lemma~\ref{lem:SherLoopSplit} is equivalent to a Sherman Loop.
\end{lemma}
\begin{proof} Let $p_A\colon A\rtail A\oplus A$ be the canonical $\M$-morphism of a direct sum square, and $\tau_A\colon A\oplus A\rtail A\oplus A$ be the permutation isomorphism swapping components. Define the following loop 
\begin{equation}\label{eq:SherLoopSplit}
\ones{O}{O} \to \ones{A}{A} \xrightarrow{\iota_\alpha} \ones{A\oplus A}{A\oplus A} \xrightarrow{l_{\tau}}  \ones{A\oplus A}{A\oplus A} \leftarrow  \ones{O}{O}
\end{equation}
where 
$$\iota_\alpha:=\left(\dsquaref{O}{A}{A}{A\oplus A}{}{ }{p_A}{q_A}\,,\,\dsquaref{O}{A}{A}{A\oplus A}{}{ }{p_A\circ \alpha}{q_A} \right) \quad l_\tau:=\left(\dsquaref{O}{O}{A\oplus A}{A\oplus A}{}{ }{\tau_A}{ }\,,\,\dsquaref{O}{O}{A\oplus A}{A\oplus A}{}{ }{\tau_A}{ } \right)  $$ 
This is a Sherman Loop $G(\alpha,0,\tau_A)$, where $0$ denotes the $\M$-morphism $O\rtail A$.
	
To show $G(A,\alpha)\sim G(\alpha,0,\tau_A)$ in $\pi_1|G\calC^{\mathsf{o}}|$, we first modify $G(A,\alpha)$ in sensible ways that respects its homotopy class. An initial observation: the following loop 
\begin{equation}\label{eq:SherLoopSplit3}
\ones{O}{O} \to \ones{A}{A} \xrightarrow{\iota_\alpha} \ones{A\oplus A}{A\oplus A} \leftarrow  \ones{O}{O}.
\end{equation}
is homotopic to $G(A,\alpha)$, since the following diagram
$$\begin{tikzcd}
\ones{A}{A} \ar[rr,"\iota_\alpha"] \ar[dr,"\alpha"] && \ones{A\oplus A}{A\oplus A}\\
&\ones{A}{A} \ar[ur,"\iota_{A}"]
\end{tikzcd}$$
bounds a 2-simplex, where the 1-simplex $\iota_A$ denotes two copies of the direct sum square $A\oplus A$ (and is thus diagonal). Further, the Loop~\eqref{eq:SherLoopSplit3} is homotopic to 
\begin{equation}\label{eq:SherLoopSplit2}
\ones{O}{O} \to \ones{A}{A} \xrightarrow{\iota_\alpha} \ones{A\oplus A}{A\oplus A} \xrightarrow{1_{A\oplus A}}  \ones{A\oplus A}{A\oplus A} \leftarrow  \ones{O}{O},
\end{equation}
since all we did was insert a degenerate 1-simplex $1_{A\oplus A}$. It remains to show that $G(\alpha,0,\tau_A)$ is homotopic to Loop~\eqref{eq:SherLoopSplit2}. But this follows from noting that the triangles in the diagram below bound 2-simplices
$$\begin{tikzcd}
\ones{A\oplus A}{A\oplus A} \ar[dr,swap,"l_\tau"]  \ar[r,"1_{A\oplus A}"] & \ones{A\oplus A}{A\oplus A} & \ones{O}{O} \ar[dl] \ar[l]\\
& \ones{A\oplus A}{A\oplus A} \ar[u,"l_\tau"]
\end{tikzcd}\qquad .$$
	\vspace{-0.8em}
\end{proof}

\subsection{Explicit Descriptions of 2-Simplices}\label{sec:Nen2simp} 

This section explicitly constructs the two key homotopies claimed by Lemma~\ref{lem:Nen4}, necessary to finish the proof of Theorem~\ref{thm:NenashevDSES}

\begin{claim}\label{claim:NenL1} Loop~\eqref{eq:NenashevL1} is homotopic to loop $L$.
\end{claim}
\begin{proof} Recall: in order to establish that the two loops are homotopic, it suffices to show that the indicated triangles of the Diagram~\eqref{eq:NenBigL1} are boundaries of 2-simplices. We describe the 2-simplices explicitly below.
	\begin{itemize}
		\item Triangle (1). Consider
		\vspace{-0.1em}
		\begin{equation*}
\begin{tikzcd}
		A\oplus A'\ar[r, >->,"f_0"] & P \ar[dr,phantom,"\square"] \ar[r, >->,"1\oplus C\oplus C'"] & P\oplus C\oplus C'\\
		&	C\oplus C' \ar[r, >->] \ar[u, {Circle[open]}->,"g_0"]& C\oplus C'\oplus C\oplus C' \ar[u, {Circle[open]}->,"g_0\oplus 1"] \\
		&& C\oplus C' \ar[u, {Circle[open]}->] 
		\end{tikzcd}\quad \begin{tikzcd}
		A\oplus A'\ar[r, >->,"f_1"] & Q \ar[r, >->,"h_t"]  \ar[dr,phantom,"\square"]  & V\\
		&	C\oplus C' \ar[r, >->]\ar[u, {Circle[open]}->,"g_1"]& C\oplus C' \oplus C\oplus  C' \ar[u, {Circle[open]}->,"j"] \\
		&& C\oplus C' \ar[u, {Circle[open]}->] 
		\end{tikzcd}
		\end{equation*}
		[Why is this a 2-simplex? The indicated square on the right diagram is the distinguished square $t$ of Lemma~\ref{lem:Nen3}. The left indicated square is distinguished by Lemma~\ref{lem:ADTsquares} (i). The rest are obvious.]
	
		\item Triangle (2).
				\vspace{-0.1em}
		\begin{equation*}
\begin{tikzcd}
		A\oplus A'\ar[r, >->,"f_1"] & Q \ar[dr,phantom,"\square"] \ar[r, >->,"1\oplus C\oplus  C'"] & Q\oplus C\oplus C'\\
		&	C\oplus C' \ar[r, >->] \ar[u, {Circle[open]}->,"g_1"]& C\oplus C'\oplus C\oplus  C' \ar[u, {Circle[open]}->,"g_1\oplus 1"] \\
		&& C\oplus C' \ar[u, {Circle[open]}->] 
		\end{tikzcd}\quad \begin{tikzcd}
		A\oplus A'\ar[r, >->,"f_1"] & Q \ar[r, >->,"h_t"]  \ar[dr,phantom,"\square"]  & V\\
		&	C\oplus C' \ar[r, >->]\ar[u, {Circle[open]}->,"g_1"]& C\oplus C' \oplus C\oplus C' \ar[u, {Circle[open]}->,"j"] \\
		&& C\oplus C' \ar[u, {Circle[open]}->] 
		\end{tikzcd}
		\end{equation*}
[Why is this a 2-simplex? Analogous to Triangle (1).]
		\item Triangle (3).
		\begin{equation*}
		\begin{tikzcd}
		P\ar[r, >->,"1\oplus C\oplus C'"] & P\oplus C\oplus  C' \ar[dr,phantom,"\square"] \ar[r, >->,"\theta\oplus 1"] & Q\oplus C\oplus  C'\\
		&	 C\oplus C' \ar[r, >->,"1"] \ar[u, {Circle[open]}->,"P\oplus 1"]&  C\oplus C' \ar[u, {Circle[open]}->,"Q\oplus 1"] \\
		&& O \ar[u, {Circle[open]}->] 
		\end{tikzcd}\quad \begin{tikzcd}
		Q\ar[r, >->,"h_t "] & V \ar[r, >->,"1"]  \ar[dr,phantom,"\square"]  & V\\
		&	C\oplus C' \ar[r, >->,"1"]\ar[u, {Circle[open]}->,"k_t"]&  C\oplus C' \ar[u, {Circle[open]}->,"k_t	"] \\
		&& O \ar[u, {Circle[open]}->] 
		\end{tikzcd}
		\end{equation*}
		[To show that this is a 2-simplex, notice Lemma~\ref{lem:Nen3} already verified that
		\[t':=\left(\dsquaref{O}{C\oplus C'}{Q}{V}{}{}{h_t}{k_t}\right)\]
		is a distinguished square. The rest follows from Lemma~\ref{lem:ADTsquares} (i).]
		\item Triangle (4).
		
		\begin{equation*}
		\begin{tikzcd}
		P\ar[r, >->,"\theta"] & Q \ar[dr,phantom,"\square"] \ar[r, >->,"1\oplus C\oplus C'"] & Q\oplus C\oplus  C' \\
		&	O \ar[r, >->] \ar[u, {Circle[open]}->]& C\oplus C' \ar[u, {Circle[open]}->,"Q\oplus 1"] \\
		&& C\oplus C'  \ar[u, {Circle[open]}->] 
		\end{tikzcd}\quad \begin{tikzcd}
		Q\ar[r, >->,"1"] & Q \ar[r, >->,"h_t"]  \ar[dr,phantom,"\square"]  & V\\
		&	O\ar[r, >-> ]\ar[u, {Circle[open]}->]& C\oplus C' \ar[u, {Circle[open]}->,"k_t	"] \\
		&& C\oplus C' \ar[u, {Circle[open]}->] 
		\end{tikzcd}
		\end{equation*}
[Why is this a 2-simplex? Immediate from Lemma~\ref{lem:ADTsquares}.]
	\end{itemize}	
	
\end{proof}

\begin{claim}\label{claim:NenL2} Loop~\eqref{eq:NenashevL2} is homotopic to loop $L$.
\end{claim}
\begin{proof} By analogy with Claim~\ref{claim:NenL1}, all indicated triangles in Diagram~\eqref{eq:NenBigL2} can be shown to bound 2-simplices. The only subtlety lies in verifying that Triangles (1') and (2') do in fact bound the 2-simplices below, but this follows from $V$ being a restricted pushout (see remarks above Diagram~\eqref{eq:NenBigL1}).

	\begin{itemize}
		\item Triangle (1'). 
		\begin{equation*}
 \begin{tikzcd}
		A\oplus A'\ar[r, >->,"f_0"] & P \ar[dr,phantom,"\square"] \ar[r, >->,"1\oplus C\oplus  C'"] & P\oplus C\oplus  C'\\
		&	C\oplus C' \ar[r, >->] \ar[u, {Circle[open]}->,"g_0"]& C\oplus C' \oplus C\oplus C' \ar[u, {Circle[open]}->,"g_0\oplus 1"] \\
		&& C\oplus C' \ar[u, {Circle[open]}->] 
		\end{tikzcd}\quad \begin{tikzcd}
		A\oplus A'\ar[r, >->,"\alpha\oplus \alpha'"] & B\oplus B' \ar[r, >->,"h_u"]  \ar[dr,phantom,"\square"]  & V\\
		&	C\oplus C' \ar[r, >->]\ar[u, {Circle[open]}->,"\delta\oplus\delta'"]& C\oplus C' \oplus C\oplus C' \ar[u, {Circle[open]}->,"j"] \\
		&& C\oplus C' \ar[u, {Circle[open]}->] 
		\end{tikzcd}
		\end{equation*}
		\item Triangle (2').		
		\begin{equation*}
		\!\!\!\!\!\!\!\!\!\!\!\!\!\!\!\!\!\!\!	\begin{tikzcd}
		A\oplus A'\ar[r, >->,"f_1"] & Q \ar[dr,phantom,"\square"] \ar[r, >->,"1\oplus C\oplus C'"] & Q\oplus C\oplus C'\\
		&	C\oplus C' \ar[r, >->] \ar[u, {Circle[open]}->,"g_1"]& C\oplus C' \oplus C\oplus C' \ar[u, {Circle[open]}->,"g_1\oplus 1"] \\
		&& C' \ar[u, {Circle[open]}->] 
		\end{tikzcd}\quad \begin{tikzcd}
		A\oplus A'\ar[r, >->,"\alpha\oplus \alpha'"] & B\oplus B' \ar[r, >->,"h_u"]  \ar[dr,phantom,"\square"]  & V\\
		&	C\oplus C' \ar[r, >->]\ar[u, {Circle[open]}->,"\delta\oplus \delta'"]& C\oplus C' \oplus C\oplus C' \ar[u, {Circle[open]}->,"j"] \\
		&& C' \ar[u, {Circle[open]}->] 
		\end{tikzcd}
		\end{equation*}
		\item Triangle (3').
		\begin{equation*}
		\begin{tikzcd}
		P\ar[r, >->,"1\oplus C\oplus C'"] & P\oplus C\oplus C' \ar[dr,phantom,"\square"] \ar[r, >->,"\theta\oplus 1"] & Q\oplus C \oplus  C'\\
		&	 C\oplus C' \ar[r, >->,"1"] \ar[u, {Circle[open]}->,"P\oplus 1"]&  C\oplus C' \ar[u, {Circle[open]}->,"Q\oplus 1"] \\
		&& O \ar[u, {Circle[open]}->] 
		\end{tikzcd}\quad \begin{tikzcd}
		B\oplus B'\ar[r, >->,"h_u "] & V \ar[r, >->,"1"]  \ar[dr,phantom,"\square"]  & V\\
		&	C\oplus C' \ar[r, >->,"1"]\ar[u, {Circle[open]}->,"k_u"]&  C\oplus C' \ar[u, {Circle[open]}->,"k_u"] \\
		&& O \ar[u, {Circle[open]}->] 
		\end{tikzcd}
		\end{equation*}
		\item Triangle (4').
		
		\begin{equation*}
		\begin{tikzcd}
		P\ar[r, >->,"\theta"] & Q \ar[dr,phantom,"\square"] \ar[r, >->,"1\oplus C\oplus  C'"] & Q\oplus C\oplus C' \\
		&	O \ar[r, >->] \ar[u, {Circle[open]}->]& C\oplus C' \ar[u, {Circle[open]}->,"Q\oplus 1"] \\
		&& C\oplus C'  \ar[u, {Circle[open]}->] 
		\end{tikzcd}\quad \begin{tikzcd}
		B\oplus B' \ar[r, >->,"1"] & B\oplus B' \ar[r, >->,"h_u"]  \ar[dr,phantom,"\square"]  & V\\
		&	O\ar[r, >-> ]\ar[u, {Circle[open]}->]& C\oplus C' \ar[u, {Circle[open]}->,"k_u	"] \\
		&& C\oplus  C' \ar[u, {Circle[open]}->] 
		\end{tikzcd}
		\end{equation*}
	\end{itemize}
\end{proof}

\subsection{Permutation Lemma} This section isolates a permutation lemma that clarifies how direct sums interact with reordering; this will play a key role in the study of admissible triples in Section~\ref{sec:admtriples}.

\begin{lemma}[Permutation Lemma]\label{lem:permute} Setup: 
	\begin{itemize}
		\item Let $X_1,\dots, X_m$ be 1-simplices in $G\calC$ of the form
		\[ X_j:=\left( \,\,\dsquaref{O}{Z}{A_j}{B_j}{ }{ }{\alpha_j}{\beta_j}\quad,\quad \dsquaref{O}{Z}{A'_j}{B'_j}{ }{ }{\alpha'_j}{\beta'_j}\,\,\right), \] 	
		so that all have the same quotient object $Z$. 
		\item For any permutation $\sigma\in S_m$, define
		\[ \widetilde{X}_{\sigma(j)}:=\left( \,\,\dsquaref{O}{Z}{A_j}{B_j}{ }{ }{\alpha_j}{\beta_j}\quad,\quad \dsquaref{O}{Z}{A'_{\sigma(j)}}{B'_{\sigma(j)}}{ }{ }{\alpha'_{\sigma(j)}}{\beta'_{\sigma(j)}}\,\,\right),\]
		i.e. the LHS square is fixed while the RHS is permuted by $\sigma$.
	\end{itemize}	
	\underline{Then}, there is a homotopy between
	$$X_1\oplus \dots \oplus X_m \quad \text{and} \quad \widetilde{X}_{\sigma(1)}\oplus \dots \oplus \widetilde{X}_{\sigma(m)}.$$
\end{lemma}
\begin{proof} We first treat the case for adjacent swaps, before extending to arbitrary permutations.

	\subsubsection*{Step 0: Adjacent Swaps} Assume $m=2$ and $\sigma=(12)$. Let $\calW\subseteq  G\calC\times G\calC$ be the simplicial subset whose higher $n$-simplices are pairs 
	$$(\alpha_1,\alpha_2)\in G\calC\times G\calC ([n])$$
sharing the same 0th-face.\footnote{In other words, $\calW:=GG\calC$, the $G$-construction applied twice to $S\calC$.} 
Explicitly, for any $n>0$, if 
	$$\alpha_r=\begin{pmatrix} O \rtail \doubleunderline{A_{r,0}\rtail \dots \rtail A_{r,n}}\\ O \rtail A'_{r,0}\rtail \dots \rtail A'_{r,n} \end{pmatrix}\qquad r\in \{1,2\},$$
	then we require all quotient index diagrams to agree:
	$$	\begin{tikzcd}
\frac{	A'_{1,j}}{A'_{1,i}} \ar[r,>->,""] \ar[dr,phantom,"\square"]& \frac{	A'_{1,j}}{A'_{1,l}} 
	\\
\frac{	A'_{1,k}}{A'_{1,i}} \ar[r,>->] \ar[u, {Circle[open]}->] & \frac{	A'_{1,k}}{A'_{1,l}} \ar[u, {Circle[open]}->] 	\end{tikzcd} = \begin{tikzcd} \frac{	A_{1,j}}{A_{1,i}} \ar[r,>->,""] \ar[dr,phantom,"\square"]& \frac{	A_{1,j}}{A_{1,l}} 
\\
\frac{	A_{1,k}}{A_{1,i}} \ar[r,>->] \ar[u, {Circle[open]}->] & \frac{	A_{1,k}}{A_{1,l}} \ar[u, {Circle[open]}->] 	\end{tikzcd} =\begin{tikzcd}
\frac{	A_{2,j}}{A_{2,i}} \ar[r,>->,""] \ar[dr,phantom,"\square"]& \frac{	A_{2,j}}{A_{2,l}} 
\\
\frac{	A_{2,k}}{A_{2,i}} \ar[r,>->] \ar[u, {Circle[open]}->] & \frac{	A_{2,k}}{A_{2,l}} \ar[u, {Circle[open]}->] 	\end{tikzcd} = \begin{tikzcd}
\frac{	A’_{2,j}}{A’_{2,i}} \ar[r,>->,""] \ar[dr,phantom,"\square"]& \frac{	A'_{2,j}}{A'_{2,l}} 
\\
\frac{	A’_{2,k}}{A’_{2,i}} \ar[r,>->] \ar[u, {Circle[open]}->] & \frac{	A'_{2,k}}{A'_{2,l}} \ar[u, {Circle[open]}->] 	\end{tikzcd},$$
for all $0\leq i\leq  l\leq  k\leq  j\leq n$.\footnote{Use the obvious identifications whenever the indices are equal -- e.g. when $j=l$, set $A'_{1,j}/A'_{1,l}=O$.} In particular, any pair of 1-simplices $(X_1,X_2)\in G\calC\times G\calC$ with common quotient belongs to $\calW$.

	We now set up the main construction of the proof. For a monotone map $\beta\colon [n] \to [1]$, set
	$$i:=\max\{j\mid \beta(j)=0\},$$
	with the convention $i=-1$ if $\beta$ is constantly 1. Then, define the map
	\begin{align*}
	h\colon \calW \times [1] ([n]) & \longrightarrow G\calC ([n]) \\
	\left(\left(\alpha_1,\alpha_2\right)\,,\,\beta\right) &\longmapsto \begin{pmatrix} O \rtail \doubleunderline{A_{1,0}\oplus A_{2,0}\rtail \dots \rtail A_{1,i}\oplus A_{2,i} \rtail A_{1,i+1}\oplus A_{2,i+1}\rtail \dots \rtail A_{1,n}\oplus A_{2,n}}\\ O \rtail A'_{2,0}\oplus A'_{1,0}\rtail \dots \rtail A'_{2,i}\oplus A'_{1,i} \rtail A'_{1,i+1}\oplus A'_{2,i+1}\rtail \dots \rtail A'_{1,n}\oplus A'_{2,n} \end{pmatrix},
	\end{align*}
	where $(\alpha_1,\alpha_2)$ is an $n$-simplex of $\calW$.
	
	In English: the map $h$ defines a simplicial homotopy between $\alpha_1+\alpha_2$ and the permuted sum. When $i=-1$, $h$ gives $\alpha_1+\alpha_2$. When $i\geq 0$, $h$ modifies $\alpha_1+\alpha_2$ by the following rule:
	\begin{itemize}
		\item[$\diamond$] \textbf{Stages $j\leq i$.} Apply the obvious isomorphisms to permute the summands of the relevant $\M$ and $\E$-morphisms (the quotient index diagrams are left unchanged);
		\item[$\diamond$] \textbf{Stages $j\geq i+1$.} Leave the $\M$-morphisms as is, and permute the $\E$-morphisms accordingly to make a simplex in $G\calC$. 
	\end{itemize}
Using Lemma~\ref{lem:DirectSum}, one checks that $h$ is a simplicial map. In particular, $h$ defines a homotopy 
		$$X_1\oplus X_2 \sim  \widetilde{X}_{\sigma(1)} \oplus \widetilde{X}_{\sigma(2)},$$
		where $X_1,X_2$ are 1-simplices with common quotient $Z$.
	\subsubsection*{Step 1: Arbitrary Permutations} Let $m$ be any positive integer. The case for $m=1$ is trivial. As for $m>1$, any permutation $\sigma\in S_m$ is a finite product of adjacent transpositions. For each adjacent transposition, say $(k\,\, k+1)$ for some $0\leq k <m$, apply Step 0 to get 
		$$X_k\oplus X_{k+1} \sim  \widetilde{X}_{\sigma(k)} \oplus \widetilde{X}_{\sigma(k+1)},$$
		before extending to get
				$$X_1\oplus \dots \oplus X_k\oplus X_{k+1}\oplus \dots \oplus X_m \sim  X_1\oplus \dots \oplus \widetilde{X}_{\sigma(k)} \oplus \widetilde{X}_{\sigma(k+1)} \oplus \dots \oplus X_m;$$
	this uses the fact that the $H$-space addition on $G\calC$ is homotopy associative. Hence, concatenating all the homotopies produced by the adjacent transpositions, deduce the desired homotopy for $\sigma$. 
\end{proof}

\subsection{Admissible Triples}\label{sec:admtriples} This section proves Lemma~\ref{lem:admtriple}, which characterises the free homotopy class of admissible triples under mild assumptions. The result appears as Corollary~\ref{cor:admtrip}, following a key lemma proved below. As before, recall Convention~\ref{conv:K1} that $K_1(\calC) = \pi_1|G\calC^{\mathsf{o}}|$, where $G\calC^{\mathsf{o}}$ is the base-point component.

We begin by clarifying the group structure on $K_1(\calC)$ before turning to the main argument. The standard group action is given by concatenation of loops, but since $G\calC$ is an $H$-space, we can say more.

\begin{observation}\label{obs:K1DirectSum} $K_1(\calC)$ is generated by 1-simplices of $G\calC^{\mathsf{o}}$. The group action is equivalent to taking their direct sums. In particular, given any 1-simplex  
		$$	X:=	\begin{pmatrix}  &O\rtail & \doubleunderline{A \xrtail{\alpha} B}\\ &O  \rtail & A'\xrtail{\alpha'}B'\end{pmatrix},
	$$
its corresponding inverse is given by swapping the top and bottom rows 
	$$	Y:=	\begin{pmatrix}  &O\rtail & \doubleunderline{A' \xrtail{\alpha'} B'}\\ &O  \rtail & A\xrtail{\alpha}B\end{pmatrix}.$$
\end{observation}
\begin{proof} That $K_1(\calC)$ is generated by 1-simplices in $G\calC^{\mathsf{o}}$ is Remark~\ref{rem:piGC}. Next, by the usual Eckmann-Hilton argument, the standard group action of $\pi_1|G\calC^{\mathsf{o}}|$ is homotopic to the $H$-space action induced by the $G$-construction -- which, by Theorem~\ref{thm:Gconstruction}, corresponds to taking direct sums. It remains to verify the claim about inverses. Taking the sum of $X$ and $Y$ above, we get
	$$X + Y = \begin{pmatrix}  &O\rtail & \doubleunderline{A\oplus A'\xrtail{\alpha\oplus \alpha'}B\oplus B' }\\ &O  \rtail & A'\oplus A\xrtail{\alpha'\oplus \alpha}B'\oplus B \end{pmatrix}\sim \begin{pmatrix}  &O\rtail & \doubleunderline{A\oplus A'\xrtail{\alpha\oplus \alpha'}B\oplus B'}\\ &O  \rtail & A\oplus A'\xrtail{\alpha\oplus \alpha'}B\oplus B' \end{pmatrix},$$
with the obvious quotients. By Permutation Lemma~\ref{lem:permute}, we can permute the summands on the bottom row of $X+Y$ without changing its homotopy class; hence, $X+Y$ is homotopic to a diagonal square. Since the canonical loop of a diagonal square bounds a 2-simplex, conclude that $\lrangles{X+Y}=0$ in $K_1$, as desired.\footnote{ {\em Notes for the cautious reader.} The choice of maximal spanning tree $T$ in Remark~\ref{rem:piGC} determines a canonical loop for each 1-simplex $f$ in $G\calC^{\mathsf{o}}$, so $\lrangles{f}$ is well-defined in $K_1$. Explicitly, let us build $T$ by first including $\{(O,O)\to(A,A')\}$ with $A$ the chosen common quotient, before extending to a maximal tree by Zorn’s Lemma. In particular, diagonal squares still vanish in $K_1$ since their canonical loops bound a 2-simplex (see Theorem~\ref{thm:NenSurj}, Relation~(N1), proved independently of Lemma~\ref{lem:admtriple}).}\end{proof}

\begin{construction} Let $\calT$ be a triangle contour, presented as in Diagram~\eqref{eq:tricontour}. Given any vertex $(A,A')$ in $G\calC$, one can construct a new triangle contour $(A,A')\oplus \mathcal{T}$ by formal direct sum
	\begin{equation}
	\begin{tikzcd}
	&(P_1\oplus A,P'_1\oplus A') \ar[dr,"e_1\oplus (A\text{,}A')"]\\
	(P_0\oplus A,P'_0\oplus A') \ar[ur,"e_0\oplus (A\text{,}A')"] \ar[rr,"e_2\oplus (A\text{,}A')"] && (P_2\oplus A,P'_2\oplus A')
	\end{tikzcd} 
	\end{equation}
That $(A,A')\oplus \mathcal{T}$ is well-defined follows from Lemma~\ref{lem:ADTsquares}.
\end{construction}

\begin{lemma}[Key Lemma] Let $\calT=(e_0,e_1,e_2)$ be an admissible triple, where each $e_i$ is a double exact square. Following the notation of ~\eqref{eq:tricontour}, this assembles into a pair of flag diagrams, presented below.
	\begin{equation}\label{eq:admissibleDES}
	\begin{tikzcd}
	P_0 \ar[r, >->,"\alpha_{0,1}"] & P_1 \ar[dr,phantom,"\square"] \ar[r, >->,"\alpha_{1,2}"] & P_2\\
	&	P_{1/0} \ar[r, >->,"\alpha_{1/0,2/0}"] \ar[u, {Circle[open]}->,"\alpha_{1/0,1}"]& P_{2/0} \ar[u, {Circle[open]}->,swap,"\alpha_{2/0,2}"] \\
	&& P_{2/1} \ar[u, {Circle[open]}->,swap,"\alpha_{2/1,2/0}"] 
	\end{tikzcd}\quad \begin{tikzcd}
	P_0 \ar[r, >->,"\alpha'_{0,1}"] & P_1 \ar[dr,phantom,"\square"] \ar[r, >->,"\alpha'_{1,2}"] & P_2\\
	&	P_{1/0} \ar[r, >->,"\alpha'_{1/0,2/0}"] \ar[u, {Circle[open]}->,"\alpha'_{1/0,1}"]& P_{2/0} \ar[u, {Circle[open]}->,swap,"\alpha'_{2/0,2}"] \\
	&& P_{2/1} \ar[u, {Circle[open]}->,swap,"\alpha'_{2/1,2/0}"] 
	\end{tikzcd}
	\end{equation}		
Let $l(\calT)$ be the double exact square associated to $\calT$, and $\mu(l(\calT))$ be its canonical loop.
	\underline{Then} the loop $\calT=e_0e_1e_2^{-1}$ is freely homotopic to 
	$(P_2,P_2)\oplus \mu(l(\calT)).$
\end{lemma}

\begin{proof}\hfill
	\subsubsection*{Step 0: Setup}
	 Start by taking repeated restricted pushouts
	
	\begin{equation}\label{eq:admtripRP}
	\!\!\!\!\!\!\!\!\!\!\!\!\!\!\!\!\!\!\!\!\!\!\!\!\!\!\!\!\!\!\!\!\!\!\!\!\!\!\!\!\small{\begin{tikzcd} & O \ar[r,>->] \ar[d,>->]& P_{1/0} \ar[d,>->] \ar[r,>->,"\alpha_{1/0,2/0}"] & P_{2/0} \ar[d,>->]\\
		P_0 \ar[d,>->,swap,"\alpha_{0,1}"] \ar[r,>->,"\alpha_{0,2}"] & P_2\ar[r,>->,"1\oplus P_{1/0}"] \ar[d,>->]& P_2\oplus P_{1/0} \ar[d,>->] \ar[r,>->,"1\oplus\alpha_{1/0,2/0}",yshift=3.5]&  P_{2}\oplus P_{2/0} \ar[d,>->]\\
		P_1  \ar[r,>->] \ar[d,>->,swap,"\alpha_{1,2}"] & Q \ar[r,>->] \ar[d,>->] & Z_0 \ar[d,>->] \ar[r,>->]& Z_1 \ar[d,>->]\\
		P_2 \ar[r,>->] & W_0\ar[r,>->]  & W_1 \ar[r,>->] & W_2\\
		\end{tikzcd}} \,\,\quad \small{\begin{tikzcd}& O \ar[r,>->] \ar[d,>->]& P_{1/0} \ar[d,>->] \ar[r,>->,"\alpha'_{1/0,2/0}"] & P_{2/0} \ar[d,>->] \\
		P_0 \ar[d,>->,swap,"\alpha'_{0,1}"] \ar[r,>->,"\alpha'_{0,2}"] & P_2\ar[r,>->,"1\oplus P_{1/0}"] \ar[d,>->]& P_2\oplus P_{1/0} \ar[d,>->] \ar[r,>->,"1\oplus\alpha'_{1/0,2/0}",yshift=3.5]&  P_{2}\oplus P_{2/0} \ar[d,>->]\\
		P_1  \ar[r,>->] \ar[d,>->,swap,"\alpha'_{1,2}"] & Q' \ar[r,>->] \ar[d,>->] & Z'_0 \ar[d,>->] \ar[r,>->]& Z'_1 \ar[d,>->]\\
		P_2 \ar[r,>->] & W'_0\ar[r,>->]  & W'_1 \ar[r,>->] & W'_2\\
		\end{tikzcd}},
	\end{equation}
	 For readability, we have renamed the restricted pushouts so that, e.g. $Q:= P_2\star_{P_0} P_1$ is the restricted pushout induced by $P_1\xltail{\alpha_{0,1}}P_0 \xrtail{\alpha_{0,2}} P_2$. Notice: all pairs of $\M$-morphisms in Diagram~\eqref{eq:admtripRP} form 1-simplices. 
	This fact, combined with Corollary~\ref{cor:RPtrick}, gives a roadmap for how to connect the vertices of the two loops.
	
\subsubsection*{Step 1: Reformulation} The restricted pushouts in Step 0 assemble into the diagram of 1-simplices below.
	\begin{equation*}
	\!\!\!\!\!\!\!\!\!\!\!\!\!\!\!\!\!\!\!\small{\begin{tikzcd}
		&& \ones{W_2}{W'_2}  \\ 
		\ones{P_2\oplus P_{1/0}}{P_2\oplus P_{1/0}}\ar[rr,blue,"\gamma_1"]  \ar[urr,"a"{name=Y6},violet] && 	\ones{P_2\oplus P_{2/0}}{P_2\oplus P_{2/0}} \ar[u,"b"{name=Y7},violet] \ar[from=Y6, to=Y7,phantom,"(1)", yshift=-5]\\
		& 	\ones{P_2}{P_2} \ar[ur,blue,"\gamma_2",""{name=E2}] \ar[ul,"\gamma_0",swap,blue,""{name=E1}] \end{tikzcd}}\qquad \begin{tikzcd}
	&&&& \ones{W_2}{W_2'}   \\
	\ones{Q}{Q'}  \ar[rrrru,"e"{name=U1},violet] \ar[rrrr, "h"{name=Y0}]&&&& \ones{W_0}{W_0'} \ar[u, violet,"f"{name=U0}] \ar[from=U0,to=U1,phantom,"(1')",yshift=-10] \\
	&\ones{P_1}{P_1} \ar[rr, "e_1"{name=Y0},red] \ar[ul,""{name=Y1}] \ar[urrr,bend left=5,""{name=Y3}]&& \ones{P_2}{P_2} \ar[ur,""{name=Y4}] \ar[from=Y3, to=Y1,phantom,"(2')", xshift=-15]\ar[from=Y3, to=Y4,phantom,"(3')",,yshift=-5]
	\\
	&& \ones{P_0}{P_0} \ar[uull, bend left=35,""{name=Z3}] \ar[uurr, bend right=35,""{name=Z2}]\ar[d,"e_2"] \ar[ul,red,swap,"e_0"{name=Z6}] \ar[ur,red,"e_2"{name=Z7}]
	\ar[from=Z6,to=Z3,phantom,"(4')"]\ar[from=Z7,to=Z2,phantom,"(5')"]
	\\
	&& \ones{P_2}{P_2} \ar[uuurr, bend right=30,violet,swap, "g",""{name=Z0}] 
	\ar[from=Z2, to=Z0, phantom, "(7')",xshift=-50,yshift=-30]  
	\ar[uuull, bend left=30,violet,"d"{name=Z1}]
	\ar[from=Z1, to=Z3, phantom, "(6')",xshift=50,yshift=-30]
	\end{tikzcd}
	\end{equation*}
	
\noindent Here, $(P_2,P_2)\oplus \mu(l(\calT))$ is presented as the blue edges $\gamma_0 \gamma_1 \gamma_2^{-1}$ in the LHS diagram above, while the loop $\calT=e_0e_1e_2^{-1}$  as the red edges in the RHS diagram. Some preliminary observations:
\begin{enumerate}[label=(\alph*)]
	\item On the RHS diagram: the loop $dhg^{-1}$ is obtained by taking restricted pushouts of 
	$$\ones{P_2}{P_2}\leftarrow \ones{P_0}{P_0}\to \ones{P_1}{P_1}\qquad \ones{P_2}{P_2}\leftarrow \ones{P_0}{P_0}\to \ones{P_2}{P_2}\quad \text{and}\quad \ones{Q}{Q'}\leftarrow \ones{P_1}{P_1}\to \ones{P_2}{P_2}.$$
	By Corollary~\ref{cor:RPtrick}, Triangles (2') - (7') are all 2-simplices. It is also easy to see that they assemble into the diagram above. The only subtlety is verifying Triangles (3') and (5') share the same face $(P_2,P'_2)\to (W_0,W'_0)$ -- but this follows from our hypothesis that $\calT$ is an {\em admissible} triple. Conclude that $\calT$ is freely homotopic to the loop $dhg^{-1}$.
	\item Triangles (1) and (1') are obtained by taking restricted pushouts of 
	$$\ones{W_1}{W'_1}\leftarrow \ones{P_2\oplus P_{1/0}}{P_2\oplus P_{1/0}}\to \ones{P_2\oplus P_{2/0}}{P_2\oplus P_{2/0}}\qquad \text{and}\quad \ones{Z_1}{Z_1'}\leftarrow \ones{Q}{Q'}\to \ones{W_0}{W'_0}$$
	respectively, before applying Corollary~\ref{cor:RPtrick}. Conclude that $\calT$ is freely homotopic to the outer loop $def^{-1}g^{-1}$, and $(P_2,P'_2)\oplus \mu(l(\calT))$ is freely homotopic to $\gamma_0ab^{-1}\gamma_2^{-1}$.
	\item The concatenation of 1-simplices $\gamma_2 b$ is homotopic to $gf$ by applying Corollary~\ref{cor:RPtrick} to the restricted pushout
	\[\begin{tikzcd}
	\ones{P_2}{P_2} \ar[rr,"g"] \ar[d,swap,"\gamma_2"] \ar[drr] && \ones{W_0}{W_0'} \ar[d,"f"]\\
	\ones{P_2\oplus P_{2/0}}{P_{2}\oplus P_{2/0}} \ar[rr,swap,"b"]&& \ones{W_2}{W_2'}
	\end{tikzcd}.\]
	\item It is worth recalling Warning~\ref{warning:compose}: given two composable 1-simplices 
	$$z_0\colon (A,A')\to (B,B') \qquad\text{and}\qquad z_1\colon (B,B')\to (C,C')$$
with sequences $A\rtail B\rtail C$ and $A'\rtail B'\rtail C'$ in $\M$, their composite 
$$\overline{z_0z_1} \colon (A,A')\to (C,C')$$
need not satisfy
$$\lrangles{z_0}+\lrangles{z_1}=\lrangles{\overline{z_0z_1}}\qquad\text{in} \,\, K_1(\calC),$$
{\em unless} the three 1-simplices are related by a 2-simplex. In the case of Item (c), such a 2-simplex exists, which yields
\begin{equation}\label{eq:compSPLIT}
\lrangles{g}+\lrangles{f}=\lrangles{\overline{gf}} = \lrangles{\overline{\gamma_2 b}} =\lrangles{\gamma_2}+\lrangles{b}.
\end{equation}
\end{enumerate}

In sum: items (a) and (b) establishes the free homotopies 
$$\calT\sim def^{-1}g^{-1} \qquad\text{and}\qquad  (P_2,P'_2)\oplus \mu(l(\calT)) \sim \gamma_0ab^{-1}\gamma_2^{-1}.$$
Hence, to prove the lemma, it suffices to show $def^{-1}g^{-1}$ and $\gamma_0ab^{-1}\gamma_2^{-1}$ are freely homotopic to each other. This is equivalent\footnote{The setup here is more general than in Lemma~\ref{lem:Nenclosedloop}: a canonical loop for a 1-simplex may involve a zig-zag back to $(O,O)$. Nevertheless, the same cancellation argument applies, so the decomposition of loops into formal sums of 1-simplices still holds.} to showing 
$$ \lrangles{\gamma_0} + \lrangles{a} - \lrangles{b} -\lrangles{\gamma_2} = \lrangles{d} + \lrangles{e} -\lrangles{f} -\lrangles{g} \qquad \text{in $K_1$}.$$
Since $\lrangles{g}+\lrangles{f}=\lrangles{\gamma_2} +\lrangles{b}$ by item (d), this reduces to showing 
\begin{equation}\label{eq:step2ID}
\lrangles{\gamma_0}+\lrangles{a}-\lrangles{d}-\lrangles{e}=0 .
\end{equation}

\subsubsection*{Step 2: A New 2-Simplex} Recall from Observation~\ref{obs:K1DirectSum} that addition in $K_1$ corresponds to taking direct sums of 1-simplices, and subtraction to swapping the top and bottom rows. Hence, with Step 1 in mind, we construct a 2-simplex whose boundary corresponds to the concatenation of 1-simplices $\lrangles{\gamma_0}-\lrangles{d}$ and $\lrangles{a}-\lrangles{e}$ (up to permutation of summands). 

Explicitly, this 2-simplex is presented by
$$	\begin{pmatrix}  &O\rtail & \doubleunderline{\,P_2\oplus P_2 \rtail  P_2\oplus P_{1/0}\oplus Q'\rtail W_2\oplus W'_2\, }\\ &O  \rtail & P_2\oplus P_2 \rtail  Q\oplus P_2\oplus P_{1/0}\rtail W_2\oplus W'_2\,\end{pmatrix},
$$
which corresponds to the flag diagrams
\begin{equation*}
\!\!\!\!\!\!\!\!\!\!\!\!\!\!\!\!\!\!\!\!\!\!\!\!\!\!\!\!\!\small{\begin{tikzcd}
P_2\oplus P_2\ar[r, >->] & P_2\oplus P_{1/0}\oplus Q'  \ar[dr,phantom,"\square"] \ar[r, >->] & W_2\oplus W'_2\\
&	P_{1/0}\oplus P_{1/0} \ar[r, >->] \ar[u, {Circle[open]}->]& P_{2/0}\oplus P_{2/0}\oplus P_{2/0}\oplus P_{2/0}\ar[u, {Circle[open]}->] \\
&& P_{2/0}\oplus P_{2/1}\oplus P_{2/0}\oplus P_{2/1}\ar[u, {Circle[open]}->] 
\end{tikzcd}\!\!\begin{tikzcd}
P_2\oplus P_2\ar[r, >->] & Q\oplus P_2\oplus P_{1/0} \ar[dr,phantom,"\square"] \ar[r, >->] & W_2\oplus W'_2\\
&	P_{1/0}\oplus P_{1/0} \ar[r, >->] \ar[u, {Circle[open]}->]& P_{2/0}\oplus P_{2/0}\oplus P_{2/0}\oplus P_{2/0}\ar[u, {Circle[open]}->] \\
&& P_{2/0}\oplus P_{2/1}\oplus P_{2/0}\oplus P_{2/1}\ar[u, {Circle[open]}->] 
\end{tikzcd}}
\end{equation*}

Step 2 constructs this 2-simplex in stages. We continue with the $\M$-morphisms in Step 0, now explicitly labelled as needed. 

\subsubsection*{Step 2a} By Axiom (PQ), we have the following exact squares
\begin{equation}\label{eq:Step3a}
\dsquaref{O}{ P_{1/0}}{P_2}{Q}{}{}{q_1}{\beta_{1/0}} \qquad\text{and}\qquad  \dsquaref{O}{ P_{1/0}}{P_2}{Q'}{}{}{q'_1}{\beta'_{1/0}}.
\end{equation}
Applying Lemma~\ref{lem:DirectSum}, construct the following pair of distinguished squares
\begin{equation}\label{eq:3a}
\begin{tikzcd}
O \ar[rr,>->] \ar[d,{Circle[open]}->] \ar[drr,phantom,"\square"]&& P_{1/0}\oplus P_{1/0}  \ar[d,{Circle[open]}->,"t_0"]\\
P_2\oplus P_2 \ar[rr,>->,"s_0"]&& P_2\oplus P_{1/0}\oplus Q'
\end{tikzcd} \qquad 
\begin{tikzcd}
O \ar[rr,>->] \ar[d,{Circle[open]}->]\ar[drr,phantom,"\square"]&& P_{1/0}\oplus P_{1/0}  \ar[d,{Circle[open]}->," t_1"]\\
P_2\oplus P_2 \ar[rr,>->,"s_1"]&& Q\oplus P_2\oplus P_{1/0}
\end{tikzcd}
\end{equation}
\begin{equation*}
s_0=\begin{pmatrix}
1 & 0\\
0 & 0\\
0 & q'_1
\end{pmatrix},\,\, s_1=\begin{pmatrix}
q_1 & 0\\
1 & 0\\
0 & 0
\end{pmatrix},\,\, t_0=\begin{pmatrix}
0 & 0\\
1 & 0\\
0 & \beta'_{1/0}
\end{pmatrix},\,\, t_1=\begin{pmatrix}
\beta_{1/0} & 0\\
0 & 0\\
0 & 1
\end{pmatrix}.
\end{equation*}
\subsubsection*{Step 2b} In this stage, we show that the quotients of the restricted pushouts in Step 0 can be chosen to align with those of the admissible triple $\calT$. The following claim makes this precise.

\begin{claim}\label{claim:step-2b} Fix the morphisms of the given admissible triple $\calT=(e_0,e_1,e_2)$, as well as the $\M$-morphisms in Step 0. \underline{Then}, one can construct a pair of flag diagrams
		\begin{equation}\label{eq:3bWARMUP}
	\begin{tikzcd}
	P_0 \ar[r, >->,"\alpha_{0,1}"] & P_1 \ar[dr,phantom,"\square"] \ar[r, >->,"\alpha_{1/2}"] & P_2\\
	&	P_{1/0} \ar[r, >->,"\alpha_{1/0,2/0}"] \ar[u, {Circle[open]}->,"\alpha_{1/0,1}"]& P_{2/0} \ar[u, {Circle[open]}->,swap,"\alpha_{2/0,2}"] \\
	&& P_{2/1} \ar[u, {Circle[open]}->,swap,"\alpha_{2/1,2/0}"] 
	\end{tikzcd}\quad \begin{tikzcd}
	P_2 \ar[r, >->,"q_1"] & Q \ar[dr,phantom,"\square"] \ar[r, >->,"w_0"] & W_0\\
	&	P_{1/0} \ar[r, >->,"\beta_{1/0,2/0}"] \ar[u, {Circle[open]}->,"\beta_{1/0}"]& P_{2/0} \ar[u, {Circle[open]}->,swap,"\beta_{2/0}"] \\
	&& P_{2/1} \ar[u, {Circle[open]}->,swap,"\beta_{2/1,2/0}"] 
	\end{tikzcd}
	\end{equation}
satisfying the identities
	$$\alpha_{1/0,2/0}=\beta_{1/0,2/0} \qquad \text{and} \qquad \alpha_{2/1,2/0}=\beta_{2/1,2/0}\;.$$
\end{claim}
\begin{proof}[Proof of Claim] Apply Lemma~\ref{lem:quotFilt} to the filtrations
\begin{equation}\label{eq:2b-filtration}
P_0\rtail P_1\rtail P_2 \qquad\text{and}\qquad P_{2}\rtail Q\rtail W_0.
\end{equation}
We now make the corresponding choices of quotients explicit. By Axiom (PQ) and Convention~\ref{conv:restrPO-quotient}, choose quotients of $W_0$ such that the following diagrams commute in the ambient pCGW category $\calC$. 
\begin{equation}\label{eq:3bPOproperty}
\begin{tikzcd}
P_1 \ar[r,>->,"\alpha_{1,2}"] \ar[d,>->,swap,"q_1"] & P_2 \ar[dr,phantom,"\circlearrowleft"]\ar[d,>->,swap,"q_2"] & P_{2/1} \ar[l,swap,"\alpha_{2/1,2}",{Circle[open]}->] \ar[d,>->,"="]\\
Q \ar[r,>->,swap,"w_0"]& W_0 & P_{2/1} \ar[l,swap,swap,"\beta_{2/1}",{Circle[open]}->]
\end{tikzcd} \qquad\text{and}\qquad  \begin{tikzcd}
P_0\ar[r,>->,"\alpha_{0,2}"] \ar[d,>->,swap,"\alpha_{0,2}"] & P_2 \ar[dr,phantom,"\circlearrowleft"]\ar[d,>->,swap,"q_2"] & P_{2/0} \ar[l,swap,"\alpha_{2/0,2}",{Circle[open]}->] \ar[d,>->,"="]\\
P_2 \ar[r,>->,swap,"q_2"]& W_0 & P_{2/0} \ar[l,swap,swap,"\beta_{2/0}",{Circle[open]}->]
\end{tikzcd} \quad .
\end{equation}
Hence, let the morphisms $\alpha_{1/0,1}$, $\alpha_{2/0,2}$, $\beta_{1/0}$, and $\beta_{2/0}$ be the obvious quotients associated to~\eqref{eq:2b-filtration}. Finally, in the notation of Diagram~\eqref{eq:3bWARMUP}, choose the quotients of $\M$-morphisms $\alpha_{1/0,2/0},\beta_{1/0,2/0}\colon P_{1/0}\rtail P_{2/0}$ so that the following $\E$-squares commute:
\begin{equation}\label{eq:3bsquare}
\begin{tikzcd}
P_{2/1} \ar[d,{Circle[open]}->,swap,"\alpha_{2/1,2/0}"]\ar[r,{Circle[open]}->,"="] \ar[dr,phantom,"\circlearrowleft"]& P_{2/1} \ar[d,{Circle[open]}->,"\alpha_{2/1,2}"]\\
P_{2/0}\ar[r,{Circle[open]}->,swap,"\alpha_{2/0,2}"]& P_2
\end{tikzcd} \qquad \begin{tikzcd}
P_{2/1} \ar[d,{Circle[open]}->,swap,"\beta_{2/1,2/0}"]\ar[r,{Circle[open]}->,"="] \ar[dr,phantom,"\circlearrowleft"]& P_{2/1} \ar[d,{Circle[open]}->,"\beta_{2/1}"]\\
P_{2/0} \ar[r,{Circle[open]}->,swap,"\beta_{2/0}"]& W_0
\end{tikzcd}\,\,.
\end{equation}

Our claim then follows by translating Diagrams~\eqref{eq:3bPOproperty} and \eqref{eq:3bsquare} into commutative diagrams in $\calC$, and performing a diagram-chase. There are two cases to check.
	\begin{itemize}
		\item \textbf{Case 1:} $\E\subseteq \calC$. In which case, we get the identity
		$$\beta_{2/0}\circ \beta_{2/1,2/0}=\beta_{2/1}=q_2\circ \alpha_{2/1,2}.$$
		Since 
		$$\alpha_{2/1,2}=\alpha_{2/0,2}\circ \alpha_{2/1,2/0},$$
		this yields 
		$$  \beta_{2/0}\circ \beta_{2/1,2/0}=q_2\circ \alpha_{2/0,2}\circ \alpha_{2/1,2/0}.$$
		In particular, notice $\beta_{2/0}=q_2\circ \alpha_{2/0,2}$
		by Diagram~\eqref{eq:3bPOproperty}. Hence, since $\E$-morphisms are monic in $\calC$ by Axiom (M), deduce that  
		$$  \beta_{2/1,2/0}=\alpha_{2/1,2/0}.$$
In other words, $\alpha_{1/0,2/0}$ and $\beta_{1/0,2/0}$ are formal kernels of the same $\E$-morphism, and are thus equivalent up to codomain-preserving isomorphism by Axiom (K). Hence, without loss of generality,\footnote{Explicitly, one may insert an isomorphism $\kappa\colon P_{1/0}\to P_{1/0}$ into the relevant distinguished squares so that $\alpha_{1/0,2/0}$ and $\beta_{1/0,2/0}$ coincide. This corresponds to readjusting the chosen quotient for $q_1\colon P_2\rtail Q$ in Step 2a by an isomorphic representative; the reader can check this does not affect our ability to construct the 2-simplex.} we may assume
	$$\alpha_{1/0,2/0}=\beta_{1/0,2/0}.$$

		\item \textbf{Case 2:} $\E^\opp\subseteq \calC$. Dual to Case 1.
	\end{itemize}
\end{proof}

\subsubsection*{Step 2c} Leveraging Step 2b, we construct two pairs of distinguished squares by a similar procedure.
\begin{itemize}
	\item \textbf{First Pair.} Apply Axiom (DS) to the restricted pushouts of the spans 
	$$P_2\oplus P_{2/0}\ltail P_2\rtail W_0 \quad \text{and}\quad P_2\oplus P_{2/0}\ltail P_2\rtail W'_0, $$
and obtain distinguished squares
	\[\dsquaref{P_{2/0}}{P_{2/0}\oplus P_{2/0}}{W_0}{W_2}{}{\beta_{2/0}}{}{v_0}\qquad \dsquaref{P_{2/0}}{P_{2/0}\oplus P_{2/0}}{W'_0}{W'_2}{}{\beta'_{2/0}}{}{v_1}.\]
Composing with the distinguished squares constructed by Claim~\ref{claim:step-2b}, we obtain 
	\begin{equation}\label{eq:3c1}
\dsquaref{P_{1/0}}{P_{2/0}\oplus P_{2/0}}{Q}{W_2}{u_0}{\beta_{1/0} }{ }{ v_0} \qquad \dsquaref{P_{1/0}}{P_{2/0}\oplus P_{2/0}}{Q'}{W'_2}{u_1}{ \beta'_{1/0}}{}{ v_1 } 
\end{equation}
where the top maps are given by
\[
u_0=\begin{pmatrix}\alpha_{1/0,2/0}\\[4pt]0\end{pmatrix},\qquad
u_1=\begin{pmatrix}\alpha'_{1/0,2/0}\\[4pt]0\end{pmatrix}.
\]
We omit an explicit description of the $\E$-morphisms \(v_0,v_1\), since they play no further role in the argument.

\item \textbf{Second Pair.} Construct the following diagram pair
\[\begin{tikzcd}
P_{1/0}\oplus P_{2/0} \ar[dr,phantom,"\square",blue]\ar[d,{Circle[open]->}] &P_{1/0} \ar[l,>->,"1\oplus P_{2/0}",swap ] \ar[d,{Circle[open]->}] \ar[dr,phantom,"\square",red]\ar[r,>->,"\alpha_{1/0,2/0}"] & P_{2/0} \ar[d,{Circle[open]->}] \\
W_1 & P_{2}\oplus P_{1/0}\ar[r,>->]\ar[l,>->] & P_2\oplus P_{2/0} 
\end{tikzcd}\qquad \begin{tikzcd}
P_{1/0}\oplus P_{2/0} \ar[dr,phantom,"\square",blue]\ar[d,{Circle[open]->}] &P_{1/0} \ar[l,>->,"1\oplus P_{2/0}",swap] \ar[d,{Circle[open]->}] \ar[dr,phantom,"\square",red]\ar[r,>->,"\alpha'_{1/0,2/0}"] & P_{2/0} \ar[d,{Circle[open]->}] \\
W'_1 & P_{2}\oplus P_{1/0}\ar[r,>->]\ar[l,>->] & P_2\oplus P_{2/0} 
\end{tikzcd}\,\,;\]
the red squares are constructed by Lemma~\ref{lem:DirectSum} (v), the blue squares by applying Axiom (DS) to the restricted pushouts of the spans
$$P_{2}\oplus P_{1/0}\ltail P_{2}\rtail W_0 \qquad\text{and}\qquad P_{2}\oplus P_{1/0}\ltail P_{2}\rtail W'_0.$$
Applying Axiom (DS) once more to the above diagram pair, we obtain
\begin{equation}\label{eq:3c2}
\dsquaref{P_{1/0}}{P_{2/0}\oplus P_{2/0}}{P_{2}\oplus P_{1/0}}{W_2}{u_0}{P_2\oplus 1 }{ }{ w_0} \qquad \dsquaref{P_{1/0}}{P_{2/0}\oplus P_{2/0}}{P_2\oplus P_{1/0}}{W'_2}{u_1}{ P_2\oplus 1}{}{ w_1 }.
\end{equation}
\end{itemize}
We then play the same game as in Step 2a. Applying Lemma~\ref{lem:DirectSum} (v), combine Diagrams~\eqref{eq:3c1} and \eqref{eq:3c2} to construct 
\begin{equation}\label{eq:3cDS}
\!\!\!\!\!\!\!\!\!\dsquaref{P_{1/0}\oplus P_{1/0}}{P_{2/0}\oplus P_{2/0}\oplus P_{2/0}\oplus P_{2/0}}{P_{2}\oplus P_{1/0}\oplus Q'}{W_2\oplus W'_2}{u_3}{t_0}{ }{z_0 } \, \dsquaref{P_{1/0}\oplus P_{1/0}}{P_{2/0}\oplus P_{2/0}\oplus P_{2/0}\oplus P_{2/0}}{Q\oplus P_{2}\oplus P_{1/0}}{W_2\oplus W'_2}{u_3}{t_1 }{ }{z_1 }
\end{equation}
whereby
\begin{equation*}
u_3=u_0\oplus u_1, \qquad z_0=w_0\oplus v_1, \qquad z_1=v_0\oplus w_1,
\end{equation*}
and $t_0,t_1$ are defined as in Step 2a.
\subsubsection*{Step 2d} We now construct our 2-simplex. Start by horizontally composing Diagrams~\eqref{eq:3a} and ~\eqref{eq:3cDS}. Then, vertically compose Diagram~\eqref{eq:3cDS} with the following (identical) pair of exact squares, obtained by applying Axiom (K):
\begin{equation}
\small{\dsquaref{O}{P_{2/0}\oplus P_{2/1}\oplus P_{2/0}\oplus P_{2/1}}{P_{1/0}\oplus P_{1/0}}{P_{2/0}\oplus P_{2/0}\oplus P_{2/0}\oplus P_{2/0}}{}{ }{u_3}{} }\quad \small{\dsquaref{O}{P_{2/0}\oplus P_{2/1}\oplus P_{2/0}\oplus P_{2/1}}{P_{1/0}\oplus P_{1/0}}{P_{2/0}\oplus P_{2/0}\oplus P_{2/0}\oplus P_{2/0}}{}{ }{u_3}{}}\quad.
\end{equation}
And we are done.

\subsubsection*{Step 3: Finish} We now prove $(P_2,P_2)\oplus \mu(l(\calT))$ is freely homotopic to $\calT$. Summarising our work above:
\begin{itemize}
	\item \textbf{Step 1:} $\calT\sim def^{-1}g^{-1}$.
	\item \textbf{Step 1:} $(P_2,P_2)\oplus \mu(l(\calT)) \sim \gamma_0ab^{-1}\gamma_2^{-1}.$
	\item \textbf{Step 1:} $gf\sim \gamma_2 b$. In fact, by Equation~\eqref{eq:compSPLIT}, their composites split:
	$$\lrangles{g}+\lrangles{f} = \lrangles{\overline{gf}}=\lrangles{\overline{\gamma_2b}}=\lrangles{\gamma_2}+\lrangles{b}.$$
	\item \textbf{Step 2:} We have the following 2-simplex  
$$	\begin{pmatrix}  &O\rtail & \doubleunderline{\,P_2\oplus P_2 \rtail  P_2\oplus P_{1/0}\oplus Q'\rtail W_2\oplus W'_2\, }\\ &O  \rtail & P_2\oplus P_2 \rtail  Q\oplus P_2\oplus P_{1/0}\rtail W_2\oplus W'_2\,\end{pmatrix}.
$$
\end{itemize}
By Step 1, our problem reduces to showing that the concatenation $de$ is homotopic to $\gamma_0 a$. Presented as generators of $\pi_1|G\calC ^{\mathsf{o}}|$, this amounts to showing
$$\lrangles{\gamma_0}+\lrangles{a}-\lrangles{d} -\lrangles{e}=0.$$
By Step 2, the 2-simplex yields the equation\footnote{This step implicitly applies the extension of Proposition~\ref{prop:baseline} noted in Remark~\ref{rem:piGC}, allowing all 1-simplices in $G\calC^{\mathsf{o}}$ rather than only double exact squares.}
\begin{align}\label{eq:step4}
\!\!\!\!\!\!\!\!\!\!\left\langle 	\begin{pmatrix}  &O\rtail & \doubleunderline{\,P_2\oplus P_2 \rtail  P_2\oplus P_{1/0}\oplus Q' \, }\\ &O  \rtail & P_2\oplus P_2 \rtail  Q\oplus P_2\oplus P_{1/0} \,\end{pmatrix} \right\rangle &+ \left\langle 	\begin{pmatrix}  &O\rtail & \doubleunderline{\,  P_2\oplus P_{1/0}\oplus Q' \rtail W_2\oplus W'_2 \, }\\ &O  \rtail & Q\oplus P_2\oplus P_{1/0} \rtail W_2\oplus W'_2\,\end{pmatrix} \right\rangle \nonumber  \\ 
&= \left\langle 	\begin{pmatrix}  &O\rtail & \doubleunderline{\,P_2\oplus P_2 \rtail  W_2\oplus W'_2 \, }\\ &O  \rtail & P_2\oplus P_2 \rtail  W_2\oplus W'_2 \,\end{pmatrix} \right\rangle . 
\end{align}

Let us start by analysing the first term. Compute
$$\!\!\!\!\!\! \!\left\langle 	\begin{pmatrix}  &O\rtail & \doubleunderline{\, {\color{purple}P_2}\oplus {\color{blue} P_2} \rtail  {\color{purple} P_2\oplus P_{1/0}}\oplus {\color{blue} Q'} \, }\\ &O  \rtail & {\color{blue} P_2} \oplus {\color{purple}  P_2} \rtail {\color{blue} Q}\oplus  {\color{purple} P_2\oplus P_{1/0}} \,\end{pmatrix} \right\rangle  = \left\langle 	\begin{pmatrix}  &O\rtail & \doubleunderline{\, {\color{purple}P_2}\oplus {\color{blue} P_2} \rtail  {\color{purple} P_2\oplus P_{1/0}}\oplus {\color{blue} Q'} \, }\\ &O  \rtail & {\color{purple} P_2} \oplus {\color{blue}  P_2} \rtail {\color{purple} P_2\oplus P_{1/0}} \oplus  {\color{blue} Q} \,\end{pmatrix} \right\rangle   = \lrangles{\gamma_0} - \lrangles{d}; $$
the first equality applies Permutation Lemma~\ref{lem:permute}, the second equality applies Observation~\ref{obs:K1DirectSum}. The same permutation argument can be applied to the other terms of Equation~\eqref{eq:step4}. This assembles to yield
$$\lrangles{\gamma_0} - \lrangles{d} + \lrangles{a} - \lrangles{e} = \lrangles{\overline{\gamma_2b}} -\lrangles{\overline{gf}} = 0,$$
where the final equality is justified by Step 1, Equation~\eqref{eq:compSPLIT}, completing the proof.
\end{proof}

\begin{corollary}\label{cor:admtrip} If an admissible triple $\calT$ features only double exact squares, then its loop is freely homotopic to $\mu(l(\calT))$. 
\end{corollary}
\begin{proof}
	
	Consider the operation
	\begin{equation}
	(A,A')\oplus\argu \colon G\calC\to G\calC,
	\end{equation}
	which adds a vertex $(A,A')$ to all the nodes of an $n$-simplex of $G\calC$ (in the sense of Lemma~\ref{lem:DirectSum}). Given any edge $(A,A')\to (B,B')$, this induces a simplicial homotopy 
	between the maps 
	\begin{equation}
	(A,A')\oplus\argu \longrightarrow (B,B')\oplus\argu.
	\end{equation} 
	In particular, if the loop $\calT$ features only double exact squares, then there exists an edge $(O,O)\to (P_2,P_2)$. By the above, this implies $(P_2,P_2)\oplus \mu(l(\calT))$ is homotopic to $\mu(l(\calT))$. 
\end{proof}

\bibliography{biblio}

\end{document}